\documentclass[aos]{imsart}

	\usepackage{xr}
\RequirePackage[OT1]{fontenc}
\RequirePackage{amsthm,amsmath}
\RequirePackage[numbers]{natbib}
\RequirePackage[colorlinks,linkcolor=blue,citecolor=blue,urlcolor=blue]{hyperref}
\usepackage{lipsum}
\usepackage{amssymb}
\usepackage{hyperref}
\usepackage{extarrows}
\usepackage{bm}
\usepackage{bbm}
\usepackage{epstopdf}
\usepackage{booktabs}
\usepackage{multirow}

\usepackage[usenames, dvipsnames]{color}
\usepackage{float}
\usepackage{amsmath}
\usepackage{enumitem}
\usepackage{graphicx}
\usepackage{subcaption}
\usepackage{comment}
\usepackage{epstopdf}

\numberwithin{equation}{section}

\theoremstyle{plain} 
\newtheorem{thm}{Theorem}[section]
\newtheorem{lem}[thm]{Lemma}

\newtheorem{pro}[thm]{Proposition}
\newtheorem{assum}[thm]{Assumption}
\newtheorem{defn}[thm]{Definition}

\theoremstyle{remark}
\newtheorem{rem}[thm]{Remark}

\newcommand{\Tr}{\mathrm{Tr}}

\newcommand{\E}{{\mathbb E }}
\newcommand{\R}{{\mathbb R }}
\newcommand{\N}{{\mathbb N}}

\newcommand{\ii}{\mathrm{i}}

\newcommand{\bs}{\boldsymbol}

\renewcommand{\mathbf}[1]{\bs{#1}}
 
\begin{document}

\begin{frontmatter}

% "Title of the paper"
\title{Singular vector and singular subspace distribution for the matrix denoising model}
\runtitle{Singular vector and singular subspace distribution}

\begin{aug}
\author{\fnms{Zhigang} \snm{Bao}\thanksref{t1}\ead[label=e1]{mazgbao@ust.hk}},
\author{\fnms{Xiucai} \snm{Ding}\thanksref{t4}\ead[label=e4]{xiucai.ding@mail.toronto.edu}},
\author{\fnms{Ke} \snm{Wang}\thanksref{t2}\ead[label=e2]{kewang@ust.hk}}

\thankstext{t1}{Z.G. Bao was  partially supported by Hong Kong RGC grant ECS 26301517, GRF  16300618,  GRF 16301519, and NSFC 11871425.}
\thankstext{t4}{X.C. Ding  was partially supported by NSERC of Canada (Jeremy Quastel).}
\thankstext{t2}{Ke Wang was partially supported by Hong Kong RGC grant GRF 16301618 and GRF 16308219.}
\runauthor{Z.G. Bao, X.C.Ding, and K. Wang.}

\affiliation{Hong Kong University of Science and Technology \\ University of Toronto and Duke University \\ Hong Kong University of Science and Technology}

\address{Department of Mathematics,  \\ Hong Kong University of Science and Technology, \\ Hong Kong\\
\printead{e1}\\
\phantom{E-mail:\ }}

\address{Department of Statistical Sciences \\ University of Toronto \\Canada\\
\printead{e4}\\
\phantom{E-mail:\ }}

\address{Department of Mathematics,  \\ Hong Kong University of Science and Technology, \\ Hong Kong\\
\printead{e2}\\
\phantom{E-mail:\ }}

\end{aug}

\runauthor{Zhigang Bao, Xiucai Ding, Ke Wang}

\begin{abstract}
In this paper, we study the  matrix denoising model $Y=S+X$, where $S$ is a low rank deterministic signal matrix and $X$ is a random noise matrix, and both are $M\times n$. In the scenario that $M$ and $n$ are comparably large and the signals are supercritical, we study the  fluctuation of the outlier singular vectors of $Y$,  under fully general assumptions on the structure of $S$ and the distribution of $X$.  More specifically, we derive the limiting distribution of angles between the principal singular vectors of $Y$ and their deterministic counterparts, the  singular vectors of $S$.  Further, we also derive the distribution of the distance between the subspace spanned by the principal singular vectors of $Y$ and that spanned by the singular vectors of $S$.  It turns out that the limiting distributions  depend on the structure of  the singular vectors of $S$ and the distribution of $X$, and thus they are non-universal.  Statistical applications of our results to singular vector and singular subspace inferences are also discussed.
\end{abstract}



\end{frontmatter}

\section{Introduction}
Consider an $M \times n$ noisy matrix $Y$ modeled as
\begin{align}\label{defn_m1}
Y= S+X,
\end{align}
where $S$ is a low-rank deterministic matrix with fixed rank $r$ and $X$ is a real random noise matrix. We assume that $S$ admits the  singular value decomposition 
\begin{equation}\label{eq_ssvd}
S=U D V^*=\sum_{i=1}^r d_i \mathbf{u}_i\mathbf{v}_i^*,
\end{equation}
where $D=\mathrm{diag}(d_1,\ldots,d_r)$ consists of  the singular values of $S$ and we assume $d_1>\ldots>d_r>0$;  $U=(\mathbf{u}_1,\ldots,\mathbf{u}_r) \in \mathbb{R}^{ M \times r}$ and $V=(\mathbf{v}_1,\ldots,\mathbf{v}_r) \in \mathbb{R}^{n \times r}$ are the matrices consisting of the $\ell^2$-normalized left and right singular vectors. 
For the noise matrix $X=(x_{ij})$ in (\ref{defn_m1}), we assume
that the entries $x_{ij}$'s are i.i.d. real  random variables with
\begin{equation}\label{eq_meanva}
\mathbb{E} x_{ij}=0,  \quad \mathbb{E}|x_{ij}|^2=\frac{1}{n}.
\end{equation}
 For simplicity, we also assume the existence of all moments, i.e., for every integer $q \ge 3,$ there is some constant $C_q>0,$ such that
 \begin{equation}\label{eq_moment}
 \mathbb{E}|\sqrt{n} x_{ij}|^q \leq C_q<\infty.
 \end{equation}
 This condition can be  weakened to the existence of  some sufficiently high order moment. But we do not pursue this direction here.  We remark that although we are primarily interested in the real case, our method  also applies to the case when $X$ is a complex noise matrix. 
 
  In practice, $S$ is often called the {\it signal} matrix which contains the information of interest.  In  the high-dimensional setup,  when $M$ and $n$ are comparably large, we are interested in the inference of $S$ or its left and right singular spaces, which are the subspaces spanned by $\mathbf{u}_i$'s or $\mathbf{v}_i$'s, respectively. Such a problem arises in many scientific applications such as  
matrix denoising \cite{Ding,donoho2014}, multiple signal classification (MUSIC) \cite{HACHEM2013427,6095424} and multidimensional scaling \cite{FSZZ2018, ds2018}. We call the model in (\ref{defn_m1}) the matrix denoising model, which is also known as the signal-plus-noise model in the literature.   We refer to Section \ref{subsec:application} for more introduction on  the application aspects.
 
 We denote the  singular value decomposition of $Y$ by 
\begin{align}\label{eq_ysvd}
Y= \widehat{U} \Lambda \widehat{V}^*=\sum_{i=1}^{M \wedge n} \sqrt{\mu}_i  \widehat{\mathbf{u}}_i \widehat{\mathbf{v}}_i^*,
\end{align}
where $\mu_1\geq \cdots\geq \mu_{M\wedge n}$ are the squares of the non-trivial singular values, and $\widehat{\mathbf{u}}_i$'s and $\widehat{\mathbf{v}}_i$'s are the $\ell^2$-normalized sample singular vectors. 
Here  $\widehat{U}=(\widehat{\mathbf{u}}_1, \ldots, \widehat{\mathbf{u}}_{M} ) $  and $\widehat{V}=(\widehat{\mathbf{v}}_1, \ldots, \widehat{\mathbf{v}}_{n} ) $ and $\Lambda$ is $M\times n$ with singular values on its main diagonal. 

In this paper, we are interested in the distributions of the principal left and right singular vectors of $Y$ and the subspaces spanned by them.  On singular vectors, a natural quantity to look into is the projection of a sample principal singular vector onto its deterministic counterpart, i.e.,  $|\langle\widehat{\mathbf{u}}_i,\mathbf{u}_i\rangle|$ and  $ |\langle\widehat{\mathbf{v}}_i,\mathbf{v}_i\rangle|$, which characterizes  the deviation of an original signal from the noisy one. On singular spaces, the natural estimators for $U$ and $V$ are their noisy counterparts
\begin{align*}
\widehat{U}_r = (\widehat{\mathbf{u}}_1,\ldots,\widehat{\mathbf{u}}_r) \quad \text{and}\quad \widehat{V}_r=(\widehat{\mathbf{v}}_1,\ldots,\widehat{\mathbf{v}}_r),
\end{align*}
respectively, i.e., the matrices consisting of  the first $r$ left and right singular vectors of $Y$, respectively.  To measure the distance between $\widehat{U}_r$ and $U$, or $\widehat{V}_r$ and $V$, we consider the following matrix of the cosine principal angles between two subspaces (see \cite[Section 6.4.3]{GVL12} for instance):
\begin{align*}
\cos \Theta(\widehat{V}_r, V) = \text{diag}(
\sigma_1^{V},\ldots, \sigma_r^{V}),\qquad  \cos \Theta(\widehat{U}_r, U) = \text{diag}(
\sigma_1^{U},\ldots, \sigma_r^{U}),
\end{align*}
where $\sigma_i^V$'s and $\sigma_i^U$'s are the singular values of the matrices $\widehat{V}_r^* V$ and $\widehat{U}_r^* U$, respectively. 
Therefore, an appropriate measure of the distance between the subspaces is  $L:=\| \cos \Theta (U, \widehat{U}_r) \|_F^2$ for the left singular subspace or $R:=\| \cos \Theta (V, \widehat{V}_r) \|_F^2$ for the right singular subspace, where $\|\cdot\|_F^2$ stands for the Frobenius norm. Note that $L$ and $R$ can also be written as  
{ \begin{align}
&L:= \sum_{i,j=1}^r |\langle\widehat{\mathbf{u}}_i,\mathbf{u}_j\rangle|^2= \frac12 \Big(2r-\|\widehat{U}_r\widehat{U}_r^*-UU^*\|_F^2\Big), \label{eq_leftsubspace}\\
 &R:=\sum_{i,j=1}^r |\langle\widehat{\mathbf{v}}_i,\mathbf{v}_j\rangle|^2=\frac12 \Big(2r-\|\widehat{V}_r\widehat{V}_r^*-VV^*\|_F^2\Big). \label{18012401}
\end{align}
}

In this paper, we are interested in the following high-dimensional regime: for some small constant $\tau\in (0, 1)$ we have 
\begin{equation}\label{eq_ratioassumption}
M\equiv M(n),\qquad  \ y \equiv y_n:=\frac{M}{n}\to c\in [\tau, \tau^{-1}], \quad \text{as  } n\to \infty.
\end{equation}
Our main results are on the limiting distributions of individual $|\langle\widehat{\mathbf{v}}_i,\mathbf{v}_i\rangle|^2$ (resp. $ |\langle\widehat{\mathbf{u}}_i,\mathbf{u}_i\rangle|^2$)  and $R$ (resp. $L$) when the signal strength, $d_i$'s, are supercritical (c.f. Assumption \ref{assum_main}).  They are detailed in  Theorems \ref{thm_mainthm}, \ref{thm. right subspace}, after necessary notations are introduced.  In the rest of this section, we review some related literature from both theoretical and applied perspectives.

\subsection{On finite-rank deformation of random  matrices} \label{subsec:previous} From the theoretical perspective, 
our model in (\ref{defn_m1}) is in the category of the fixed-rank deformation of the random matrix models in the Random Matrix Theory, which also includes the deformed Wigner matrix and the spiked sample covariance matrix as  typical examples. There are a vast of work devoted to this topic and the primary interest is to investigate the limiting behavior of the extreme eigenvalues and the associated eigenvectors of the deformed models.  Since the seminal work of Baik, Ben Arous and P\'{e}ch\'{e} \cite{BBP}, it is now well-understood that the extreme eigenvalues undergo a so-called BBP transition along with the change of the strength of the deformation.  Roughly speaking, there is a critical value such that the extreme eigenvalue of the deformed matrix will stick to the right end point of the limiting spectral distribution of the undeformed random matrix if the strength of the deformation is less than or equal to the critical value, and will otherwise jump out of the support of the limiting spectral distribution. In the latter case, we call the extreme eigenvalue as an outlier, and the associated eigenvector as an outlier eigenvector. Moreover, the fluctuation of the extreme eigenvalues in different regimes (subcritical, critical and supercritical) are also identified in \cite{BBP} for the complex spiked covariance matrix.  We also refer to \cite{ BaiS, BGN, BGN1, BaiY,  CDM2016, Ding, DY, KY13, DP2007} and the reference therein  for the first-order limit of the extreme eigenvalue of various fixed-rank deformation models. The fluctuation of the extreme eigenvalues of various models have been considered in \cite{bai2008, BaiY, bao2017, bao2015, benaych-georges2011,CDF,  CDF12,  elkaroui2007, BV131, BV132,  KY13,  DP2007, FP07, P2006, PRS13, KY14AOP}. Especially, the fluctuations of the outliers are shown to be non-universal for the deformed Wigner matrices, first  in \cite{CDF} under certain special assumptions on the structure of the deformation and the distribution of the matrix entries, and then in \cite{KY13} in full generality. 

The study on the behavior of the extreme eigenvectors has been mainly focused on the level of the first order limit \cite{BGN, BGN1,  CM2017,  Ding, FJS, DP2007}. In parallel to the results of the extreme eigenvalues, it is known that the  eigenvectors  are delocalized  in the subcritical case and have a bias on the direction of the deformation in the supercritical case. It is  recently observed  in \cite{BEKYY2016} that a deformation close to the critical regime will cause a bias even for the non-outlier eigenvectors.  On the level of the fluctuation, the limiting behavior of the extreme eigenvectors has  not been  fully studied yet.  By establishing a general universality result of the eigenvectors of the sample covariance matrix in the null case, the authors of \cite{BEKYY2016} are able to show that the law of the  eigenvectors of the spiked covariance matrices are asymptotically Gaussian in the subcritical regime. More specifically, the generalized components of the eigenvectors (i.e. $\langle \widehat{\bm{v}}_i, \bm{w} \rangle$ for any deterministic vector $\bm{w}$) are $\chi^2$ distributed. For spiked  Gaussian sample covariance matrices, in the supercritical regime,  the fluctuation of a fixed-dimensional normalized  subvector of the outlier eigenvector is proved to be Gaussian in \cite{DP2007}, but this result cannot tell the distribution of $\langle \widehat{\bm{v}}_i, \bm{v}_i \rangle$. Under some special assumptions on the structure of the  deformation and the distribution of the random matrix entries, it is shown in \cite{CDM2018} that the eigenvector distribution of a generalized deformed Wigner matrix model is non-universal in the supercritical regime.  In the current work, we aim at establishing the non-universality for the outlier singular vectors  for the matrix  denoising model under fully general assumptions on the structure of the deformation $S$ and the distribution of the random matrix $X$. This can be regarded as an eigenvector counterpart of the result on the outlying eigenvalue distribution in \cite{KY13}.

\subsection{On singular subspace inference}\label{subsec:application} From the applied perspective, 
our model (\ref{defn_m1}) appears prominently in the study of signal processing \cite{KSM1998, bjohn2011}, machine learning \cite{wang2011, yang2016} and statistics \cite{CZ2018, 2018arXiv180200381C,  donoho2014, FZ2018}. For instance, in the study of image denoising, $S$ is treated as the true image \cite{levin2011} and in the problem of classification, $S$ contains the  underlying true mean vectors of samples \cite{CZ2018}. In both situations, we need to understand the asymptotics of the singular vectors and subspace of $S$, given the observation $Y.$ In addition, the statistics $R$ and $L$ defined in (\ref{18012401}) can be used for the inference of the structure of the singular subspace of $S.$ We remark that these statistics have been used extensively to explore the properties of singular subspace.  To name a few,  in \cite{HQC}, the authors studied the problem of testing whether the sample singular subspace is equal to some given subspace; in \cite{CTP}, the authors studied the eigenvector inference problems for the correlated stochastic block model; in \cite{HTB}, the authors analyzed the impact of dimensionality reduction for subspace clustering algorithms; and in \cite{CZ2018}, the authors studied the  high-dimensional clustering problem and the canonical correlation analysis. In the high-dimensional regime (\ref{eq_ratioassumption}), to the best of our knowledge, the distributions of  $R$ and $L$ have not been studied yet in the literature.

In the situation when $M$ is fixed, the sample eigenvectors of  $XX^*$ are normally distributed \cite{AT}. When $M$ diverges with $n,$ many interesting results have been proposed under various assumptions. One line of the work is to derive the perturbation bounds for the perturbed singular vectors based on Davis-Kahan's theorem. For instance, in \cite{OVW}, the authors improve the perturbation bounds of Davis-Kahan theorem to be nearly optimal. In \cite{CZ2018}, the authors study similar problems and their related statistical applications. Most recently, in the papers \cite{FWZ2018,FZ2018, ZB2017}, the authors derive  the $\ell^{\infty}$ pertubation bounds  assuming that the population vectors were delocalized (i.e. incoherent).  The other line of the work is to study the asymptotic normality of the spectral projection under various regularity conditions.  In such cases,  the singular vectors of $S$ can be estimated using those of $Y$ and some  Gaussian approximation technique can be employed.  Considering the Gaussian data samples $\mathbf{x}_i \simeq \mathcal{N}(\mathbf{0}, \Sigma), i=1,2,\cdots,n$ and $X=(\mathbf{x}_i),$ under the assumption that the order of $\frac{\text{Tr} \Sigma}{\| \Sigma \|}$ is much smaller than $n,$  in \cite{KLN,KL, KLAOS}, the authors prove that  the eigenvectors of $X X^*$  are asymptotically normally distributed, whose variance depends the eigenvectors of $\Sigma$. Furthermore, in \cite{Xia}, assuming that $m$ such random matrices $X_i, i=1,2,\cdots,m$ are available, the author shows that  the singular vectors  of $S$ can be estimated via trace regression using  matrix nuclear
norm penalized least squares estimation (NNPLS). Under the assumption that $r^4 K \log^3 m=o(m), \ K=\max\{M,n\},$ the author shows that the principal angles of the subspace estimated using NNPLS are asymptotically normal.

\subsection{Organization}   The rest of the  paper is organized as follows. In Section \ref{s.result and method}, we state our main results and summarize our method for the proofs. In Section \ref{sec:simuandstat}, we design Monte Carlo simulations to demonstrate the accuracy of our main results and briefly illustrate their applications through a  hypothesis testing problem. In Section  \ref{s. 18092602}, we introduce some main technical results including the isotropic local law and also derive the Green function representation for our statistics.  In Section \ref{sec_mainthm}, we prove Theorems \ref{thm_mainthm}, based on the recursive estimate in Proposition \ref{prop_iteration}.  
We state more  simulation results, further discussions of statistical applications, the proofs of Theorem \ref{thm. right subspace}  and some technical lemmas in the supplementary material \cite{BDW_suppl}.

\section{Main results and methodology} \label{s.result and method}
In this section, we state our main results, and briefly summarize our proof strategy.

 \subsection{Main results} \label{main results}
In this paper,  the singular values of $S$ are assumed to satisfy the following {\it supercritical } condition.

{
\begin{assum}[Supercritical condition]\label{assum_main} There exist a constant $C>0$ and a (small) constant $\delta>0,$ such that 
\begin{equation*}
{y^{1/4}+\delta\leq d_r<\cdots<d_2<d_1\leq C }, \quad \quad \min_{1 \leq j \neq i \leq r}|d_i-d_j| \geq \delta.
\end{equation*}
\end{assum} 
}
\begin{rem}
The first inequality above ensures that the first $r$ singular values of $Y$ are outliers, and the threshold $y^{1/4}$ is the analogous BBP transition point in \cite{BBP}. The second inequality guarantees that the outliers of $Y$ are well separated from each other. We also assume that $d_1, \cdots, d_r$ are bounded by some constant $C$.  All these conditions can be weakened. For instance, we do allow the existence of the subcritical and critical $d_i$'s if we only focus on the outlier singular vectors.  Also, the separation of $d_i$'s by an order $1$ distance $\delta$ is not necessary. In  \cite{BEKYY2016}, a much weaker  separation of  order $n^{-1/2+\epsilon}$ is enough for the discussion of the eigenvalues. Moreover, we can also extend our results to the case when  $d_1, \cdots, d_r$ diverge with $n$. But we do not pursue these directions in the current paper.  
\end{rem}

In the sequel, we will only state the results for the right singular vectors and the right singular subspace. The results for the left ones can be obtained from the right ones by simply considering  the transpose (with a rescaling) of our matrix model in (\ref{defn_m1}). To state our results, we need more notations. First, we define
\begin{align}
p(d):=\frac{(d^2+1)(d^2+y)}{d^2}.  \label{18012901}
\end{align}
For each $i\in [r]$, we will write  $p_i\equiv p(d_i)$ for short. Recall (\ref{eq_ysvd}). In \cite[Theorem 3.4]{Ding}, it has been shown that $p_i$ is the  limit of $\mu_i$. Further, we set
\begin{equation} \label{eq_defna1a2}
 a(d):=\frac{d^4-y}{d^2(d^2+1)}.
\end{equation}
It has been proved in \cite{Ding} that  $a(d_i)$ are the  limits of   $|\langle \mathbf{v}_i, \widehat{\mathbf{v}}_i \rangle |^2$ respectively (see Lemma \ref{lem:Ding} in \cite{BDW_suppl}).   
We also denote by $\kappa_l$   the $l$-th cumulant of the random variables $\sqrt{n} x_{ij}$. For a vector $\mathbf{w}=(w(1),\ldots, w(m))^T$  and $l\in \N$, we introduce the notation 
\begin{align*}
\mathbf{s}_l (\mathbf{w}) := \sum_{i=1}^m w(i)^l.
\end{align*}
Set
\begin{align}\label{eq_defnthed}
\theta(d):=\frac{d^4 + 2 y d^2 + y}{d^3 (d^2+1)^2}, \quad\quad
\psi(d):=\frac{d^6 - 3 y d^2 - 2 y}{d^3 (d^2+1)^2}, 
\end{align}
and
\begin{align}
\mathcal{V}^E(d):=&\frac{2}{d^4 - y} \bigg(  2 y (y +1) \theta(d)^2 - \frac{y(y-1)(5y+1)}{d (d^2+1)^2}\theta(d)  \nonumber \\
&+  \frac{(d^4+y)(d^2 + y)^2 }{d^3 (d^2+1)^2} \psi(d)+ \frac{2 y^{2} (y-1)^2}{d^2 (d^2+1)^4} \bigg). \label{eq_ved}
\end{align}
For the right singular vectors, we have the following theorem. 
\begin{thm}[Right singular vectors]\label{thm_mainthm} Assume that  (\ref{eq_meanva}), (\ref{eq_moment}), (\ref{eq_ratioassumption}) and Assumption \ref{assum_main} hold. For  $i\in [r]$, define the random variable
\begin{align}\label{def:Upsilon}
\Delta_i :=- 2\sqrt{n} \theta(d_i) \mathbf{u}_i^* X \mathbf v_i - \frac{2\psi(d_i)}{d_i^2} \big(\frac{\kappa_3}{n} \mathbf{s}_1(\mathbf u_i) \mathbf{s}_{1}(\mathbf v_i) \big),
\end{align}
and let
$\mathcal{Z}_i\sim \mathcal{N}\left(0, \mathcal{V}_{i} \right)$ be a random variable independent of $\Delta_i$,  where
\begin{align*}
\mathcal{V}_i&:= \mathcal{V}^E(d_i) -\frac{4}{d_i}\theta(d_i) \psi(d_i) \big( \frac{\kappa_3}{ \sqrt n} \mathbf{s}_3(\mathbf u_i) \mathbf{s}_{1}(\mathbf v_i) \big) + \frac{4}{d_i} \theta(d_i)^2 \big( \frac{\kappa_3}{ \sqrt n} \mathbf{s}_1(\mathbf u_i) \mathbf{s}_{3}(\mathbf v_i) \big)\\
&\qquad+ \frac{\psi(d_i)^2}{d_i^2} \kappa_4  \mathbf{s}_4(\mathbf u_i) + \frac{y \theta(d_i)^2}{d_i^2} \kappa_4  \mathbf{s}_4(\mathbf v_i).
\end{align*}
Then for any $i\in [r]$ and any bounded continuous function $f$, we have 
\begin{equation*}
\lim_{n \rightarrow \infty} \Big( \mathbb{E} f \left(\sqrt{n} \left(  |\langle\mathbf{v}_i, \widehat{\mathbf{v}}_i\rangle  |^2-a(d_i)  \right) \right)-\mathbb{E} f(\Delta_i + \mathcal{Z}_{i})  \Big)=0.
\end{equation*}
\end{thm}

\begin{rem}
In \cite{KY13}, the authors obtain the non-universality  for the limiting distributions of the outliers (outlying eigenvalues) of the deformed Wigner matrices. The limiting distributions admit similar forms as the limiting distribution for the outlier singular vectors for our models. One might notice that the third or the fourth cumulants of the entries of the Wigner matrices are allowed to be different in \cite{KY13}. An extension along this direction is also straightforward for our result. 
\end{rem}

We discuss a few special cases of interest. For simplicity, we assume that $S$ has rank $r = 1$ and drop all the subindices. 
\begin{rem}\label{rem_singu1}  If the entries of $\sqrt n X$ are standard Gaussian random variables (i.e. $\kappa_3=\kappa_4=0$), then $\Delta \simeq \mathcal N(0,  4\theta(d)^2)${ (see Definition \ref{defn_asymptotic} for the meaning of $\simeq$).  Hence, we find} $\Delta+\mathcal Z$ is asymptotically distributed as
\begin{align*}
\mathcal N \left(0, 4 \theta(d)^2 + \mathcal V^E(d)\right).
\end{align*}
\end{rem}
\begin{rem}\label{rem_singu2} If both $\mathbf u$ and $\mathbf v$ are delocalized in the sense that { $\|\mathbf u\|_{\infty}=o(1)$ and $\|\mathbf v\|_{\infty}=o(1)$.}  Then $\mathbf s_l(\mathbf u) = o(1)$ and $\mathbf s_l(\mathbf v) = o(1)$ for $l=3,4$. {By (\ref{eq_meanva}), (\ref{eq_moment}) and the fact $\| \bm{u} \|_2= \| \bm{v} \|_2=1$, we find that $\mathbb{E}(\bm{u}^*X \bm{v})=0$ and $\mathbb{E}(\bm{u}^*X \bm{v})^2=n^{-1}.$ Then} we conclude from Lyapunov's  CLT for triangular array that 
\begin{align}\label{eq:limitofLambda}
\Delta \simeq \mathcal N \left( - \frac{2\psi(d)}{d^2} \big(\frac{\kappa_3}{n} \mathbf{s}_1(\mathbf{u}) \mathbf{s}_1(\mathbf{v}) \big), 4 \theta(d)^2 \right),
\end{align}
 and therefore $\Delta + \mathcal Z$ has asymptotically the same distribution as 
\begin{align*}
\mathcal N \left(- \frac{2\psi(d)}{d^2} \big(\frac{\kappa_3}{n} \mathbf{s}_1(\mathbf{u}) \mathbf{s}_1(\mathbf{v}) \big), 4 \theta(d)^2 + \mathcal{V}^E(d)\right).
\end{align*}
The only difference from the Gaussian case is a shift caused by the non-vanishing third cumulant.
\end{rem}
\begin{rem}\label{rem_singu3} If one of $\mathbf u$ and $\mathbf v$ is delocalized, say $\| \mathbf u\|_\infty=o(1)$, then $\Delta$ still has the limiting distribution in \eqref{eq:limitofLambda}. Therefore $\Delta + \mathcal Z$ has asymptotically the same distribution as a Gaussian
random variable with mean $$- \frac{2\psi(d)}{d^2} \big(\frac{\kappa_3}{n} \mathbf{s}_1(\mathbf{u}) \mathbf{s}_1(\mathbf{v}) \big)$$ and variance
$$4 \theta(d)^2 + \mathcal{V}^E(d) + \frac{4}{d}\theta(d)^2 \big( \frac{\kappa_3}{ \sqrt n} \mathbf{s}_1(\mathbf{u}) \mathbf{s}_{3}(\mathbf{v}) \big) +y \frac{\theta(d)^2}{d^2} \kappa_4  \mathbf{s}_4(\mathbf{v}). $$
\end{rem}

\begin{rem}\label{rem_singu4} If neither $\mathbf u$ nor $\mathbf v$ is delocalized, then $\Delta + \mathcal Z$ is no longer Gaussian in general. For example, if $\mathbf u=\mathbf e_1$ and $\mathbf v= \mathbf f_1$ where $\mathbf e_1$ and $\mathbf f_1$ are the canonical basis vectors in $\R^M$ and $\R^n$ respectively, then $\Delta + \mathcal Z$ is asymptotically distributed as 
$$-2\theta(d) \sqrt n X_{11} + \mathcal N \left(0, \mathcal{V}^E(d) + \kappa_4 \frac{\psi(d)^2 + y \theta(d)^2}{d^2} \right),$$
which depends on the distribution of $X_{11}$ and thus is non-universal. 
\end{rem}

If the assumptions of Theorem 2.3 hold,  we conclude from Remarks 2.6--2.9 that $|\langle \bm{v}_i, \widehat{\bm{v}}_i \rangle|^2$ always has a Gaussian fluctuation if either the entries of $X$ are Gaussian or one of $\bm{u}_i$ and $\bm{v}_i$ is delocalized in the sense $\|\bm{u}_i\|_\infty=o(1)$ or $\|\bm{v}_i\|_\infty=o(1)$. In the general setting when the noise matrix is non-Gaussian, the detailed distribution will rely on both the structure of the singular vectors and the noise matrix $X.$  

Next,  we study the distributions of the right  singular space.
{  For two vectors $\mathbf w_a=(w_a(1),\ldots,w_a(m))^T, a=1,2$, we denote
\begin{align*}
\mathbf s_{k,l} (\mathbf w_1, \mathbf w_2):=\sum_{i=1}^m w_1(i)^k w_2(i)^l.
\end{align*}
}
 Recall $R$ from \eqref{18012401}. We have the following theorem.
\begin{thm}[Right singular subspace] \label{thm. right subspace} Assume that  (\ref{eq_meanva}), (\ref{eq_moment}), (\ref{eq_ratioassumption}) and Assumption \ref{assum_main} hold. Let $\Delta = \sum_{i=1}^r \Delta_i$,  where $\Delta_i$ is defined in \eqref{def:Upsilon}. Let $\mathcal{Z}$ be a random variable independent of $\Delta$ with law $ \mathcal{Z} \sim \mathcal N(0,  \mathcal{V})$, where
\begin{align*}
 \mathcal V &:= \sum_{i=1}^r \mathcal{V}^E(d_i) +\kappa_4\sum_{i,j=1}^r \left(  \frac{\psi(d_i) \psi(d_j)}{d_i d_j} \mathbf{s}_{2,2}(\mathbf u_i, \mathbf u_j) + y \frac{\theta(d_i) \theta(d_j)}{d_i d_j} \mathbf{s}_{2,2}(\mathbf v_i, \mathbf v_j) \right)\\
&\qquad+ \frac{\kappa_3}{\sqrt n}\sum_{i,j=1}^r \frac{4}{d_i}\theta(d_j) \Big(\theta(d_i) \mathbf{s}_{2,1}(\mathbf v_i, \mathbf v_j) \mathbf{s}_1(\mathbf u_j) -\psi(d_i) \mathbf{s}_{2,1}(\mathbf u_i, \mathbf u_j) \mathbf{s}_1(\mathbf v_j)  \Big).
\end{align*}
Then  for any bounded continuous function $f$, we have that 
\begin{equation*}
\lim_{n \rightarrow \infty} \Big( \mathbb{E} f \big(\sqrt{n} \big( R- \sum_{i=1}^r a(d_i)  \big) \big)-\mathbb{E} f(\Delta +\mathcal{Z})  \Big)=0.
\end{equation*}
\end{thm}

\subsection{Proof strategy} 
In this subsection, we briefly describe our proof strategy.  We first  review the method used in a related work \cite{KY13}, and then we highlight the novelty of our strategy.

As we  mentioned, in \cite{KY13}, the authors derive the distribution of outliers (outlying eigenvalues) of the fixed-rank deformation of Wigner matrices.  The main technical input is the isotropic local law for Wigner matrices, which provides a precise large deviation estimate for the quadratic form $\langle\mathbf{u}, (W-z)^{-1}\mathbf{v}\rangle$ for any deterministic vectors $\mathbf{u}, \mathbf{v}$. Here  $W$ is a  Wigner matrix. It turns out that an outlier of the deformed Wigner matrix can also be  approximated by  a quadratic form of the Green function, of the form $\langle\mathbf{u}, (W-z)^{-1}\mathbf{u}\rangle$. So one can turn to establish the law of the quadratic form  of the Green function instead. In \cite{KY13}, the authors decompose the proof into three steps. First, the law is  established for the GOE/GUE, the Gaussian Wigner matrix, for which orthogonal/unitary invariance of the matrix can be used to facilitate the proof. In the second step of going beyond Gaussian matrix, in order to capture the independence of the Gaussian part and the non-Gaussian part of the limiting distribution of the outliers, the authors construct an intermediate matrix in which  most of the matrix entries are replaced by the Gaussian ones while those with coordinates corresponding to the large components of $\mathbf{u}$ are kept as generally distributed.  The intermediate matrix allows one to use the nice properties of the Gaussian ensembles such as orthogonal/unitary invariance for the major part of the matrix, and meanwhile keeps the non-Gaussianity induced by the small amount of generally distributed entries. In the last step, the authors of \cite{KY13} derive the law for the fully generally distributed Wigner matrix by further conducting a Green function comparison with the intermediate matrix.

For our problem, similarly, we will use the isotropic law of the sample covariance matrix in \cite{BEKYY2014, KY14} as a main technical input. It turns out that for the singular vectors, we can approximately represent $\sqrt{n}|\langle \widehat{\mathbf{u}}_i, \mathbf{u}_i\rangle |$ (after appropriate {centering}) in terms of a quantity of the form 
\begin{align}
\mathcal{Q}_i=\sqrt{n}\Big(\text{Tr} (G(p_i))-\Pi_1(p_i))A_i+ \text{Tr}(G'(p_i)-\Pi_1'(p_i))B_i \Big), \label{19091901}
\end{align}
where $G$ is the Green function of the linearization of the sample covariance matrix and $\Pi_1$ is the deterministic approximation of $G$; see (\ref{green2}) and (\ref{eq_piz}) for the definitions.  Here both $A_i$ and $B_i$ are deterministic fixed-rank matrices. Hence, differently from the outlying eigenvalues or singular values, the Green function representation of the singular vectors also contains the derivative of the Green function.  More importantly, instead of the three step strategy in \cite{KY13}, here we derive the law of the above $\mathcal{Q}_i$ directly for generally distributed matrix. Recall $\Delta_i$ defined in (\ref{def:Upsilon}), whose random part is proportional to $\mathbf{u}_i^*X\mathbf{v}_i$, which is simply a linear combination of the entries of $X$.   Inspired by \cite{KY13}, we decompose $\Delta_i$ into two parts, say $\widetilde{\Delta}_{i}$ and $\widehat{\Delta}_i$. The former contains the linear combination of $x_{k\ell}$'s for those indices $k,\ell$ corresponding to the large components $u_{ik}$ and $v_{i\ell}$ in $\mathbf{u}_i$ and $\mathbf{v}_i$. The latter  contains the linear combinations of the rest of $x_{k\ell}$'s. Note that $\widehat{\Delta}_i$ is asymptotically normal by CLT since the coefficients of $x_{k\ell}$'s are small. However, $\widetilde{\Delta}_{i}$ may not be normal. 
The key idea of our strategy is to  show the following recursive estimate: For any fixed $k\in \mathbb{N}$, we have 
\begin{align}
\mathbb{E} (\mathcal{Q}_i-\widetilde{\Delta}_i)^{k} \mathrm{e}^{\ii t \widetilde{\Delta}_i}= (k-1) \widetilde{\mathcal{V}}_i \mathbb{E}  (\mathcal{Q}_i-\widetilde{\Delta}_i)^{k-2} \mathrm{e}^{\ii t \widetilde{\Delta}_i}+o(1), \label{180924100}
\end{align}
for some positive number $\widetilde{\mathcal{V}}_i$. 
Choosing $t=0$, we can derive the asymptotic normality of $\mathcal{Q}_i-\widetilde{\Delta}_i$ for (\ref{180924100}) by the recursive moment estimate. Choosing $t$ to be arbitrary, we can further deduce from (\ref{180924100})  that
\begin{align*}
\mathbb{E}\mathrm{e}^{\ii s(\mathcal{Q}_i-\widetilde{\Delta}_i)+\ii t\widetilde{\Delta}_i}= \mathbb{E}\mathrm{e}^{\ii s(\mathcal{Q}_i-\widetilde{\Delta}_i)}\mathbb{E}\mathrm{e}^{\ii t\widetilde{\Delta}_i}+o(1).
\end{align*}
Then asymptotic independence between $\mathcal{Q}_i-\widetilde{\Delta}_i$ and $\widetilde{\Delta}_i$ follows.  Hence, we prove both the asymptotic normality and asymptotic independence from (\ref{180924100}). The method of  using the recursive estimate to get the large deviation bounds for Green function or  some functional of the Green functions has been previously used in the context of the Random Matrix Theory. For instance, we refer to \cite{2016arXiv160508767O}. However, as far as we know, it is the first time to use the recursive estimate to show the normality and the independence simultaneously for the functionals of the Green functions.

 Moreover, we remark that 
 the approach in this paper can also be applied to derive the distribution of the outlier eigenvectors of the spiked sample covariance matrix \cite{BDWW} and the deformed Wigner matrix.

Finally, we briefly compare the methods used in this paper and  the  related work \cite{CDM2018}.  In \cite{CDM2018}, the authors study the distribution of $|\langle \widehat{\bm{v}}, \bm{e}_1 \rangle|^2$ of a deformed Wigner matrix whose deformation is a block diagonal  deterministic Hermitian matrix containing one large spike $\theta\mathbf{e}_1\mathbf{e}_1^*$ which creates one outlier of the deformed Wigner matrix. Here $\widehat{\bm{v}}$ is the random outlier eigenvector. By Helffer-Sj{\" o}strand formula, they represent $|\langle\widehat{\bm{v}}, \bm{e}_1  \rangle|^2$ in terms of an integral (over $z$) of  $\bm{e}_1^*(W-z)^{-1} \bm{e}_1$. In contrast to our work, the major difference in \cite{CDM2018} is that they establish the limiting distribution for the whole process $\bm{e}_1^*(W-z)^{-1} \bm{e}_1$ in $z$, and then use functional limit theorem to conclude the limit of the integral. In our work, relying on the isotropic law, we first integrate out  the contour integral approximately. This results in the linear combination in (\ref{19091901}), and then we only need to consider the joint distribution of the quadratic form of $G$ and $G'$ at a single point $p(d_i)$. 
Moreover, in \cite{CDM2018}, the authors decompose the quadratic form $\bm{e}_1^*(W-z)^{-1} \bm{e}_1$ into two parts using Schur's complement, where one of them can be proved to be Gaussian using an extension of the CLT for quadratic forms as in the previous work \cite{CDF}. It is worth noticing that the independence between the Gaussian and non-Gaussian parts follows directly from the special structure of the model in \cite{CDM2018}. However, in \cite{KY13} and our work, since we do not have structural assumptions on $S,$ we need to make more dedicated efforts for the independence (see \cite[Proposition 7.12]{KY13} and Proposition \ref{lem:main}).

\vspace{3mm} 

\section{Simulations and statistical applications} \label{sec:simuandstat}
\subsection{Numerical simulations}\label{s.numerical}  In this section, we present some numerical simulations for  our results stated in Section \ref{main results}.  For the simulations, we consider two specific distributions for our noise matrix. We assume  that $\sqrt{n} x_{ij}$'s are i.i.d. $\mathcal{N}(0,1)$ or i.i.d. with  the distribution $\frac13\delta_{\sqrt{2}}+\frac{2}{3}\delta_{-\frac{1}{\sqrt{2}}}$.  
We call these two types of noise as Gaussian noise and Two-Point noise, respectively.  It is easy to check that the 3rd and 4th cumulants of the distribution $\frac13\delta_{\sqrt{2}}+\frac{2}{3}\delta_{-\frac{1}{\sqrt{2}}}$ are $\kappa_3=\frac{1}{\sqrt{2}}$ and $\kappa_4=-\frac{3}{2}.$ In the sequel, let $\{\mathbf e_i\}_{i=1}^M$ and $\{\mathbf f_j\}_{j=1}^n$  be the canonical basis of  $\R^M$ and $\R^n$, respectively. Denote by $\mathbf 1_m$ the all-one vector in $\R^m$.  

Assume that $S$ has rank $r=1$ and admits the  singular value decomposition $S=d \mathbf u^T \mathbf v$. Set the dimension ratio $y=M/n=0.5$.  We present the simulations corresponding to the special cases discussed in Remarks \ref{rem_singu1} - \ref{rem_singu4}.  Specifically, we consider following four cases: 1. Gaussian noise, $\bm{u}=\bm{e}_1$ and $\bm{v}=\bm{f}_1$; 2. Two-point noise, $\bm{u}=\bm{1}_{M}/\sqrt{M}$ and $\bm{v}=\bm{1}/\sqrt{n}$; 3. Two-point noise, $\mathbf{u}=\bm{1}_M /\sqrt{M}$ and $\bm{v}=\bm{f}_1$;  4. Two-point noise, $\bm{u}=\bm{e}_1$ and $\bm{v}=\bm{f}_1.$ The normalization of $\sqrt{n} (|\langle \widehat{\mathbf{v}}, \mathbf{v}  \rangle|^2-a(d))$ listed in the above cases are chosen according to the calculations in Remarks \ref{rem_singu1} - \ref{rem_singu4}. For case 4, we further subtract the non-Gaussian part $-2\theta(d)\sqrt{n}X_{11}$ from the statistic. Hence, in all four cases, we expect  that the asymptotic distributions are  normal.  We denote the normalized statistics of the above four cases as $\mathcal{R}_g, \mathcal{R}_{dt}, \mathcal{R}_{pt}$ and $\mathcal{R}_{st},$, respectively, and we refer to the supplementary material  \cite[Section A]{BDW_suppl} for more details on the definitions.

In Figure \ref{fig_fig1} of \cite{BDW_suppl}, we plot the ECDFs of of $\mathcal{R}_g, \mathcal{R}_{dt}, \mathcal{R}_{pt}, \mathcal{R}_{st}$ in subfigures (A), (B), (C), (D) respectively, for $n=500$ and various values of $d=2,3,5,10$. The distributions of these quantities are fairly close to the standard normal distribution.  In \cite[Section A]{BDW_suppl}, we also record the probabilities for different quantiles of the empirical cumulative distributions (ECDFs) of the above statistics, they are fairly close to standard Gaussian even for a small sample size $n=200.$

\vspace{2mm}

\subsection{Statistical applications}\label{sec:statapp} In this section, we will briefly discuss the applications of our main results  to  the singular vector and singular subspace estimation and inference, and leave more details to the supplementary material \cite{BDW_suppl}.  

We start with the estimation part and focus on the right singular vector and subspace.   The estimation of singular vector and  subspace is important in the recovery of low-rank matrix based on noisy observations (see for instance \cite{CZ2018, CTP,  donoho2014} and reference therein). It is clear that (see Lemma \ref{lem:Ding} in \cite{BDW_suppl}) the sample singular vector is concentrated on a cone with axis 
parallel to the true singular vector. The aperture of the cone is determined by the deterministic function $a(d)$ defined in (\ref{eq_defna1a2}).  Further, when $d$ increases, the sample singular vector will get closer to the true singular vector in  $\ell^2$ norm.  It can be seen from the result in Theorem \ref{thm_mainthm} that the variance of the fluctuation also decays when $d$ increases. This phenomenon is recorded in Figure  \ref{fig_fig11} in the supplementary material \cite{BDW_suppl}.

Empirically, it can be seen from Figure \ref{fig_fig11} in \cite{BDW_suppl} that for a sequence of $y\in [\frac{1}{10},10]$, when $d>5,$ the variance part is already very small and hence the fluctuation can be ignored. Further, when $d>7.5,$ we can use the sample singular vector to estimate the true singular vector since their inner product is rather close to $1$.  Finally, note that the noise type will affect the variance of the fluctuation. Especially when the noise has negative $\kappa_3$ and $\kappa_4,$ we can ignore the fluctuation for a smaller value of $d.$ Once the singular vectors  are estimated, the estimation of the singular subspace follows.

Next, we consider the inference of the singular vectors and subspace of $S.$ Recall the decomposition in (\ref{eq_ssvd}).  For brevity, here we focus our discussion on the inference of $V$,  assuming that $U,D$ and the necessary parameters of $X$ (e.g. cumulants of the entries of X) are known. In the supplementary material \cite{BDW_suppl}, we will also briefly discuss the possible extension of our results to adapt to the situation when $D$ and the parameters of $X$ are not known. 
Especially, using Theorem \ref{thm_mainthm} we can test whether a singular vector $\mathbf{v}_i$ is equal to a given vector $\mathbf{v}_{i0}$, which can be formulated as 
\begin{align}\label{eq_test0}
\mathbf{H}_0: \mathbf{v}_i= \mathbf{v}_{i0}, \qquad \mathbf{H}_a: \mathbf{v}_i\neq  \mathbf{v}_{i0},  \tag{T0}
\end{align} 
and we can choose the testing statistic to be
\begin{align*}
\bm{S}_0:= \sqrt{n}(|\langle \widehat{\mathbf{v}}_i, \mathbf{v}_{i0}\rangle|^2-a(d_i)).
\end{align*}
 Further, using Theorem \ref{thm. right subspace}, one can test if the matrix  $V$ is equal to a given matrix, 
which  can be formulated as 
\begin{align}\label{test_1}
\mathbf{H}_0: V=V_0, \quad \mathbf{H}_a: V \neq V_0,  \tag{T1}
\end{align}
where $V_0=(\mathbf{v}_{10}, \ldots, \mathbf{v}_{i0})$ is a given matrix consisting of orthonormal columns.
We can choose the testing statistic to be 
\begin{align}\label{eq_s1_defn}
\bm{S}_1 &= \sqrt{n}\Big(\sum_{i,j=1}^r |\langle\widehat{\mathbf{v}}_i,\mathbf{v}_{j0}\rangle|^2-\sum_{i=1}^r a(d_i)\Big)\nonumber\\
&=\sqrt{n}\Big(\frac12 \Big(2r-\|\widehat{V}_r\widehat{V}_r^*-V_0V_0^*\|_F^2\Big)-\sum_{i=1}^r a(d_i)\Big). 
\end{align} 
We remark here that in some cases like $X$ is Gaussian, we can see  from Theorem \ref{thm. right subspace} that $\bm{S}_1$ is not a good statistic to distinguish $V_0$ from $V_0O$ for some deterministic $r\times r$ orthogonal matrix $O$. Specifically, one cannot tell if $\widehat{V}_r$ is the matrix of the  singular vectors of the model $X+UDV_0^*$ or  $X+UD(V_0O)^*$, since $V_0V_0^*=(V_0O)(V_0O)^*$ in (\ref{eq_s1_defn}) and the limiting distribution of $\bm{S}_1$ does not depend on $V$ when $X$ is Gaussian. Hence, we do not expect the statistic $\bm{S}_1$ to be powerful for the test (T1) when the alternative is of the form $V_0O$ in some cases like Gaussian noise. In other words, in this case, what one can test is if $VV^*=V_0V_0^*$. Nevertheless, one can still do the test (T1) by using the testing statistic of the diagonal parts of $\bm{S}_1$ only, i.e., $\bm{S}_{1d}=\sqrt{n}\Big(\sum_{i}^r |\langle\widehat{\mathbf{v}}_i,\mathbf{v}_{i0}\rangle|^2-\sum_{i=1}^r a(d_i)\Big)$.  Under the null hypothesis, $\mathbf{S}_{1d}$ has the same distribution as $\bm{S}_1$ since it will be clear that $ |\langle\widehat{\mathbf{v}}_i,\mathbf{v}_{j0}\rangle|^2$ is negligible if $i\neq j$, in the null case.  But note that the limiting distribution of $\mathbf{S}_{1d}$ is no longer invariant under taking right orthogonal transformation for $V_0$.   Hence, it can be used to test if $V=V_0$.

We mention that both (\ref{eq_test0}) and (\ref{test_1}) could be useful in many scientific disciplines, especially when the singular vectors of $S$ are sparse and have practical meanings. For instance, an important goal of the study of gene expression data for cancer is to simultaneously identify related genes and subjects grouped together according to the cancer types \cite[Section 2]{LSHM}. For this purpose, the right singular vectors are used to visualize the gene grouping (see Figure 1 of \cite{LSHM}) and the left singular vectors are used to represent the subject grouping (see Figure 2 of \cite{LSHM}). Other examples include the study of the nutrition content data of different foods \cite{LSHM} and the mortality rate data after expanding on suitable basis functions \cite[Section 3]{YMB}.  In the literature, various algorithms have been proposed to estimate the sparse singular vectors, for instance see \cite{Ding, LSHM, YMB, yang2016}. From the statistical perspective, with the above estimates, it is natural to do inference on the singular vectors.  For instance, for the gene expression data of lung caner,  researchers may be interested in testing whether a certain type of cancer is determined by a subset of genes and this is related to doing inference on the right singular vectors and right singular subspace.

Since we assume that $U,D$ and the necessary parameters of $X$ (e.g. cumulants of the entries of X) are known, we can carry out the $z$-score test to test $\mathbf{H}_0$ in both (\ref{eq_test0}) and (\ref{test_1}). Due the similarity of (\ref{eq_test0}) and (\ref{test_1}), we focus on (\ref{test_1}) and leave the detailed discussions and simulations to the supplementary material \cite{BDW_suppl}.

\section{Techincal tools and Green function representations} \label{s. 18092602}  
This section is devoted to providing some basic notions and technical tools, which will be needed often in our proofs for the  theorems. The basic notions are given in Section \ref{subsec:pre}. A main technical input for our proof is the isotropic local law for the sample covariance matrix obtained in \cite{BEKYY2014, KY14}. It will be stated in Section \ref{subsec:local}. In subsection \ref{subsec:key}, we represent (asymptotically) $|\langle \widehat{\mathbf{v}}_i, \mathbf{v}_i \rangle |^2$'s and  $R$ (c.f.  \eqref{18012401}) in terms of  the Green function. The discussion is based on  \cite{Ding}, where the limits for $| \langle \widehat{\mathbf{u}}_i, \mathbf{u}_j \rangle |^2$ and $|\langle \widehat{\mathbf{v}}_i, \mathbf{v}_j \rangle |^2$ are studied. We then collect a few auxiliary definitions in Section \ref{subsec:aux}.

\subsection{Basic notions}\label{subsec:pre}

For a positive integer $n$, we denote by  $[n]$ the set $\{1,\cdots,n\}$.  
Let $\mathbb{C}^+$ be the complex upper-half plane. Further,  we define the following  linearization for our model
\begin{align}
 \mathcal{Y}(z):=\mathcal{U}\mathcal{D}(z)\mathcal{U}^*+H(z), \quad\quad z=E+\ii\eta \in \mathbb{C}^+,  \label{180130110}
\end{align}
where
\begin{align}
\mathcal{U}:= \left( 
\begin{array}{ccc}
U  &~\\
~ & V
\end{array}
\right),\quad  \mathcal{D}(z):= \sqrt{z}\left( 
\begin{array}{ccc}
~  &D\\
D & ~
\end{array}
\right), \quad  H(z):=\sqrt{z}\left( 
\begin{array}{ccc}
~  &X\\
X^* & ~
\end{array}
\right).  \label{18072301}
\end{align}
In the sequel, we will often omit $z$  and simply write $\mathcal{Y}\equiv \mathcal{Y}(z), \mathcal{D}\equiv\mathcal{D}(z)$ and $H\equiv H(z)$ when there is no confusion.

We denote the empirical spectral distributions (ESD) of the matrices $XX^*$ and $X^*X$ by 
\begin{equation*}
F_1(x):=\frac{1}{M} \sum_{i=1}^M \mathbf{1}_{\{\lambda_i(XX^*) \leq x\}}, \qquad  F_2(x):=\frac{1}{n} \sum_{i=1}^n \mathbf{1}_{\{\lambda_i(X^*X) \leq x\}}.
\end{equation*}
$F_1(x)$ and $F_2(x)$  are known to satisfy the Marchenko-Pastur (MP) law \cite{MP1967}. More precisely, almost surely, $F_1(x)$ converges  weakly  to a non-random limit $F_{1y}(x)$ which has a density function  given by  
\begin{equation*}
\rho_{1}(x):=
\begin{cases}
\frac{1}{2  \pi x y} \sqrt{(\lambda_{+}-x)(x-\lambda_{-})}, \ &\text{if } \lambda_- \le x \le \lambda_+,\\
0, \ &\text{otherwise},
\end{cases}
\end{equation*}
and has a point mass $1-1/y$ at the origin if $y>1$, where $\lambda_{+}=(1 + \sqrt y)^2$ and $\lambda_{-}=(1 - \sqrt y)^2$. Further, {the Stieltjes's transform of $F_{1y}$ } is given by
\begin{equation}\label{eq_mp1}
m_{1}(z):={\int \frac{1}{x-z} \, {\rm d} F_{1y}(x)}=\frac{1-y-z+\ii \sqrt{(\lambda_{+}-z)(z-\lambda_{-})}}{2zy} \quad \text{for } z \in \mathbb{C}^+,
\end{equation}
where the square root denotes the complex square root with a branch cut on the negative real axis. Similarly,  almost surely, $F_2(x)$ converges  weakly  to a non-random limit $F_{2y}(x)$ which has a  density function given by
\begin{equation*}
\rho_{2}(x):=
\begin{cases}
\frac{1}{2 \pi x} \sqrt{(\lambda_{+}-x)(x-\lambda_{-})}, \ &\text{if } \lambda_- \le x \le \lambda_+,\\
0, \ &\text{otherwise},
\end{cases}
\end{equation*}
and a point mass $1-y$ at the origin if $y<1$. The corresponding Stieltjes's transform is 
\begin{equation}\label{eq_mp2}
m_{2}(z):={ \int \frac{1}{x-z} \, {\rm d} F_{2y}(x)}=\frac{y-1-z+\ii \sqrt{(\lambda_{+}-z)(z-\lambda_{-})}}{2z}.
\end{equation}

Our estimation relies on  the local MP law \cite{PY14} and its isotropic version  \cite{BEKYY2014, KY14}, which provide sharp large deviation  estimates for the  Green functions
\begin{equation*}
G(z)=(H-z)^{-1}, \quad \mathcal{G}_1(z)=(XX^*-z)^{-1}, \qquad \mathcal{G}_2(z)=(X^*X-z)^{-1}.
\end{equation*} 
Here we recall the definition in (\ref{18072301}). By Schur complement, one can derive
\begin{equation} \label{green2}
G(z) = \left( {\begin{array}{*{20}c}
   { \mathcal{G}_1(z)} &  z^{-1/2}\mathcal{G}_1(z)X \\
   z^{-1/2}X^{*} \mathcal{G}_1(z) & \mathcal{G}_2(z) \\
\end{array}} \right).
\end{equation}
The Stieltjes transforms for the ESD of $XX^*$ and  $X^*X$ are defined by
\begin{equation}
m_{1n}(z)=\frac{1}{M} \text{Tr} \mathcal{G}_1(z)=\frac{1}{M} \sum_{i=1}^M G_{ii}(z), \quad m_{2n}(z)=\frac{1}{n} \text{Tr} \mathcal{G}_2(z)=\frac{1}{n} \sum_{\mu=M+1}^{M+n} G_{\mu \mu}(z). \label{18072401}
\end{equation} 
It is well-known that $m_{1n}(z)$ and $m_{2n}(z)$ have nonrandom approximates $m_1(z)$ and $m_2(z)$, which are the Stieltjes transforms for the MP laws defined in (\ref{eq_mp1}) and (\ref{eq_mp2}). Specifically, for any fixed $z\in \mathbb{C}^+$, the following hold,   
\begin{equation*}
m_{1n}(z)-m_{1}(z)\stackrel{a.s.}\longrightarrow 0, \quad m_{2n}(z)-m_{2}(z) \stackrel{a.s.}\longrightarrow 0.
\end{equation*}
Furthermore, one can easily check that $m_{1}(z)$ and $m_{2}(z)$ satisfy  the following self-consistent equations (see \cite{bsbook} for instance)
\begin{align}\label{eq_self1}
m_{1}(z)+\frac{1}{z-(1-y)+zym_{1}(z)}=0,
\end{align}
\begin{align} \label{eq_self2}
m_{2}(z)+\frac{1}{z+(1-y)+zm_{2}(z)}=0.
\end{align} 
We can also derive the following simple relation from the definitions 
\begin{align}\label{eq:m12c}
m_{1}(z) = \frac{y^{-1}-1}{z} + y^{-1} m_{2}(z).
\end{align}

Next we summarize some basic identities   in the following lemma without proof. They can be checked from (\ref{eq_mp1}) and (\ref{eq_mp2}) via elementary calculations.  

\begin{lem}\label{lem_m1m2quanty} Denote $p \equiv p(x)$ in \eqref{18012901}. For  any $x>y^{1/4},$  we have 
\begin{align*}
&m_{1}(p)=\frac{-1}{x^2+y}, \qquad m_{2}(p)=\frac{-1}{x^2+1}, \\
&m_{1}'(p)=\frac{x^4}{ (x^2+y)^2 (x^4-y)}, \qquad m_{2}' (p)=\frac{x^4}{(x^2+1)^2 (x^4-y)}.
\end{align*}
Furthermore, denote by $\mathcal{T}(t)=t m_{1}(t)m_{2}(t)$. We have 
\begin{equation*}
\mathcal{T}(p)=x^{-2}, \quad\quad \mathcal{T}^{\prime}(p)=(y-x^4)^{-1}.
\end{equation*}
\end{lem}
In the sequel, we also need the following notion on high probability events.
\begin{defn}[High probability event] We say that an $n$-dependent event $E\equiv E(n)$ holds with high probability if, for any large $\varphi>0,$  
\begin{equation*}
\mathbb{P}(E) \geq 1- n^{-\varphi},
\end{equation*}
for sufficiently large $n \geq n_0( \varphi).$
\end{defn} 

We also adopt the notion of {\it stochastic domination}  introduced in \cite{EKY2013}. 
 \begin{defn} [Stochastic domination] Let
 \begin{equation*}
 \mathsf{X}=(\mathsf{X}^{(n)}(u):  n \in \mathbb{N}, \ u \in \mathsf{U}^{(n)}), \   \mathsf{Y}=(\mathsf{Y}^{(n)}(u):  n \in \mathbb{N}, \ u \in \mathsf{U}^{(n)}),
 \end{equation*}
be two families of nonnegative random variables, where $\mathsf{U}^{(n)}$ is a possibly $n$-dependent parameter set. We say that $\mathsf{X}$ is stochastically dominated by $\mathsf{Y},$ uniformly in $u,$ if for all small $\epsilon$ and large $ \varphi,$ we have 
\begin{equation*}
\sup_{u \in \mathsf{U}^{(n)}} \mathbb{P} \Big( \mathsf{X}^{(n)}(u)>n^{\epsilon}\mathsf{Y}^{(n)}(u) \Big) \leq n^{- \varphi},
\end{equation*}   
for large enough $n \geq  n_0(\epsilon, \varphi).$ In addition,  we use the notation $\mathsf{X}=O_{\prec}(\mathsf{Y})$ if $|\mathsf{X}|$ is stochastically dominated by $\mathsf{Y},$ uniformly in $u.$ Throughout this paper, the stochastic domination will always be uniform in
all parameters (mostly are matrix indices and the spectral parameter $z$) that are not explicitly fixed.
 \end{defn}

\subsection{Isotropic local laws} \label{subsec:local} The key ingredient in our estimation is a special case of the anisotropic local law derived in \cite{KY14}, which is essentially the isotropic local law previously derived in \cite{BEKYY2014}.  Let $\oplus$ be the direct sum of two matrices. Set 
\begin{equation}\label{eq_piz}
\Pi_1(z):=m_{1}(z)I_M \oplus m_{2}(z) I_n.
\end{equation}

We will need the isotropic local law outside the spectrum of the MP law. For $\lambda_+=(1+y^{1/2})^2,$ define the spectral domain 
\begin{equation}\label{def:S0}
\mathbf{S}_o \equiv \mathbf{S}_o(\tau):=\{z=E+\ii\eta \in \mathbb{C}^+: \lambda_++ \tau \leq E \leq \tau^{-1}, \ 0\leq \eta \leq \tau^{-1}\},
\end{equation}
 where $\tau>0$ is a fixed small constant. Recall $m_{1n}$ and $m_{2n}$ defined in (\ref{18072401}). 

\begin{comment}

\begin{lem}Fix $\tau>0$. For any unit deterministic vectors $\mathbf{u}, \mathbf{v} \in \mathbb{R}^{p+n},$ we have
\begin{itemize}
\item (Within the limiting spectrum: \cite[Theorem 3.6]{KY14} and \cite[Theorem 3.1]{PY14})
\begin{equation*}
\langle \mathbf{u}, (G(z)-\Pi(z)) \mathbf{v} \rangle= O_{\prec}(\Psi(z)) \quad\text{and}\quad |m_2(z)- m(z) | = O_{\prec}(n^{-1} \eta^{-1})
\end{equation*}
uniformly for $z\in \mathbf{S}$.

\item (Beyond the limiting spectrum: \cite[Theorem 3.7]{KY14} and \cite[Theorem 3.12]{BEKYY2014})
\begin{equation*}
\langle \mathbf{u}, (G(z)-\Pi(z)) \mathbf{v} \rangle =O_{\prec} \left( \sqrt{\frac{\operatorname{Im} m_{2c}(z)}{n\eta}} \right)
\end{equation*}
uniformly for $z\in \mathbf{S}_o$.

\end{itemize}
\end{lem}

\end{comment}

\begin{lem}[Theorem 3.7 of \cite{KY14}, Theorem 3.12 of \cite{BEKYY2014} and Theorem 3.1 of \cite{PY14}]\label{lem_localoutside}Fix $\tau>0,$ for  any unit deterministic  $\mathbf{u}, \mathbf{v} \in \mathbb{R}^{M+n},$ we have 
\begin{align}
&\langle \mathbf{u}, (G(z)-\Pi_1(z)) \mathbf{v} \rangle =O_{\prec} \Big( \sqrt{\frac{\operatorname{Im} m_2(z)}{n\eta}} \Big), \label{eq_orginallocal}\\
&|m_{1n}(z)-m_{1}(z)|=O_{\prec}(\frac{1}{n}), \quad\quad \ |m_{2n}(z)-m_{2}(z)| =O_{\prec}(\frac{1}{n}), \label{18091401}
\end{align}
uniformly in $z \in \mathbf{S}_o$.
\end{lem}
\begin{rem} The bounds in (\ref{18091401}) cannot be directly read from  any of Theorem 3.7 of \cite{KY14}, Theorem 3.12 of \cite{BEKYY2014} or Theorem 3.1 of \cite{PY14}. In all these theorems, a weaker bound $O_\prec(\frac{1}{n\eta})$ is stated  for $z$ both inside and outside of the support of the limiting spectral distribution. Here since our parameter $z$ can be real, we use the stronger bound $\frac{1}{n}$ instead of $\frac{1}{n\eta}$. For $z\in \mathbf{S}_o$, such a bound follows from the rigidity estimates of eigenvalues in \cite{PY14} and the definition of the Stieltjes transform easily. Specifically, by (3.7) in \cite{PY14}, we know that for $a=1,2$,  $\sup_{t\in \mathbb{R}}|F_a(t)-F_{ay}(t)|\prec \frac{1}{n}$, and further by (3.6) of \cite{PY14} we know that $\sup_{t\in\mathbb{R}: |t|\geq 2+n^{-\frac23+\varepsilon}}|F_{a}(t)-F_{ay}(t)|=0$ with high probability. Then using the integration by parts to $m_{an}(z)-m_{a}(z)= \int(t-z)^{-1}{\rm d} (F_{a}(t)-F_{ay}(t))$, one can easily conclude the bounds in (\ref{18091401}).
\end{rem}

Following from Lemma \ref{lem_localoutside},  by further using Cauchy's integral formula for derivatives, we have the following uniformly in $z\in \mathbf{S}_o$, for any given $l\in \mathbb{N}$,
\begin{equation} \label{eq_consequenceoflcoal}
\langle \mathbf{u}, (G^{(l)}(z)-\Pi^{(l)}_1(z)) \mathbf{v} \rangle=O_{\prec} \Big( \sqrt{\frac{\operatorname{Im} m_2(z)}{n\eta}} \Big). 
\end{equation} 

Denote by $\kappa=|E-\lambda_+|$.  We summarize some basic estimates of $m_{1,2}(z)$ without proof. For any two  numbers $a_n$ and $b_n$ (might be $n$-dependent), we write $a_n \sim b_n$ if there exist two positive constants $C_1$ and $C_2$ (independent of $n$) such that $C_1|b_n| \leq |a_n| \leq C_2 |b_n|.$
\begin{lem}\label{lem_m1m2d} The following estimates hold uniformly in  $z\in \mathbf{S}_o$,
\begin{align} 
&|m'_{1,2}(z)| \sim |m_{1,2}(z)| \sim 1,  \label{eq_boundm1m2}\\
&\operatorname{Im} m_{1}(z) \sim\operatorname{Im} m_{2}(z) \sim
\frac{\eta}{\sqrt{\kappa+\eta}}.  \label{eq_imest}
\end{align}
\end{lem}

Given any deterministic  bounded Hermitian matrix $A$ with fixed rank, it is easy to see from Lemma \ref{lem_localoutside} and Lemma  \ref{lem_m1m2d}, the spectral decomposition and (\ref{eq_consequenceoflcoal}) that the following estimates hold uniformly in $z \in \mathbf{S}_o$:  For any fixed $k,\ell\in \mathbb{N}$, 
\begin{align}
&\max_{\mu,\nu}\Big|(G^{(l)}(z)A)_{\mu \nu}  -(\Pi^{(l)}_1(z) A)_{\mu \nu}\Big|=O_{\prec}\Big(\frac{1}{\sqrt{n}} \Big), \nonumber\\
 &\text{Tr} G^{(l)}(z)A- \text{Tr} \Pi^{(l)}_1(z) A=O_{\prec}\Big(\frac{1}{\sqrt{n}} \Big), \nonumber \\
& \max_{\mu,\nu}\Big|(G^{(k)}(z)AG^{(l)}(z))_{\mu \nu}-(\Pi^{(k)}_1 A \Pi^{(l)}_1)_{\mu \nu}\Big|=O_{\prec} (\frac{1}{\sqrt{n}}). \label{eq_localoutside}
\end{align}

In our proof, we will rely on the estimates of powers of $G,$ i.e $G^l, l=2,3,4.$ We have the following lemma whose proof is stated in \cite{BDW_suppl}.  
\begin{lem}\label{lem.18091410} We have the following recursive relation
\begin{equation} \label{18072530}
G^2=2G'+\frac{G}{z}, \ G^3=(G^2)'+\frac{G^2}{z}, \ G^4=\frac{2}{3} (G^3)'+\frac{G^3}{z}.
\end{equation}
\end{lem}

 Recall $\Pi_1$ defined in (\ref{eq_piz}) and further define
\begin{align}
\Pi_2:=2\Pi_1'+\frac{1}{z} \Pi_1, \qquad \Pi_3:= \Pi_2'+\frac{1}{z}\Pi_2, \qquad \Pi_4:=\frac{2}{3} \Pi_3'+\frac{1}{z}\Pi_3.  \label{18072531}
\end{align}
With Lemma \ref{lem.18091410}, similarly to (\ref{eq_orginallocal}) and (\ref{eq_consequenceoflcoal}), we can get the following estimates for $l=1, 2,3,4,$
\begin{equation}\label{eq_genrealcontrol}
\langle \mathbf{u}, (G^l-\Pi_l)\mathbf{v} \rangle=O_{\prec}\Big(\frac{1}{\sqrt{n}} \Big),
\end{equation}
uniformly in $z\in \mathbf{S}_o$. For brevity, in the sequel, we will use the notation
\begin{align}
\Xi_l\equiv \Xi_l(z):= G^l(z)-\Pi_l(z), \qquad l\in \mathbb{N}.  \label{centered Green function}
\end{align}

\subsection{Green function representation}\label{subsec:key}

In this section, we represent (asymptotically)   $|\langle \widehat{\mathbf{v}}_i, \mathbf{v}_i \rangle |^2$'s and   $R$  (c.f \eqref{18012401}) in terms of the Green function. The derivation relies  on the results obtained in \cite{Ding}. Recall $p(d)$ in \eqref{18012901} and $a(d)$ in (\ref{eq_defna1a2}). For $i\in [r],$ define 
\begin{equation}\label{defn_h}
h_i(x)=\frac{x^4 p^{\prime}(x) p(x)}{(x+d_i)^2},
\end{equation}
and we use the shorthand notation $ \bar{i}=i+r.$
To state results for the right singular vectors,  we introduce a $2r \times 2r$ matrix function $W_i(x)$ for $x>0$, which has  only four non-zero entries given by
\begin{align}\label{eq_definw}
&\big( W_i(x)\big)_{ii}=m^2_{2}(x),\qquad \big(W_i (x) \big)_{\bar{i} \bar{i}}=\frac{1}{d_i^2x}, \nonumber\\
&\big(W_i(x) \big)_{i \bar{i}}=\big(W_i (x)\big)_{\bar{i}i}=- \frac{m_{2}(x)}{d_i\sqrt{x}} .
\end{align}  
We further denote the matrix function
\begin{align} 
M_{i}(x)=\mathcal{U}W_i(x)\mathcal{U}^*. \label{18072540}
\end{align}
With the above notations, we further introduce two  $(M+n) \times (M+n)$ matrices  
\begin{align} 
&A_{i}^R=-d_i^2\Big( h^{\prime}_i(d_i)M_{i}(p_i)+ h_i(d_i)p^{\prime}(d_i) M_{i}' (p_i) \Big),\nonumber \\
 &B_i^R= -d_i^2 h(d_i) p^{\prime}(d_i) M_{i}(p_i). \label{eq_ab}
\end{align}  
In light of the definition of $\mathcal{U}$ in (\ref{18072301}), we have 
\begin{align}\label{eq:new-ab}
A_i^R = \begin{pmatrix} \omega_{i1} \mathbf u_i \mathbf u_i^T & \omega_{i2} \mathbf u_i \mathbf v_i^T \\ \omega_{i3} \mathbf v_i \mathbf u_i^T & \omega_{i4} \mathbf v_i \mathbf v_i^T \end{pmatrix} ,
\quad \quad 
B_i^R = \begin{pmatrix} \varpi_{i1} \mathbf u_i \mathbf u_i^T & \varpi_{i2} \mathbf u_i \mathbf v_i^T \\ \varpi_{i3} \mathbf v_i \mathbf u_i^T & \varpi_{i4} \mathbf v_i \mathbf v_i^T \end{pmatrix}.
\end{align}
Here we used the notations
\begin{align*}
&\omega_{i1}:= -d_i^2 \big( h_i'(d_i) (W_i(p_i))_{ii} + h_i (d_i) p'(d_i) (W_i'(p_i))_{ii}\big), \\
& \omega_{i4}:= -d_i^2 \big( h_i'(d_i) (W_i(p_i))_{\bar i \bar i} + h_i (d_i) p'(d_i) (W_i'(p_i))_{\bar i \bar i}\big),\\
&\omega_{i2}=\omega_{i3}:= -d_i^2 \big( h_i'(d_i) (W_i(p_i))_{i\bar i} + h_i (d_i) p'(d_i) (W_i'(p_i))_{i \bar i}\big), \\
&\varpi_{i1}:= -d_i^2 h_i (d_i) p'(d_i) (W_i(p_i))_{ii}, \nonumber\\
& \varpi_{i4}:= -d_i^2 h_i (d_i) p'(d_i) (W_i(p_i))_{\bar i \bar i}, \nonumber\\
&\varpi_{i2}= \varpi_{i3}:= -d_i^2 h_i (d_i) p'(d_i) (W_i(p_i))_{i \bar i}.
\end{align*}

Recall the notation introduced in (\ref{centered Green function}). We have the following lemma whose proof is stated in \cite{BDW_suppl}. 
\begin{lem}\label{lem:key}Under assumptions of (\ref{eq_meanva}), (\ref{eq_moment}), (\ref{eq_ratioassumption}) and Assumption \ref{assum_main},  we have
\begin{equation*}
|\langle \mathbf{v}_i, \widehat{\mathbf{v}}_i \rangle|^2=a(d_i)+\operatorname{Tr} \big(\Xi_1(p_i) A_i^R \big)+\operatorname{Tr} \big(\Xi_1^{\prime}(p_i) B_i^R \big)+O_{\prec}(\frac{1}{n}),
\end{equation*}
Furthermore, we have 
\begin{align}
&R=\sum_{i=1}^r a(d_i)+\sum_{i=1}^r \left( \Tr  \big(\Xi_1(p_i) A_i^R\big)+\Tr \big(\Xi_1'(p_i)B_i^R\big) \right) +O_{\prec}(\frac{1}{n}). 
\label{eq_lv} 
\end{align}
\end{lem}

\subsection{Auxiliary definitions}\label{subsec:aux}

It is convenient to introduce the following notion of convergence in distribution.
\begin{defn}[{ \cite[Definition 7.3]{KY13}}] \label{defn_asymptotic}
Two sequences of random variables, $\{\mathsf X_n\}$ and $\{\mathsf Y_n\}$, are \emph{asymptotically equal in distribution}, denoted as $\mathsf X_n \simeq \mathsf Y_n,$ if they are tight and satisfy
\begin{equation*}
\lim_{n \rightarrow \infty} \big( \mathbb{E}f(\mathsf X_n)-\mathbb{E}f(\mathsf Y_n) \big)=0
\end{equation*}   
for any bounded continuous function $f$. 
\end{defn}

We also collect some basic results on convergence and equivalence in distribution in the supplementary material \cite{BDW_suppl}, Lemma \ref{lem_convergenindist}. 

The following notation from \cite[Definition 7.11]{KY13} will be convenient for us when we  replace random variables with their i.i.d copies. 
\begin{defn}\label{defn:replace} Let $\{\sigma_n\}$ be a sequence of bounded positive numbers. If $\mathsf{X}_n$ and $\mathsf{Y}_n$ are independent random variables with $\mathsf{Y}_n \simeq \mathcal{N}(0,\sigma_n^2),$ and if $\mathsf{S}_n \simeq \mathsf{X}_n+\mathsf{Y}_n,$  we write
$
\mathsf{S}_n \simeq \mathsf{X}_n+\mathcal{N}(0, \sigma_n^2).
$
\end{defn}

\section{Proof of  Theorems \ref{thm_mainthm}} \label{sec_mainthm}  

 For brevity, in this section, we omit the subindices of $d_i, \mathbf{u}_i, \mathbf{v}_i, \widehat{\mathbf{u}}_i, \widehat{\mathbf{v}}_i$ and write $d, \mathbf{u}, \mathbf{v}, \widehat{\mathbf{u}}, \widehat{\mathbf{v}}$ instead. Similarly, we write the matrices  $A_i^R$ and $B_i^R$ (c.f. \eqref{eq_ab}) as $A$ and $B$, respectively. We also write $m_{1,2}(z)$ as $m_{1,2}$ for brevity.

By Lemma \ref{lem:key},  we can reduce the problem to study 
\begin{align}\label{defn_q1}
\mathcal Q \equiv \mathcal Q(z):=\sqrt{n} \Big( \Tr \big(\Xi_1(z)A\big)+ \Tr \big(\Xi'_1(z)B\big) \Big),
\end{align}
at $z=p(d)$ (c.f.\eqref{18012901}).

 In the sequel, we will  prove the limiting distribution of $ \mathcal{Q}(z)$ at $z=p(d)$. The key task is to prove Proposition \ref{lem:main} below. In this section, we will show that Theorem \ref{thm_mainthm} follows from Proposition \ref{lem:main}. Let index $i\in [M]$ and $j\in [n]$. Denote 
 the shorthand notation 
\begin{align}
j'=j+M. \label{18091450}
\end{align}
For short, we also write  
 $ \sum_{i,j} = \sum_{i=1}^M \sum_{j=1}^n.$
 
 In order to state Proposition \ref{lem:main}, we first introduce some notations.  For a fixed small constant $\nu>0$, denote by
\begin{equation*}
\mathcal{B}(\nu):=\Big\{ (i,j)\in [M]\times[n] : |\mathbf{u}(i)|>n^{-\nu}, \ |\mathbf{v}(j)|>n^{-\nu} \Big\},
\end{equation*}
the set of the indices of those compoents with large magnitude. Since $\mathbf{u}$ and $\mathbf{v}$ are unit vectors, we have  $|\mathcal{B}(\nu)| \leq Cn^{4\nu}$ for some constant $C>0.$ Let $\mathcal{S}(\nu)$ be the complement of $\mathcal{B}(\nu)$, i.e.,
\begin{align}
\mathcal{S}(\nu)=([M]\times[n] )\setminus \mathcal{B}(\nu). \label{18072542}
\end{align}
For brevity, we introduce the notation 
\begin{equation}\label{eq_permu} 
 \mathcal{P}(\alpha_1,\ldots, \alpha_m),
 \end{equation}
to represent  the set of all the permutations of $(\alpha_1,\ldots, \alpha_m)$, where $\alpha_i$'s  can be alike.  Recall (\ref{eq_piz}) and (\ref{18072531}). We set the deterministic quantity 
\begin{align} \label{defn_detministic1}
\Delta_d \equiv \Delta_d(z):= & -\frac{\kappa_3 z^{3/2}}{n} \sum_{i,j} \Big(  (\Pi_1)_{ii} (\Pi_1)_{j' j'} \big( 2(\Pi_1 A \Pi_1)_{i j'}  +  (\Pi_1 B \Pi_1')_{i j'} + (\Pi_1' B \Pi_1)_{i j'} \big) \nonumber \\
&\quad + \frac{1}{2} \sum_{(a_1,a_2,a_3) \in \mathcal{P}(2,1,1)} (\Pi_{a_1})_{ii} (\Pi_{a_2})_{j' j'} \big( (\Pi_1 B \Pi_{a_3})_{i j'} +  (\Pi_{a_3} B \Pi_1)_{i j'}\big) \Big),
\end{align}
and the random variable 
\begin{align}
\Delta_r\equiv \Delta_r(z):= \sqrt{nz}\sum_{(i,j)\in \mathcal{B}(\nu) } x_{ij} c_{ij}, \label{18072410}
\end{align}
 where 
 \begin{align}\label{defn_cij1}
c_{ij}\equiv c_{ij}(z):= &-\sum_{\substack{ l_1,l_2\in \{ i, j'\} \\ l_1\neq l_2 }}\Big( (\Pi_1 A \Pi_1)_{l_1 l_2} -\frac{1}{2z} (\Pi_1 B \Pi_1)_{l_1 l_2}  \nonumber\\
&\qquad\qquad+ \frac{1}{2} (\Pi_{1} B \Pi_{2})_{l_1 l_2} +\frac{1}{2} (\Pi_{2} B \Pi_{1})_{l_1 l_2} \Big).
\end{align}
 Define the $M \times n $ matrix function $S\equiv S(z)=(s_{ij})$ with
\begin{align} \label{defn_sij1}
s_{ij}\equiv s_{ij}(z)&:=\sum_{\substack{ l_1,\cdots,l_4\in \{ i, j'\}  \\ l_1\neq l_4, l_2 \neq l_3 }}  \Big( (\Pi_1 A \Pi_1)_{l_1 l_2} (\Pi_1)_{l_3 l_4} -\frac{1}{2z} (\Pi_1 B \Pi_1)_{l_1 l_2} (\Pi_1)_{l_3 l_4} \nonumber\\
& \qquad\qquad\qquad\qquad + \frac{1}{2} \sum_{(a_1,a_2,a_3) \in \mathcal{P}(2,1,1)} (\Pi_{a_1} B \Pi_{a_2})_{l_1 l_2} (\Pi_{a_3})_{l_3 l_4} \Big).
\end{align}

Further, we define the function 
\begin{align}\label{def:variance}
V \equiv V(z):=\mathcal V^E(z)+ 2\frac{ \kappa_3 z^{\frac32}}{\sqrt{n}} \sum_{(i,j) \in \mathcal{S}(\nu)} c_{ij} s_{ij} + \frac{\kappa_4 z^{2} }{n} \sum_{ i,j} s_{ij}^2+z \sum_{ (i,j)\in \mathcal{S}(\nu) } c_{ij}^2,
\end{align}
where 
\begin{align}\label{eq_indivialve}
\mathcal V^E \equiv \mathcal V^E(z) := -\sqrt{z} \sum_{\alpha=1,2} \big( m_\alpha \mathfrak{a}_{1\alpha}+  \frac{m_\alpha}{2} \tilde{\mathfrak{b}}_{1\alpha}+m_\alpha' \mathfrak{b}_{1\alpha}\big). 
\end{align}
Here we refer to \eqref{def:missvar} in \cite{BDW_suppl}  for the definitions of $\mathfrak{a}_{1\alpha}$, $\mathfrak{b}_{1\alpha}$ and $ \tilde{\mathfrak{b}}_{1\alpha}$ for $\alpha=1,2$.

With $\Delta_d$ and $\Delta_r$ defined in (\ref{defn_detministic1}) and (\ref{18072410}), we  introduce the notation 
\begin{align}
\Delta\equiv \Delta(z):= \Delta_r(z)+\Delta_d(z) \label{18072502}
\end{align}
and define 
\begin{align}
Q\equiv Q(z):= \mathcal{Q}(z)-\Delta(z). \label{18072501}
\end{align}
\begin{pro}\label{lem:main}
 Under the assumptions of Theorem \ref{thm_mainthm}, we have that $Q(p_i)$ and $\Delta(p_i)$  are asymptotically independent. Furthermore, 
\begin{equation}\label{eq_qnormal}
Q(p_i) \simeq \mathcal{N}(0, V(p_i)).
\end{equation}
\end{pro}
We first show how Proposition \ref{lem:main} implies Theorem  \ref{thm_mainthm}. 
\begin{proof}[Proof of Theorem \ref{thm_mainthm}]
By Lemma \ref{lem:key} and  \eqref{defn_q1}, 
\begin{align*}
\sqrt{n} \big(|\langle \mathbf v_i, \widehat{\mathbf v}_i \rangle |^2 -a(d_i) \big) = \mathcal{Q}(p_i) + O_{\prec}(n^{-\frac12}).
\end{align*}
Here $\mathcal{Q}(p_i)$ is defined in \eqref{defn_q1} with  $(A,B)=(A^R_i, B^R_i)$  (c.f.(\ref{eq_ab})).  By Proposition \ref{lem:main}, we have that at $z=p_i$,
\begin{align*}
\mathcal Q &=  \Delta_d+ \Delta_r + Q\simeq \Delta_d + \sqrt{nz}\sum_{(i,j)\in \mathcal{B}(\nu) } x_{ij} c_{ij} + \mathcal N\big(0, V \big).
\end{align*}
Next, by Central Limit Theorem  and Lemma \ref{lem_convergenindist} in \cite{BDW_suppl}, one has
\begin{align*}
\sqrt{nz}\sum_{i,j } x_{ij} c_{ij} \simeq \sqrt{nz}\sum_{(i,j)\in \mathcal{B}(\nu) } x_{ij} c_{ij} + \mathcal N \big(0, z\sum_{(i,j) \in \mathcal{S}(\nu)} (c_{ij})^2 \big).
\end{align*}
Furthermore, by the definition of $\mathcal{S}(\nu)$, we notice that 
$$n^{-1/2}\sum_{(i,j)\in \mathcal S(\nu)} c_{ij} s_{ij} = n^{-1/2}\sum_{i,j} c_{ij} s_{ij} + O(n^{-\frac12+4\nu}).$$
Let  $C(z)=(c_{ij}(z))$ with $c_{ij}(z)$ defined  in \eqref{defn_cij1} and recall $S(z)$ from \eqref{defn_sij1}. Using Lemma \ref{lem_convergenindist} in \cite{BDW_suppl}, we conclude that
\begin{align*}
\mathcal Q(p_i) \simeq \Delta_d(p_i) + \sqrt{n p_i }\Tr (X^* C(p_i)) + \mathcal N(0,\mathcal V(p_i)),
\end{align*}
where
\begin{align*}
\mathcal V(p_i) = \mathcal V^E(p_i) + 2\frac{ \kappa_3 {p_i}^{3/2}}{\sqrt{n}} \Tr\big( C(p_i)^*S(p_i)\big) + \frac{\kappa_4 {p_i}^{2} }{n} \Tr \big( S(p_i)^* S(p_i) \big).
\end{align*}

Denote
\begin{align*}
\Delta_i =  \sqrt{n p_i }\Tr \big(X^* C(p_i)\big) +  \Delta_d(p_i)
\end{align*}
and $ \mathcal Z_i \sim \mathcal N(0,\mathcal V(p_i)),$ which is independent of $\Delta_i$. 
Next, plugging $z=p_i$ into \eqref{defn_detministic1}, \eqref{defn_cij1}, \eqref{defn_sij1}, using Lemma \ref{lem_m1m2quanty} and taking into account the definitions of 
$A^R_i, B^R_i$ in (\ref{eq_ab}), we find that
\begin{align*}
\Delta_i =- \sqrt{n} \frac{2(d_i^4 + 2y d_i^2 + y)}{d_i^3 (d_i^2 + 1)^2} \mathbf u_i^* X \mathbf v_i -\frac{2( d_i^6 -  3 y d_i^2 - 2 y) }{ d_i^5 (d_i^2 +1)^2} \Big(\frac{\kappa_3}{n} \sum_{k,l} \mathbf u_i(k)\mathbf v_i(l) \Big). 
\end{align*}
The variance $\mathcal V(p_i)$ is the sum of 
\begin{align*}
& 2 \frac{\kappa_3}{\sqrt n} p_i^{3/2}\Tr \big( C(p_i)^* S(p_i)) + \frac{\kappa_4}{n}  p_i^2 \Tr \big( S(p_i)^* S(p_i) \big) \\
 &= -\frac{ 4(d_i^4 + 2 y d_i^2 + y)(d_i^6 - 3y d_i^2 - 2y)}{d_i^7 (d_i^2+1)^4}\Big(  \frac{\kappa_3}{\sqrt n}  \sum_{k,l} \mathbf u_i(k)^3 \mathbf v_i(l)\Big)\\
 &\quad +\frac{ 4 (d_i^4 + 2 y d_i^2 +y )^2}{d_i^7 (d_i^2+1)^4}\Big( \frac{\kappa_3}{\sqrt n} \sum_{k,l} \mathbf u_i(k) \mathbf v_i(l)^3 \Big)\\
 &\quad+ \frac{(d_i^6 - 3y d_i^2 - 2y)^2 }{d_i^8 (d_i^2+1)^4}\Big(\kappa_4 \sum_k \mathbf u_i(k)^4 \Big)+\frac{(d_i^4 + 2y d_i^2 + y)^2}{d_i^8 (d_i^2+1)^4} \Big( {\kappa_4}{y_n}\sum_l \mathbf v_i(l)^4 \Big)
\end{align*}
and
\begin{align*}
\mathcal V^E(p_i)=\frac{2}{d_i^4 - y} \Big( & 2 y (y +1) \big(\frac{d^4 + 2 y d^2 + y}{d^3 (d^2+1)^2} \big)^2 - \frac{y(y-1)(5y+1)}{d_i (d_i^2+1)^2}\big(\frac{d^4 + 2 y d^2 + y}{d^3 (d^2+1)^2} \big) \\
&+  \frac{(d_i^4+y)(d_i^2 + y)^2 }{d_i^3 (d_i^2+1)^2}\big( \frac{d^6 - 3 y d^2 - 2 y}{d^3 (d^2+1)^2}\big)+ \frac{2y^{2} (y-1)^2}{d_i^2 (d_i^2+1)^4} \Big).
\end{align*}
The last expression is obtained by using the definitions of $\mathfrak{a}_{1\alpha}$, $\mathfrak{b}_{1\alpha}$ and $ \tilde{\mathfrak{b}}_{1\alpha}$ for $\alpha=1,2$  in \eqref{def:missvar} of \cite{BDW_suppl} and performing tedious yet elementary calculations.
Recall (\ref{eq_defnthed}). 
The conclusion of Theorem \ref{thm_mainthm} follows immediately by rewriting $\Delta_i$ and $\mathcal V(p_i)$ in terms of $\theta(d_i)$ and $\psi(d_i)$.
\end{proof}

The rest of this section is devoted to the proof of Proposition \ref{lem:main}.   Our proof relies on the cumulant expansion in Lemma \ref{lem_cumu} of \cite{BDW_suppl}, where we need to control the expectation.
 Throughout the proof, we will frequently use the estimates in (\ref{eq_localoutside}). These estimates hold with high probability, which do not yield bounds for the expectations directly. In order to translate the high probability bounds into those for the expectations, one needs a crude deterministic bound for the Green function on the bad event with tiny probability.   To this end,  we will work with a slight modification of the real $z=p(d)$ for Green function. Specifically, in the  proof of the following Proposition \ref{prop_iteration},  we will also use the parameter
 \begin{equation}\label{eq_zdefn} 
 z=p(d)+\ii n^{-C},
 \end{equation}
  for a large constant $C$. On the bad event, we will use the naive bound of the Green function $\|G\|\leq N^C$, which will be compensated by the tiny probability of the bad event. At the end, by the continuity of $G(\tilde{z})$ at $\tilde{z}$ away from the support of the MP law, it is (asymptotically) equivalent to work with (\ref{eq_zdefn}),  for the proof of Proposition \ref{lem:main}.  We first claim that it suffices to establish  the following recursive estimate.  
\begin{pro}\label{prop_iteration} Suppose that the assumptions of Theorem \ref{thm_mainthm} hold. Let $z_0=p(d)$ and $z$ be defined in (\ref{eq_zdefn}). We have
\begin{align}
\E Q(z) e^{\ii t \Delta(z_0)} = O_{\prec}(n^{-\frac12+4\nu}), \label{18092701}
\end{align}
and for any fixed integer $k \ge 2$, 
\begin{equation}\label{eq_qdeltarecusie}
\mathbb{E} Q^k(z) e^{\ii t \Delta(z_0)}=(k-1) V \mathbb{E} Q^{k-2}(z) e^{\ii t \Delta(z_0)}+O_{\prec}(n^{-\frac12+4\nu}).
\end{equation}
\end{pro}

The proof of Proposition \ref{prop_iteration} is our main technical task, which will be stated in Section \ref{s. proof of recursive estimate} of \cite{BDW_suppl}.
Now we first show the proof of Proposition \ref{lem:main} based on Proposition \ref{prop_iteration}. 

\begin{proof}[Proof of Proposition \ref{lem:main}] { Recall the following elementary  bound,} for any $x \in \mathbb{R}$ and  sufficiently large $N\in \mathbb{N}$,  we have 
\begin{equation}\label{eq_charbound}
\bigg|e^{\ii x}-\sum_{k=0}^N \frac{(\ii x)^k}{k!} \bigg| \leq \min \left\{\frac{|x|^{N+1}}{(N+1)!}, \frac{2|x|^N}{N!}\right\}.
\end{equation}
First, we write 
$
Q(z)=Q_R(z)+\ii Q_I(z),
$
where $Q_R(z)$ and $Q_I(z)$ stand for the real and imaginary parts of $Q(z)$ respectively. According to the choice of $z$ in (\ref{eq_zdefn}), we have the deterministic bound $|Q_I(z)|\leq N^C$  for some large positive constant $C$. Moreover, by continuity of the Green function and the Stieltjes transform, one can easily check that $|Q_I(z)|\leq N^{-C'}, $ for some large positive constant $C'$ with high probability. Using the small bound $N^{-C'}$ on the high probability event and the large deterministic bound  $ N^C$ on the tiny probability event, one can easily derive from (\ref{18092701}) and (\ref{eq_qdeltarecusie}) that 
\begin{align}
&\E Q_R(z) e^{\ii t \Delta(z_0)} =O_{\prec}(n^{-\frac12+4\nu}), \label{18092701R}\\
&\mathbb{E} Q_R^k(z) e^{\ii t \Delta(z_0)}=(k-1) V \mathbb{E} Q_R^{k-2}(z) e^{\ii t \Delta(z_0)}+O_{\prec}(n^{-\frac12+4\nu}) \label{eq_qdeltarecusieR}.
\end{align}
For any $s,t\in \R$, by (\ref{eq_charbound}),  we have
\begin{align}\label{eq_edetla1}
\mathbb{E}e^{ \ii s Q_R(z)+ \ii t\Delta(z_0)}=\sum_{k=0}^{2N-1} \frac{ (\ii s)^k }{k!} \E Q_R^k(z) e^{\ii t \Delta(z_0)} + O\left( \frac{s^{2N}}{(2N)!} \E Q^{2N}_{R}(z) \right).
\end{align}

For the error term on the right side of (\ref{eq_edetla1}), using (\ref{eq_qdeltarecusieR}) recursively for $t=0$, we first find  
\begin{equation*}
\E Q^{2N}_R(z) = (2N-1)!! V^N +O_{\prec}(n^{-\frac12+4\nu}). 
\end{equation*}
Thus, for arbitrarily small $\epsilon>0,$  by taking $N$ sufficiently large, we have
$
 \frac{(2N-1)!! V^N }{(2N)!}<\epsilon
$
and it follows that 
\begin{align}\label{eq:powerseries}
\left| \mathbb{E}e^{ \ii s Q_R(z)+ \ii t\Delta(z_0)} - \sum_{k=0}^{2N-1} \frac{ (\ii s)^k }{k!} \E Q^k_R(z) e^{\ii t \Delta(z_0)} \right| <\epsilon + O_{\prec}(n^{-\frac12+4\nu}). 
\end{align}

Using (\ref{eq_qdeltarecusieR}), we get the following estimate 
\begin{equation}\label{eq_con1}
\sum_{k=0}^{2N-1} \frac{ (\ii s)^k }{k!} \E Q^k_R(z) e^{\ii t \Delta(z_0)}= \sum_{k=0}^{N-1} \frac{(\ii s)^{2k} }{ (2k)!!} V^k \E e^{\ii t\Delta(z_0)}+O_{\prec}(n^{-\frac12+4\nu}). 
\end{equation}
Next, combing \eqref{eq_con1} with the fact 
\begin{equation*}
\exp(\frac{x^2}{2})=\sum_{k=0}^{\infty}  \frac{x^{2k}}{(2k)!!},
\end{equation*}
together with \eqref{eq:powerseries}, we conclude that 
\begin{equation}\label{eq_keyeq}
\left| \mathbb{E} e^{ \ii s Q_R(z)+ \ii t\Delta(z_0)}-e^{-\frac{1}{2} V s^2} \mathbb{E} e^{\ii t \Delta(z_0)}\right| < 2\epsilon + O_{\prec}(n^{-\frac12+4\nu}).
\end{equation}
The asymptotic independence of $Q_R(z)$ and $\Delta(z_0)$ is a consequence of (\ref{eq_keyeq}) and the fact $\epsilon$ is arbitrarily small. (\ref{eq_qnormal}) can be proved by setting $s=0.$   Although Proposition \ref{prop_iteration} is proved under the choice (\ref{eq_zdefn}), by continuity of $G$ outside of the support of MP law, we know $Q(z_0)=Q_R(z)+O(N^{-C'})$ with high probability for some positive constant $C'$. 
This concludes the proof of Proposition \ref{lem:main}.
\end{proof}

\noindent{\bf Acknowledgements.} X.C. Ding would like to thank Wei Q. Deng for many helpful discussions on the applications in statistical genetics. The authors also would like to thank two anonymous referees, the associated editor and editor for their suggestions and comments, which have significantly improved the paper. The authors also want to thank Jiang Hu for many helpful discussions.

%%%%%%%%%% Merge with supplemental materials %%%%%%%%%%
\pagebreak
%\widetext
\begin{center}
\textbf{\Large Supplementary material to ``Singular vector and singular subspace distribution for the matrix denoising model"}
\end{center}
\medskip

This file contains detailed simulation results, further discussions on statistical applications,  auxiliary lemmas,  the proofs of Theorem \ref{thm. right subspace}  and some technical lemmas of the paper \cite{BDW}. 
%%%%%%%%%% Merge with supplemental materials %%%%%%%%%%
%%%%%%%%%% Prefix a "S" to all equations, figures, tables and reset the counter %%%%%%%%%%
\setcounter{equation}{0}
\setcounter{figure}{0}
\setcounter{table}{0}
\setcounter{page}{1}
\makeatletter
\renewcommand{\theequation}{S\arabic{equation}}
\renewcommand{\thefigure}{S\arabic{figure}}
\renewcommand{\bibnumfmt}[1]{[S#1]}
\renewcommand{\citenumfont}[1]{S#1}
\renewcommand{\thesection}{\Alph{section}}
\setcounter{section}{0}
%%%%%%%%%% Prefix a "S" to all equations, figures, tables and reset the counter %%%%%%%%%%

\section{Detailed simulation results}
In this section, we state detailed simulation results for Section \ref{s.numerical} of \cite{BDW}. 

\begin{figure}[ht]
\begin{subfigure}{0.4\textwidth}
\includegraphics[width=6cm,height=4cm]{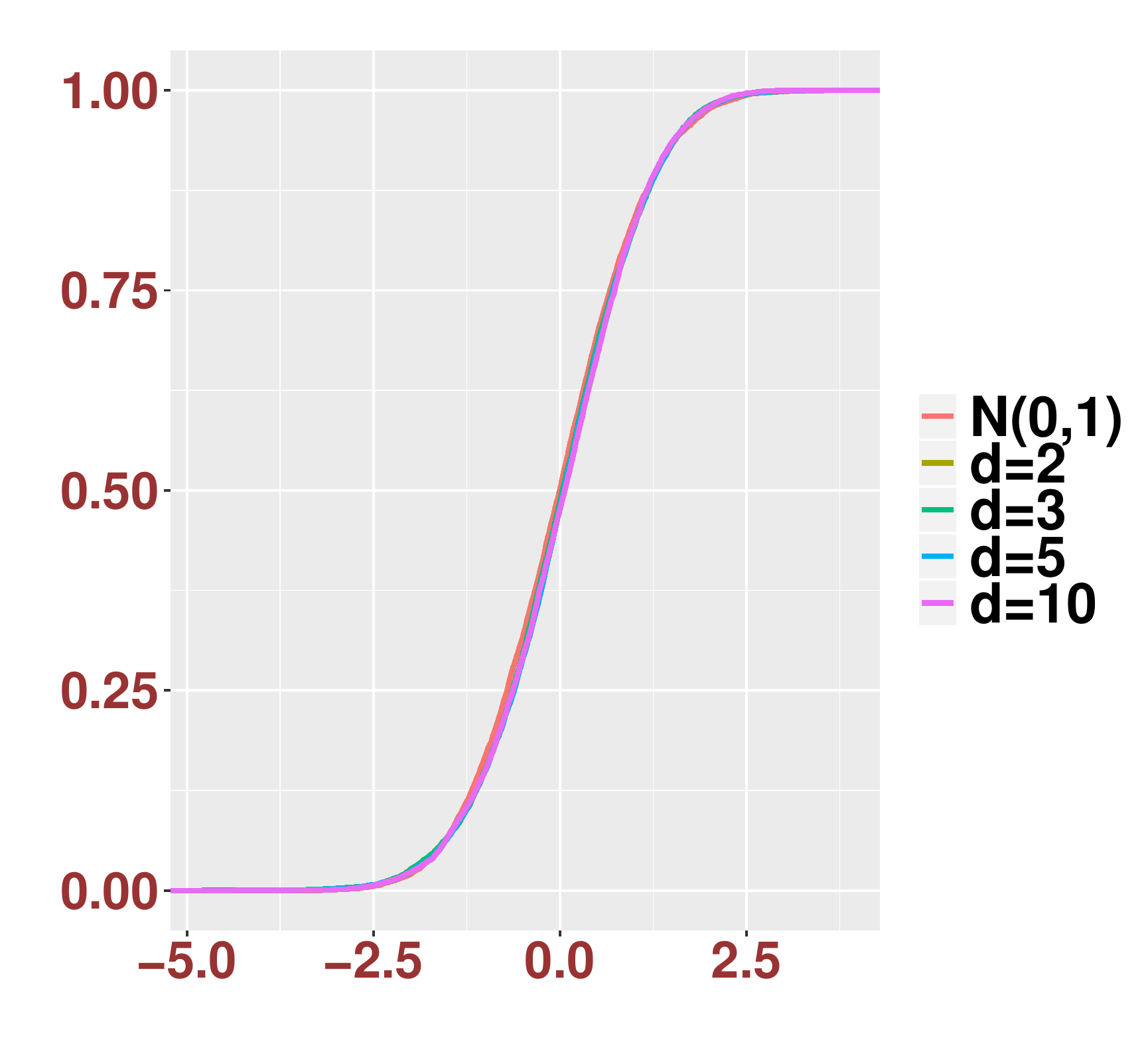}
\subcaption{ECDF of $\mathcal R_{g}$ with Gaussian noise.\\~~\\~~}
\end{subfigure}
\hspace{1cm}
\begin{subfigure}{0.4\textwidth}
\includegraphics[width=6cm,height=4cm]{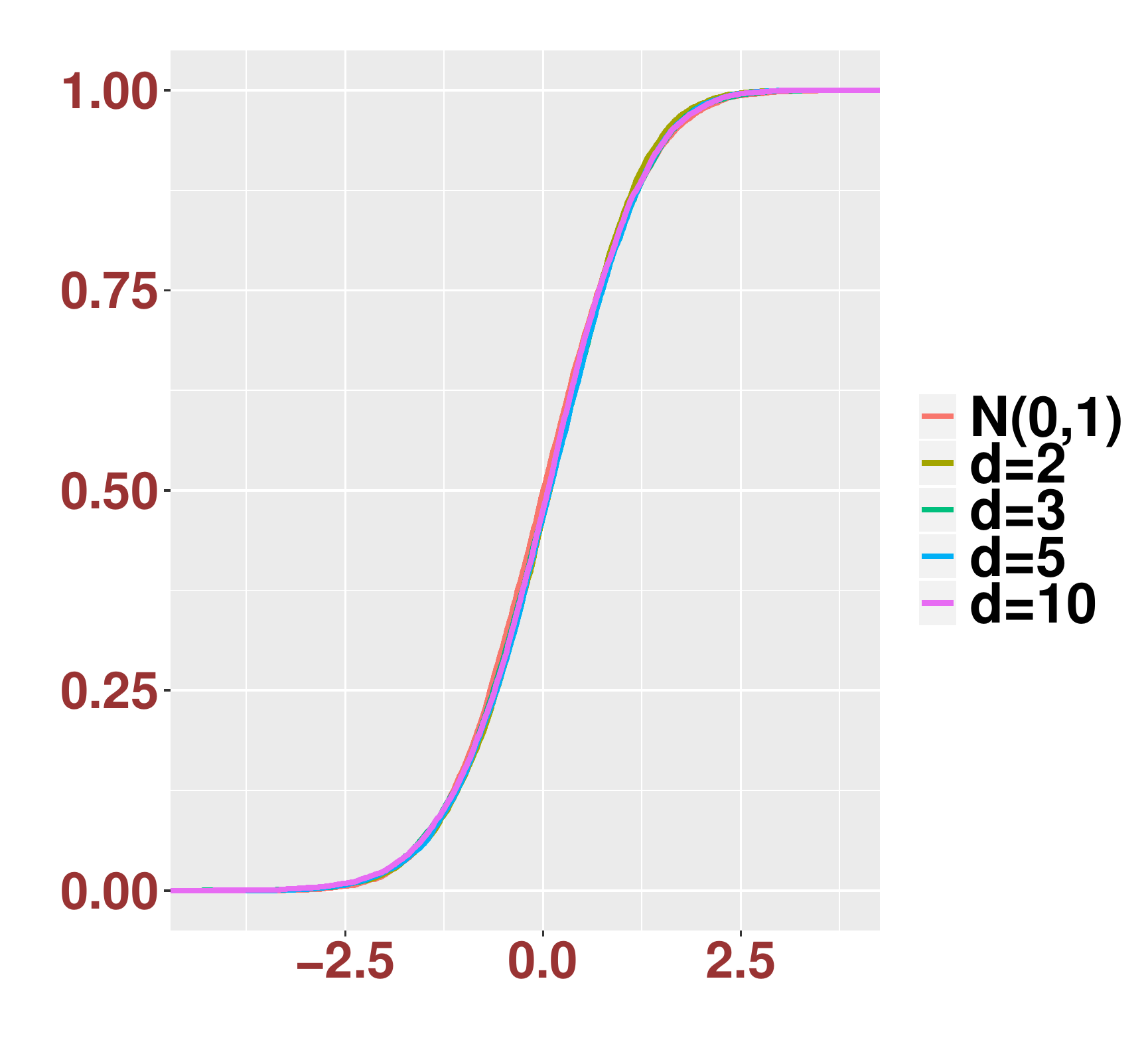}
\subcaption{ECDF of $\mathcal{R}_{dt}$ with Two-Point noise and delocalized singular vectors. }
\end{subfigure}
\begin{subfigure}{0.4\textwidth}
\includegraphics[width=6cm,height=4cm]{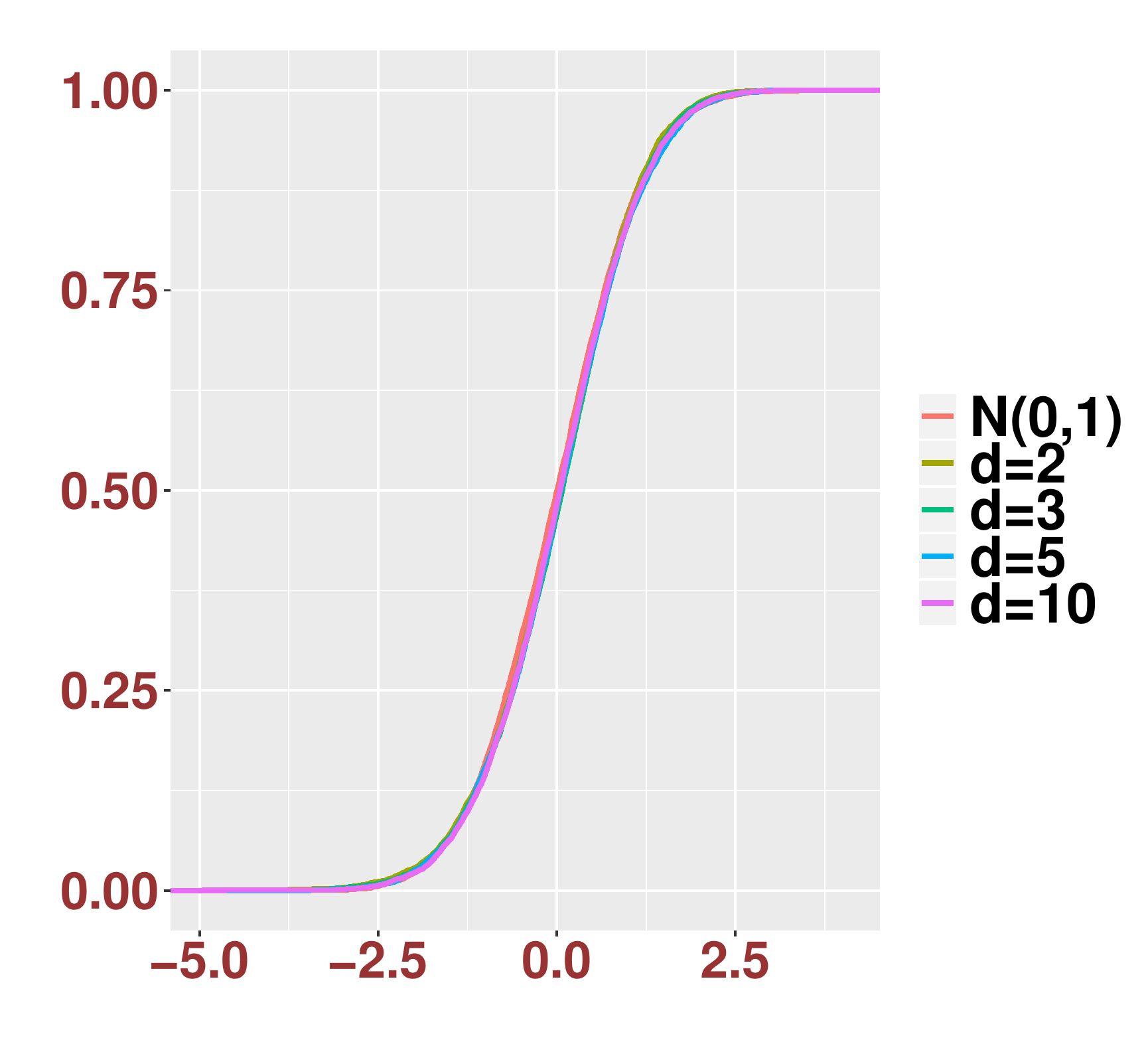}
\subcaption{ECDF of $\mathcal{R}_{pt}$ with Two-Point noise and delocalized left singular vector and sparse right singular vector.}
\end{subfigure}
\hspace{1cm}
\begin{subfigure}{0.4\textwidth}
\includegraphics[width=6cm,height=4cm]{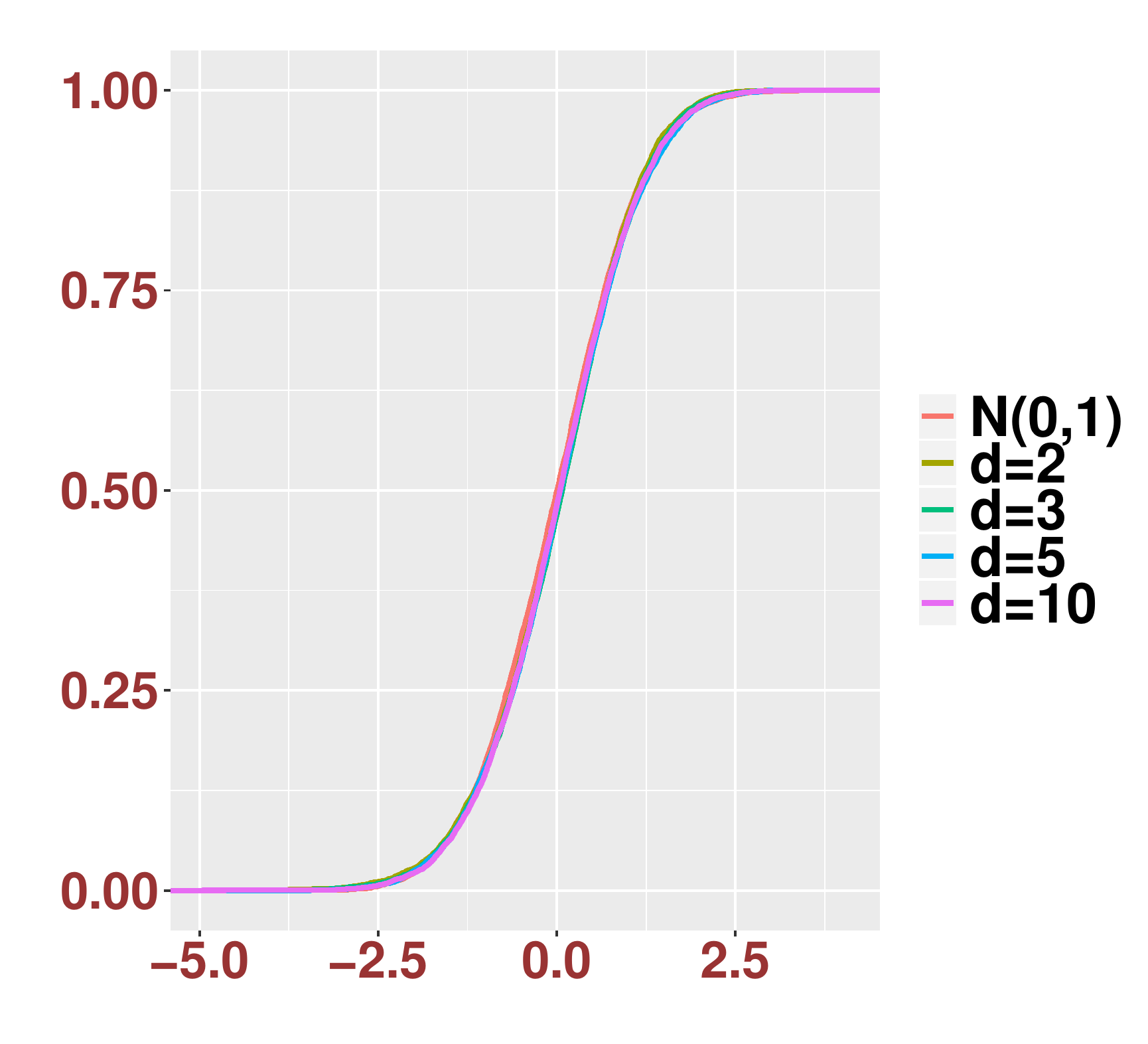}
\subcaption{ECDF of $\mathcal R_{st}$ with Two-Point noise and sparse singular vectors.\\~~\\~~}
\end{subfigure}
\caption{Plots of the ECDFs of $\mathcal{R}_g, \mathcal{R}_{dt}, \mathcal{R}_{pt}, \mathcal{R}_{st}$.}
\label{fig_fig1}
\end{figure}

\noindent{\it Case 1.} Gaussian noise. Recall the discussion in Remark \ref{rem_singu1}. In this case, the structure of the singular vectors does not play a role. We choose $\mathbf u= \mathbf e_1$ and $\mathbf v= \mathbf f_1$. Denote by
\begin{equation*}
\mathcal{R}_g:=\frac{\sqrt{n}}{\sigma}\left(|\langle \widehat{\mathbf{v}}, \mathbf{v}  \rangle|^2-a(d) \right), 
\end{equation*}
where 
\begin{equation}\label{defn_sigma}
\sigma^2=(8d^{12}+24d^{10}+26d^8+20d^6+15d^4+8d^2+2)/(2d^4(2d^4-1)(d^2+1)^4).
\end{equation}
The conclusion is that $\mathcal{R}_g$ is asymptotically $\mathcal{N}(0,1)$.  

\noindent{\it Case 2.} Two-point noise and both singular vectors of $S$ are delocalized. In the presence of Two-Point noise, the structure of the singular vectors will influence the distributions. We consider  the case that both $\mathbf u$ and $\mathbf v$ are delocalized, corresponding  to the discussion in Remark \ref{rem_singu2}. Let $\mathbf{u}=\mathbf{1}_M/\sqrt{M}$ and  $\mathbf{v}=\mathbf{1}_n/\sqrt{n}$. Then 
$$\mathcal{R}_{dt}:=\frac{1}{\sigma}\left(\sqrt{n}(|\langle \widehat{\mathbf{v}}, \mathbf{v} \rangle|^2-a(d))+\frac{d^6-1.5d^2-1}{d^5 (d^2+1)^2} \right)$$ 
is asymptotically $\mathcal{N}(0,1),$ where $\sigma$ is defined in (\ref{defn_sigma}). 

\noindent{\it Case 3.} Two-point noise and one of the singular vectors of $S$ is delocalized. We set $\mathbf{u}=\mathbf{1}_M/\sqrt{M}$ and $\mathbf{v}=\mathbf f_1$. From  Remark \ref{rem_singu3}, we know that the random variable $$ \mathcal{R}_{pt}:=\frac{\sqrt{n}}{\sigma_t} \left(|\langle \widehat{\mathbf{v}}, \mathbf{v}  \rangle|^2-a(d) \right) $$ is asymptotically $\mathcal{N}(0,1),$ where $$ \sigma_t^2= \sigma^2+2(d^4+d^2+0.5)^2/(d^7 (d^2+1)^4)-0.75(d^4+d^2+0.5)^2/(d^8 (d^2+1)^4).$$

\noindent{\it Case 4.} Two-point noise and both singular vectors of $S$ are sparse (localized). Let $\mathbf u= \mathbf e_1$ and $\mathbf{v}=\mathbf f_1$. From the proof of Proposition \ref{lem:main}, especially  the decomposition in (\ref{18072502}), 
by setting $$\mathcal{R}_{st}:=\frac{1}{\sigma_s} \left(\sqrt{n} \left(|\langle \widehat{\mathbf{v}}, \mathbf{v} \rangle|^2-a(d)\right)-\frac{2\sqrt{n}}{d^3} X_{11} \right),$$ 
with $$\sigma_s^2=(d^{16}+4d^{14}+6d^{12}+d^{10}-6d^8-2d^6+6.5d^4+6.25d^2+1.6875)/(d^8(d^2+1)^4(2d^4-1)),$$ we have that $\mathcal{R}_{st}$ is asymptotically $\mathcal N(0,1)$.

In Table \ref{table_singu}-\ref{table_singu4}, we record the probabilities for different quantiles of the empirical cumulative distributions (ECDF) of $\mathcal{R}_g, \mathcal{R}_{dt}, \mathcal{R}_{pt}, \mathcal{R}_{st}$ respectively. We choose $n=200$ or $500$. For each choice of $n$, we take $d=2,3,5,10$.  The first column corresponds to the theoretical quantile probabilities for a standard normal distribution. Each simulation is obtained with $10,000$ repetitions. From Table \ref{table_singu}, we observe that $\mathcal{R}_g$ is fairly close to standard Gaussian even for a small sample size $n=200$. (The same is also observed for $\mathcal{R}_{dt}, \mathcal{R}_{pt}, \mathcal{R}_{st}$.)

\begin{table}[H]
\setlength\arrayrulewidth{0.5pt}
\renewcommand{\arraystretch}{0.8}
\begin{center}
\begin{tabular}{clclclclclclc|c|c|c|c}
\toprule
 & \multicolumn{5}{c}{$n=200$}                                                                                                                       & \multicolumn{5}{c}{$n=500$}                                                                                                                       \\ \hline \\
Normal   & $d=2$& $d=3$  & $d=5$  & $d=10$  & $ \text{SE}$  & $d=2$ & $d=3$ & $d=5$ & $d=10$ & $\text{SE}$ \\ \hline\\
0.01 & 0.012  & 0.0134  &   0.0106 &  0.0128& 0.003 & 0.0128 & 0.0115 & 0.012  & 0.0115 & 0.002 \\
0.05 &  0.0536 & 0.0499  & 0.0466  & 0.0495  & 0.002 & 0.0525 & 0.0474  & 0.0496 & 0.0498  & 0.0014 \\
0.10& 0.0969  & 0.095  &  0.0909 & 0.0909 & 0.0066 & 0.0968 &  0.0975 & 0.0976 &  0.0961 & 0.003  \\
0.30 &  0.281 &  0.280 & 0.273  &  0.268 & 0.025  & 0.292  & 0.294 & 0.275 & 0.284 & 0.014  \\
0.50& 0.477  & 0.472  & 0.462  & 0.463 & 0.032 &  0.486 & 0.483 & 0.480  & 0.477  & 0.020   \\
0.70& 0.684 & 0.679 & 0.674  & 0.670 & 0.023   & 0.691 & 0.691 & 0.683 &  0.682 & 0.013 \\
0.90	&  0.899 & 0.899  & 0.896 & 0.901 &  0.002 & 0.898  & 0.901 & 0.898  & 0.896 & 0.002 \\
0.95 &  0.955  &  0.955   &  0.953 &  0.953 & 0.004  & 0.953 & 0.951 & 0.952 & 0.949 & 0.002\\
0.99 & 0.994  & 0.993 &  0.993 &  0.992 & 0.003  & 0.991  & 0.991 & 0.992 & 0.994  & 0.002 \\ 
 \bottomrule
\end{tabular}
\end{center}
\caption{Distribution of $\mathcal{R}_g$: Gaussian noise.  
} \label{table_singu}
\end{table}

\begin{table}[H]
\setlength\arrayrulewidth{0.5pt}
\renewcommand{\arraystretch}{0.8}
\begin{center}
\begin{tabular}{clclclclclclc|c|c|c|c}
\toprule
 & \multicolumn{5}{c}{$n=200$}                                                                                                                       & \multicolumn{5}{c}{$n=500$}                                                                                                                       \\ \hline \\
Normal   & $d=2$& $d=3$  & $d=5$  & $d=10$  & $ \text{SE}$  & $d=2$ & $d=3$ & $d=5$ & $d=10$ & $\text{SE}$ \\ \hline \\
0.01 & 0.011  &  0.011 & 0.013  & 0.013 & 0.002 & 0.0106 &  0.012 & 0.012  & 0.0106   & 0.001 \\
0.05 & 0.0455  & 0.0499 & 0.049  & 0.05  & 0.001 & 0.0473 & 0.053  &  0.0486 & 0.0496 & 0.002  \\
0.10& 0.0873  & 0.0923 & 0.0925  & 0.096 & 0.008 & 0.0905 & 0.099 & 0.0938 & 0.0945 & 0.006 \\
0.30 &  0.26 & 0.273  & 0.268 & 0.273 & 0.03 & 0.2645 & 0.28  & 0.274  & 0.276 & 0.03 \\
0.50& 0.462   & 0.469  &  0.461  & 0.466 & 0.04 & 0.46 & 0.478 & 0.47 & 0.474 & 0.03  \\
0.70& 0.668 &  0.665 & 0.67  & 0.68  & 0.03 & 0.6755 &  0.682 & 0.679 &  0.675 & 0.02   \\
0.90	& 0.892  &  0.887 & 0.887  &  0.897 & 0.009 & 0.899 &  0.898 & 0.892 & 0.895 & 0.004  \\
0.95 &  0.95 &  0.949  & 0.947 & 0.954 & 0.002 & 0.954 &  0.952 & 0.947  & 0.949 & 0.003 \\
0.99 & 0.9914 & 0.993  &  0.9914 &  0.99 & 0.001 & 0.992 & 0.992  &  0.992 & 0.992 & 0.002 \\  
 \bottomrule
\end{tabular}
\end{center}
\caption{Distribution of $\mathcal{R}_{dt}:$  Two-Point noise and delocalized singular vectors.  } \label{table_singu2}
\end{table}

\begin{table}[H]
\setlength\arrayrulewidth{0.5pt}
\renewcommand{\arraystretch}{0.8}
\captionsetup{width=1\linewidth}
\begin{center}
\begin{tabular}{clclclclclclc|c|c|c|c}
\toprule
 & \multicolumn{5}{c}{$n=200$}                                                                                                                       & \multicolumn{5}{c}{$n=500$}                                                                                                                       \\ \hline \\
Normal   & $d=2$& $d=3$  & $d=5$  & $d=10$  & $ \text{SE}$  & $d=2$ & $d=3$ & $d=5$ & $d=10$ & $\text{SE}$ \\ \hline \\
0.01 & 0.016  & 0.0151  & 0.011  & 0.0123 & 0.004  &0.011 &  0.011  & 0.011 &  0.011& 0.001 \\
0.05 & 0.053  & 0.0513 &  0.051 & 0.0464 & 0.002 & 0.051  & 0.0505 & 0.0478   & 0.0536 & 0.002  \\
0.10& 0.0976  & 0.0968 &  0.0955 &  0.0953 & 0.004 & 0.094 & 0.0959 & 0.0934  & 0.1 &  0.004 \\
0.30 & 0.273   & 0.275 & 0.279  & 0.268 & 0.03 &  0.277&  0.283 & 0.274  & 0.282 & 0.02  \\
0.50&  0.468  & 0.473 &   0.469 & 0.463 & 0.03 & 0.479 & 0.481 & 0.469 & 0.47 & 0.03 \\
0.70&  0.686 &  0.68 & 0.677   &  0.672 & 0.02 & 0.68 & 0.68 &  0.676 &   0.674 & 0.02 \\
0.90	&  0.9035 & 0.9025  & 0.895 &  0.897 & 0.004 & 0.908 & 0.897 & 0.892 & 0.891 & 0.007  \\
0.95 & 0.959 &  0.957  & 0.954 & 0.95 & 0.005 & 0.955 & 0.952 & 0.95 & 0.949 & 0.002 \\
0.99 & 0.995 & 0.991 &   0.994 & 0.993   & 0.003 & 0.993 &  0.992 & 0.993 & 0.991 & 0.002 \\  
 \bottomrule
\end{tabular}
\end{center}
\caption{ Distribution of $\mathcal{R}_{pt}:$ Two-Point noise and delocalized left singular vector and sparse right singular vector.  } \label{table_singu3}
\end{table}

\begin{table}[H]
\setlength\arrayrulewidth{0.5pt}
\renewcommand{\arraystretch}{0.8}
\begin{center}
\begin{tabular}{clclclclclclc|c|c|c|c}
\toprule
 & \multicolumn{5}{c}{$n=200$}                                                                                                                       & \multicolumn{5}{c}{$n=500$}                                                                                                                       \\ \hline \\
Normal   & $d=2$& $d=3$  & $d=5$  & $d=10$  & $ \text{SE}$  & $d=2$ & $d=3$ & $d=5$ & $d=10$ & $\text{SE}$ \\ \hline\\
0.01 & 0.0115  &  0.009&  0.008 & 0.0825 & 0.002 &  0.0099&   0.009 & 0.0098 & 0.0088 & 0.001 \\
0.05 &   0.0454 & 0.0448  & 0.042  & 0.0443 & 0.006 &  0.0469   &  0.0468& 0.045 & 0.044  & 0.004 \\
0.10&  0.0873 & 0.0886 &  0.081 & 0.0864 & 0.004 & 0.0908 & 0.095  & 0.091 &  0.0896 & 0.005   \\
0.30 & 0.266  & 0.270  & 0.269  & 0.275 & 0.030 &  0.280 &  0.278 & 0.282 & 0.270  & 0.02 \\
0.50&  0.460 & 0.463 & 0.460  & 0.453 & 0.042  &  0.467 & 0.473 &  0.478 & 0.463  & 0.03  \\
0.70& 0.666 & 0.670  & 0.660  & 0.656 &  0.037 &  0.673 & 0.680 & 0.673 & 0.663  & 0.03 \\
0.90	& 0.885 &  0.883 &  0.884 & 0.879 &  0.017  & 0.890 &  0.890 & 0.894 & 0.889 &  0.009 \\
0.95 & 0.944   & 0.940    & 0.940 &  0.939 &  0.009 & 0.943 &  0.943 & 0.948 & 0.948 &  0.005 \\
0.99 & 0.989 & 0.987 &  0.988 & 0.989  &  0.002 & 0.989  & 0.989 & 0.99 & 0.989 & 0.001\\ 
 \bottomrule
\end{tabular}
\end{center}
\caption{Distribution of $\mathcal{R}_{st}$: Two-Point noise and both singular vectors sparse.} \label{table_singu4}
\end{table}

Further, in Figure \ref{fig_fig1},  we plot the ECDFs of of $\mathcal{R}_g, \mathcal{R}_{dt}, \mathcal{R}_{pt}, \mathcal{R}_{st}$ in subfigures (A), (B), (C), (D) respectively, for $n=500$ and various values of $d=2,3,5,10$.

\section{Discussions on statistical applications}
\begin{figure}[ht]
\begin{subfigure}{0.4\textwidth}
\includegraphics[width=6cm,height=8cm]{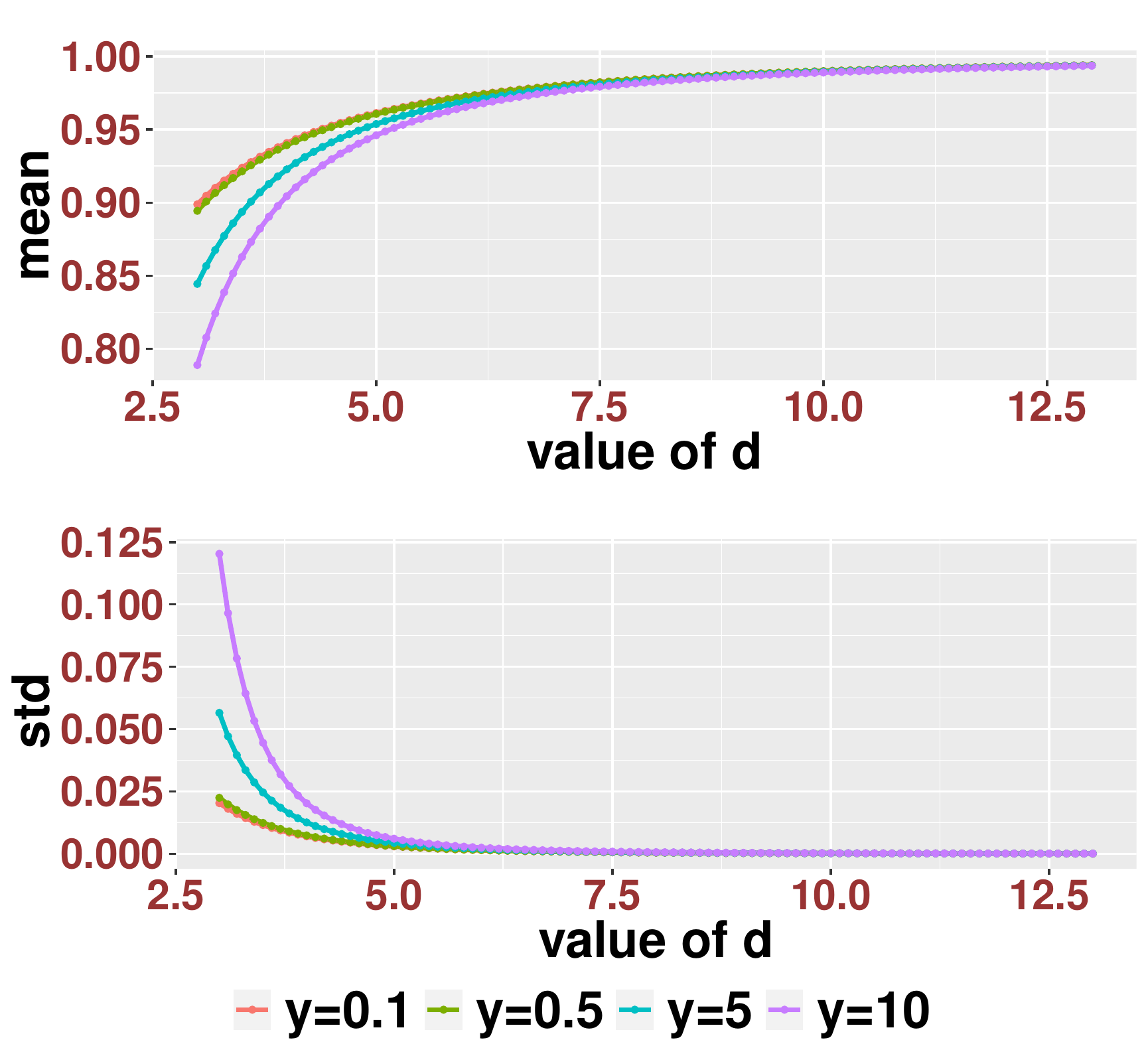}
\subcaption{Gaussian noise}
\end{subfigure}
\hspace{0.5cm}
\begin{subfigure}{0.4\textwidth}
\includegraphics[width=6cm,height=8cm]{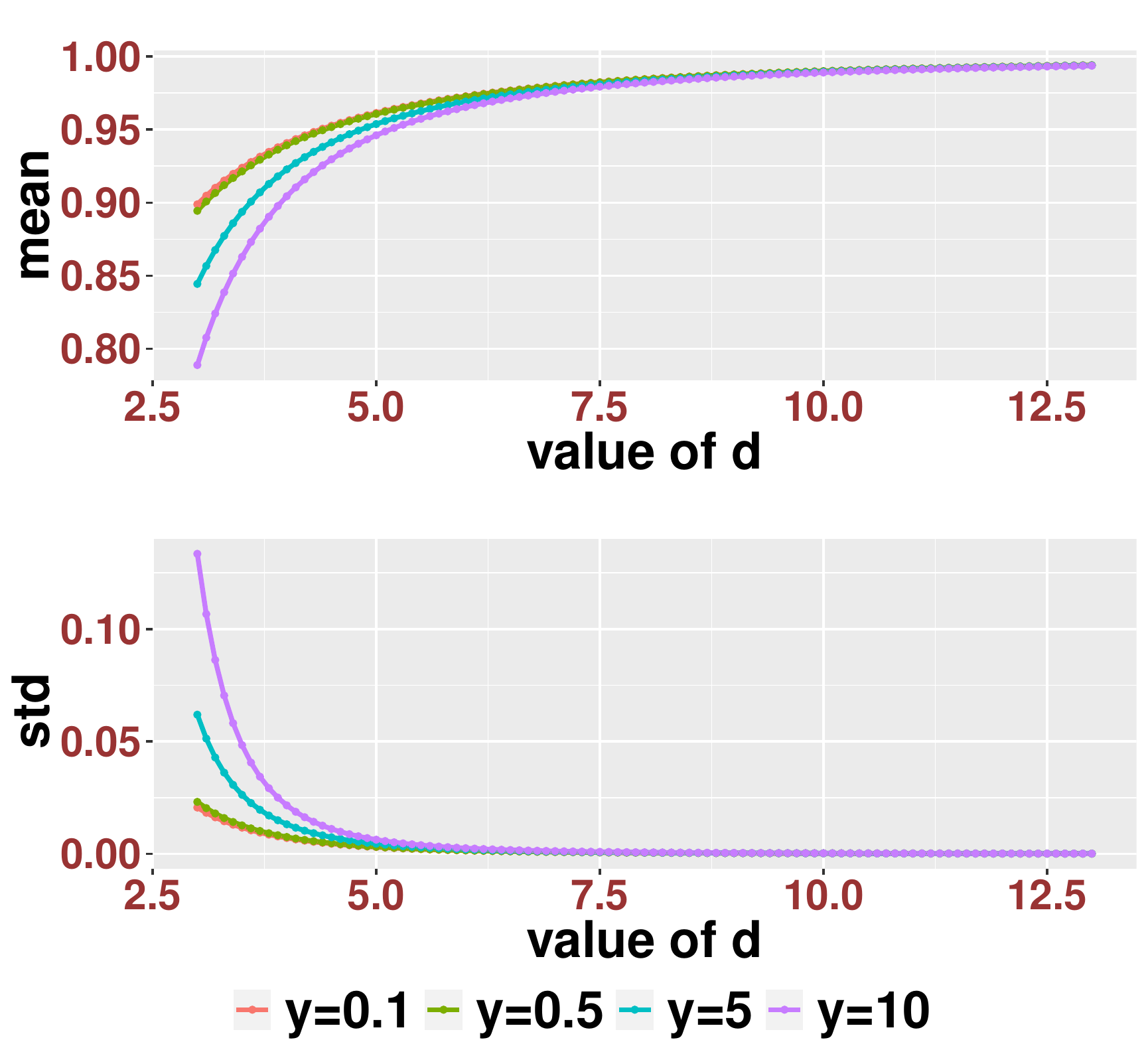}
\subcaption{Two-Point noise}
\end{subfigure}
\caption{Mean-Variance Discussion. In both of the figures, we plot the mean function $a(d)$ in the upper panel for $y=0.1, 0.5, 5, 10$ respectively for a sequence of values of $d$ lie between 3 and 13. In the lower panel, we plot the standard deviation of the fluctuation correspondingly.  Recall the definitions in (\ref{eq_defnthed}) and (\ref{eq_ved}). The standard deviation is $\sqrt{4 \theta(d)^2+\mathcal{V}^E(d)}$ for the Gaussian noise and  $\sqrt{4 \theta(d)^2+\mathcal{V}^E(d)+4 \sqrt{y} \theta(d)^2/(\sqrt{2} d) -3y \theta(d)^2/(2d^2)}$ for the Two-Point noise.  We choose the true right singular vector to be $\mathbf{f}_1$ and left singular vector to be $\mathbf{1}_M/\sqrt{M}$. Hence for the Two-Point noise, we need to add a part depending on $\kappa_3=1/\sqrt{2}$ and $\kappa_4=-3/2$. }
\label{fig_fig11}
\end{figure}
In this section, we provide simulation results of Section \ref{sec:statapp} and some further discussions on the statistical applications. We first provide the results of the mean-variance discussion of the estimation of singular vectors, which is illustrated in Figure \ref{fig_fig11}. In the following simulations, we consider the setting that the signal matrix $S$ has rank $r=2$ with the singular values $d_1=5$ and $d_2=3$. Assume $M$ is even. Assume the left singular vectors of $S$ are $\mathbf{u}_1=\frac{1}{\sqrt{M}} \mathbf{1}_M$ and $\mathbf{u}_2=\frac{1}{\sqrt{M}} (\mathbf{1}_{M/2}^T, -\mathbf{1}_{M/2}^T)^T$, a vector with the first half entries $1/\sqrt M$ and remaining entries $-1/\sqrt M$. Set $V_0=(\mathbf{f}_1, \mathbf{f}_2)$. 

Recall the definitions in  (\ref{eq_defnthed}) and (\ref{eq_ved}). When the noise is Gaussian, we use the statistic 
\begin{equation}\label{eq_t1g}
\bm{T}_{1g}=\frac{\sqrt{n}}{\sigma} \Big( \sum_{i,j=1}^2 |\langle \widehat{\mathbf{v}}_i, \mathbf{v}_j \rangle|^2-a(d_1)-a(d_2)\Big),
\end{equation}
where 
{
\begin{equation*}
\sigma^2=\sum_{i=1}^2 (4 \theta(d_i)^2+\mathcal{V}^E(d_i)). 
\end{equation*}
}

Note that $\bm{T}_{1g}$ is a scaled version of the proposed statistic $\bm{S}_1$ in (\ref{eq_s1_defn}), i.e. $\bm{T}_{1g}=\bm{S}_1/\sigma.$
When the noise is Two-point type, we use the statistic 
{
\begin{equation}\label{eq_t1t}
\mathbf{T}_{1t}:=\frac{\sqrt{n}}{\sigma_t}\Big(  \sum_{i,j=1}^2 |\langle \widehat{\mathbf{v}}_i, \mathbf{v}_j \rangle|^2-a(d_1)-a(d_2) \Big),
\end{equation} 
where 
\begin{equation*}
\sigma_t^2=\sum_{i=1}^2 \left( 4 \theta(d_i)^2 + \mathcal{V}^E(d_i) \right) -\frac{3y}{2} \sum_{i=1}^2 \frac{\theta(d_i)^2 }{d_i^2}+\frac{4\sqrt{y}}{\sqrt{2}}  \frac{\theta(d_1)^2}{d_1}.  
\end{equation*}
}
$\bm{T}_{1t}$ is also a scaled version of $\bm{S}_1.$

Under the nominal level $\alpha$, we will reject $\mathbf{H}_0$ when 
\begin{equation*}
|\mathbf{T}_{1g(t)}|>z_{1-\alpha/2},
\end{equation*}
where $z_{1-\alpha/2}$ is the $1-\alpha/2$ quantitle of a standard Gaussian random variable. In Table \ref{table_stat_single_typeone}, we record the type I error rates which show the accuracy of our proposed $z$-score test for different values of $y$ based on $10,000$ simulations. 
\begin{table}[ht]
\setlength\arrayrulewidth{1pt}
\renewcommand{\arraystretch}{1.1}
\captionsetup{width=1\linewidth}
\begin{center}
\begin{tabular}{cclclclclclclc|c}
\toprule
 & \multicolumn{4}{c}{Gaussian noise}                                                                                                                       & \multicolumn{4}{c}{Two-point noise}                                                                                                                       \\ 
 \cmidrule(lr){2-5}\cmidrule(lr){6-9}
 & \multicolumn{2}{c}{$\alpha=0.05$}                                                                                                                       & \multicolumn{2}{c}{$\alpha=0.1$}                                                                                                                      
  & \multicolumn{2}{c}{$\alpha=0.05$}                                                                                                                       & \multicolumn{2}{c}{$\alpha=0.1$}                                                                                                                       \\
  & $n=200$& $n=500$  & $n=200$& $n=500$ &  $n=200$& $n=500$  & $n=200$& $n=500$ \\ 
   \cmidrule(lr){2-5}\cmidrule(lr){6-9}
$y=0.5$ & 0.047   &  0.0482  &  0.098 & 0.0967 & 0.0501 & 0.0496 & 0.105 & 0.0945   \\
$y=1$ &  0.057 & 0.046  & 0.092  & 0.096   & 0.0488 & 0.0491 & 0.097 & 0.099   \\
 $y=2$ & 0.0494  & 0.052 & 0.0984  & 0.0955 & 0.0474 & 0.049 & 0.091 & 0.094   \\
 \bottomrule
\end{tabular}
\end{center}
\caption{ Type I error under $\mathbf{H}_0$ for (\ref{test_1}) using $z$-score test.  } \label{table_stat_single_typeone}
\end{table}

 Finally, to study the power of our test against the alternatives,  we consider the matrix  $V_a=(\mathbf{f}_1, \sqrt{1-\delta^2} \mathbf{f}_2+ \delta \mathbf{f}_3)$ for a parameter $\delta \in (0,1).$ In Figure \ref{fig_power}, we record the simulated power for different values of $\delta$ under the nominal level $\alpha=0.05$ when $X$ is a Two-point noise matrix.   We find that the power of our tests increases  when $\delta$ increases. Furthermore, at the same level of $\delta,$ the power is improved when $n$ increases.

\begin{figure}[H]
 \captionsetup{width=1\linewidth}
\includegraphics[width=8cm,height=6cm]{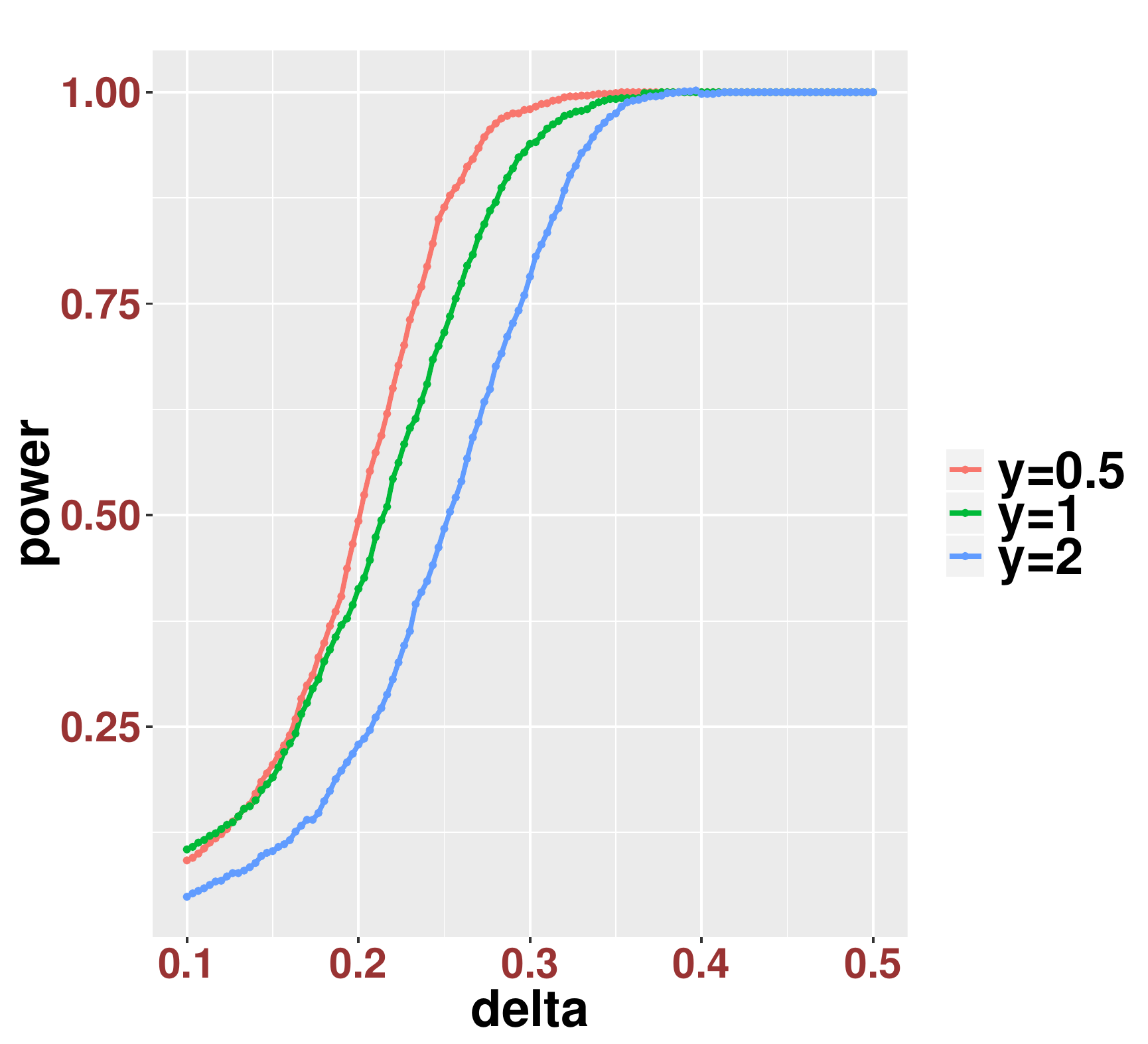}
\caption{ Power vs $\delta$ under the nominal level $\alpha=0.05$ for $y=0.5, 1, 2$ respectively. }
\label{fig_power}
\end{figure}

As we mentioned in the article \cite{BDW}, we assume that $D$, $U$ and the necessary parameters of $X$ are known and we do the hypothesis testing about $V$. Although in general we cannot drop all the a priori information about $D$, $U$, and $X$, some efforts can be made along this direction. In the sequel, for instance, 
 we discuss some possible extension to the case when $D$ and the necessary parameters of $X$ are unknown.
More specifically, recall that $d_i$ and $\sqrt{\mu_i}$ are the singular values of $S$ and $Y$, respectively. We know that $\sqrt{\mu_i}$ converges to $p(d_i)$ in probability. In our theorems for singular vectors, we can use $p^{-1}(\sqrt{\mu_i})$ to replace $d_i$. Such a replacement will change the distribution of our statistics. 
For instance, if we change $a(d_i)$ by $a(p^{-1}(\sqrt{\mu_i}))$ in Theorem \ref{thm_mainthm}, such a change will bring additional fluctuation of the statistic. However, one can still use our method to derive the limiting distribution for such a modified statistic where $d_i$'s are replaced by their estimates, i.e.,  $p^{-1}(\sqrt{\mu_i})$'s. It is simply because the fluctuation of $\sqrt{\mu_i}$ can also be written as a quadratic form of the Green function. We anyway have the joint distribution of the quadratic forms of the Green function and its derivative. So we can also derive the joint distribution of the singular vectors  and the singular values. Replacing $d_i$ in quantities like $\Delta_i$ and $\mathcal{V}_i$ by the estimator $p^{-1}(\sqrt{\mu_i})$ is completely harmless since the error  is of order $\frac{1}{\sqrt{n}}$.   
Further, in some simple case, we can also estimate the cumulants of the noise of $X.$ For instance, suppose we use the $p^{-1}(\sqrt{\mu}_i)$ to replace $d_i$ in our theorems, and further we assume that the  left singular vectors are known and we want to test whether the right singular subspace $V$ is identical to some given matrix $V_0$. In this case, we can estimate the parameters of $X$ by considering  $\widehat{X}:=Y-\widehat{S}$. Here $\widehat{S}= U\widehat{D}V^*$ with $\widehat{D}=\text{diag}(p^{-1}(\sqrt{\mu_i}))_{i=1}^r$.  Since the entries of $X$ are assumed to be i.i.d., we can estimate the the second moment of $\sqrt{n}x_{ij}$'s as the following 
\begin{align*}
\frac{1}{M}\sum_{i,j} |x_{ij}|^2= \frac{1}{M} \text{Tr} XX^*=& \frac{1}{M}\text{Tr}\widehat{X}\widehat{X}^*-\frac{1}{M}\text{Tr}(S-\widehat{S})\widehat{X}^*-\frac{1}{M}\text{Tr} \widehat{X}(S-\widehat{S})^*\nonumber\\
&+\frac{1}{M}\text{Tr} (S-\widehat{S})(S-\widehat{S})^*. 
\end{align*} 
Using the facts $\|\widehat{X}\|=O(1)$, $\|S-\widehat{S}\|=O(n^{-\frac12})$ and $\text{rank} (S-\widehat{S})=r$ which is fixed, it is easy to see that the last three terms are of order  $n^{-1}$ in probability. Further, it is easy to see  that $ \frac{1}{M} \text{Tr} XX^*$ can estimate $\mathbb{E}(\sqrt{n}x_{ij})^2$ up to an error of order $\frac{1}{n}$ in probability. Hence, we can estimate $\mathbb{E}(\sqrt{n}x_{ij})^2$ by $ \frac{1}{M}\text{Tr}\widehat{X}\widehat{X}^*$ which can be computed from the data, if the null hypothesis holds.  The other cumulants can be estimated in a similar way. 

In practice, some of the extension above could be quite important. For instance, the joint distribution of the singular values and  vectors allow us to consider the inference on statistics involving both of them. In \cite{LL}, to test whether the community memberships of the two networks are the
same in the stochastic block model, the authors proposed a statistic involving the scaled principal angles, where the scalings are the singular values (see \cite[Section 4.2]{LL} for details). Further, in \cite{DM}, the authors derived the formulas for the optimal shrinkers of the singular values under various norms. These shrinkers are essentially combinations of products of the singular values and inner products of $|\langle \mathbf{u}_i, \widehat{\mathbf{u}}_i \rangle|, |\langle \mathbf{v}_i, \widehat{\mathbf{v}}_i \rangle|$. Finally, there exist a lot works on estimating the low rank matrix $S$, to name but a few \cite{JW,N, RSV}. With the results on the joint distribution, it is possible for us to do inference on the estimation of the low-rank matrix $S$.  Nevertheless, we leave all the extensions to the future work.

\section{Preliminary results}
In this section, we list some preliminary results which will be used in the technical proof.  
\subsection{Auxiliary lemmas}
A key tool for our computation is the following cumulant expansion formula, whose proof can be found in \cite[Proposition 3.1]{LP09} and \cite[Section II]{KKP96}, for instance.  
\begin{lem}\label{lem_cumu}
 Let $\ell\in \mathbb{N}$ be fixed and let $f\in C^{\ell+1}(\mathbb{R})$. Let $\xi$ be a centered random variable with finite first $\ell+2$ moments. Let $\kappa_k(\xi)$ be the $k$-th cumulant of $\xi$. Then we have the expansion
\begin{align}
\mathbb{E}(\xi f(\xi))= \sum_{k=1}^\ell \frac{\kappa_{k+1} (\xi)}{k!} \mathbb{E} (f^{(k)}(\xi))+\mathbb{E}(\epsilon_{\ell}(\xi f(\xi))), \label{18012980}
\end{align} 
where $\epsilon_{\ell}(\xi f(\xi))$ satisfies 
\begin{align*}
|\mathbb{E}(\epsilon_{\ell}(\xi f(\xi)))| \leq C_\ell \mathbb{E}(|\xi|^{\ell+2})\sup_{|t|\leq \chi}|f^{(\ell+1)}(t)|+ C_\ell\mathbb{E}(|\xi|^{\ell+2}\mathbf{1}(|\xi|>\chi))\sup_{t\in \mathbb{R}}|f^{(\ell+1)}(t)|
\end{align*} 
for any $\chi>0$. 
\end{lem}

Note that when $\xi$ is a standard Gaussian random variable (i.e. $\kappa_i=0, i \geq 3$), \eqref{18012980} boils down to the celebrated Stein's lemma \cite{Stein1981}.  Next we introduce the identities on the derivatives of the Green functions in (\ref{green2}). These can be verified by elementary calculus so we omit the proofs. For $i\in [M]$ and $j\in [n]$, denote by  $E_{ij'}$ the $(M+n) \times (M+n)$ matrix with entry 1 on the $(i,j')$ position and 0 elsewhere. 
\begin{lem} \label{lem_greenderivative}
Let $\mathcal{E}_{ij}=E_{ij^{\prime}}+E_{j^{\prime}i}$ and $k\in \N$. We have 
\begin{align*}
&\frac{\partial^k G}{\partial x_{ij}^k}=(-1)^k k! z^{\frac{k}{2}} (G\mathcal{E}_{ij})^k G,\nonumber\\
&\frac{\partial^k (G^2)}{\partial x_{ij}^k}=(-1)^k k! z^{\frac{k}{2}}\sum_{s=0}^k (G \mathcal{E}_{ij})^s G (G \mathcal{E}_{ij})^{k-s} G. 
\end{align*}
\end{lem} 

Below we also collect some basic results on convergence and equivalence in distribution for sum of random variables.  They can be found in \cite[Lemma 7.7, 7.8 and 7.10]{KY13}.  
\begin{lem} \label{lem_convergenindist}
(1). Let $\mathsf{X}_n \simeq \mathsf{Y}_n$ and $\mathsf{R}_n$ satisfy $\lim_{n \rightarrow \infty} \mathbb{P}\Big(|\mathsf{R}_n| \leq \epsilon_n \Big)=1,$ where $\{\epsilon_n\}$ is a positive null sequence. Then
\begin{equation*}
\mathsf{X}_n \simeq \mathsf{Y}_n+\mathsf{R}_n. 
\end{equation*} 
(2). Let $\{\mathsf{X}_n\}, \{\mathsf{X}_n^{\prime}\}, \{\mathsf{Y}_n\}$ and $\{\mathsf{Y}_n^{\prime}\}$ be sequences of random variables. Suppose $\mathsf{X}_n \simeq \mathsf{X}_n^{\prime},$ $\mathsf{Y}_n \simeq \mathsf{Y}_n^{\prime},$ $\mathsf{X}_n$ and $\mathsf{Y}_n$ are independent, and $\mathsf{X}_n^{\prime}$ and $\mathsf{Y}_n^{\prime}$ are independent. Then 
\begin{equation*}
\mathsf{X}_n+\mathsf{Y}_n \simeq \mathsf{X}_n^{\prime}+\mathsf{Y}_n^{\prime}.
\end{equation*} 
(3). Let $\{\mathsf{Z}_n\}$ be a bounded deterministic sequence. Let $\{\mathsf{X}_n\}$ be random variables such that $\mathsf{X}_n$ converges weakly to $\mathsf{X}.$ Then for any bounded continuous function $f,$  as $n \rightarrow \infty,$  we have
\begin{equation*}
\mathbb{E}f (\mathsf{Z}_n \mathsf{X}_n)-\mathbb{E} f(\mathsf{Z}_n \mathsf{X}) \rightarrow 0.
\end{equation*}
\end{lem}

\subsection{Collection of derivatives}

In this part, we summarize some basic identities on the derivatives of $G$ and $Q$ defined in \eqref{18072501} without proof.  Recall the notation introduced in (\ref{eq_permu}). 

Using Lemma \ref{lem_greenderivative}, it is easy to check 
\begingroup
\allowdisplaybreaks
\begin{align}
&\Big( \frac{\partial^2 G}{\partial x^2_{ij}} W\Big)_{ab}=2z \sum_{\substack{l_1,\cdots,l_4 \in \{i,j' \}\\ l_1\neq l_2, l_3\neq l_4}} G_{a l_1} G_{l_2 l_3} (GW)_{l_4 b},\label{eq_secondgm}\\
&\Big( \frac{\partial^3 G }{\partial x_{ij}^3} W \Big)_{ab}=-6 z^{\frac32} \sum_{\substack{l_1,\cdots,l_6 \in \{i,j' \}\\ l_1\neq l_2, l_3\neq l_4, l_5\neq l_6}} G_{a l_1} G_{l_2 l_3} G_{l_4 l_5}(GW)_{l_6 b}, \label{eq_thirdgm}\\
&\Big(  \frac{\partial^4 G }{\partial x_{ij}^4}  W \Big)_{ab}=24 z^2 \sum_{\substack{l_1,\cdots,l_8 \in \{i,j' \}\\ l_1\neq l_2, l_3\neq l_4, l_5\neq l_6,l_7\neq l_8}} G_{a l_1} G_{l_2 l_3} G_{l_4 l_5} G_{l_6 l_7} (GW)_{l_8 b}. \label{eq_fourgm}
\end{align}
\endgroup
and also the following identities
\begingroup
\allowdisplaybreaks
\begin{align}
&\Big( \frac{\partial^2 G^2}{\partial x^2_{ij}} W\Big)_{ab}=2z \sum_{(a_1,a_2,a_3)\in \mathcal{P}(2,1,1)}\sum_{\substack{l_1,\cdots,l_4 \in \{i,j' \}\\ l_1\neq l_2, l_3\neq l_4}} G_{a l_1}^{a_1} G_{l_2 l_3}^{a_2} (G^{a_3}W)_{l_4 b},\nonumber\\
&\Big( \frac{\partial^3 G^2 }{\partial x_{ij}^3} W \Big)_{ab}=-6 z^{\frac32} \sum_{(a_1,\ldots,a_4)\in \mathcal{P}(2,1,1,1)}\sum_{\substack{l_1,\cdots,l_6 \in \{i,j' \}\\ l_1\neq l_2, l_3\neq l_4, l_5\neq l_6}} G_{a l_1}^{a_1} G_{l_2 l_3}^{a_2} G_{l_4 l_5}^{a_3} (G^{a_4}W)_{l_6 b}, \nonumber\\
&\Big(  \frac{\partial^4 G^2 }{\partial x_{ij}^4}  W \Big)_{ab}=24 z^2 \sum_{(a_1,\ldots,a_5)\in \mathcal{P}(2,1,1,1,1)}\sum_{\substack{l_1,\cdots,l_8 \in \{i,j' \}\\ l_1\neq l_2, l_3\neq l_4, l_5\neq l_6,l_7\neq l_8}} G_{a l_1}^{a_1} G_{l_2 l_3}^{a_2} G_{l_4 l_5}^{a_3} G_{l_6 l_7}^{a_4} (G^{a_5}W)_{l_8 b}. \label{eq_partig2w}
\end{align}
\endgroup

Similarly, using  Lemma \ref{lem_greenderivative} and a  discussion similar to (\ref{eq_parti_1}), we can also derive  
\begingroup
\allowdisplaybreaks
\begin{align}
&\frac{\partial^2 Q}{\partial x^2_{ij}}=2z \sqrt n  \sum_{\substack{ l_1,\cdots,l_4\in \{ i, j'\}  \\ l_1\neq l_4, l_2 \neq l_3 }}  \Big( (GAG)_{l_1 l_2} G_{l_3 l_4} -\frac{1}{2z} (GBG)_{l_1 l_2} G_{l_3 l_4} \nonumber\\
& \qquad\qquad\qquad\qquad + \frac{1}{2} \sum_{(a_1,a_2,a_3) \in \mathcal{P}(2,1,1)} (G^{a_1} B G^{a_2})_{l_1 l_2} G^{a_3}_{l_3 l_4} \Big), \label{eq_parti_2}\\
&\frac{\partial^3 Q}{\partial x^3_{ij}}=-6z^{\frac32} \sqrt n \sum_{\substack{ l_1,\cdots, l_6\in \{ i, j'\}  \\ l_1\neq l_6, l_2 \neq l_3, l_4\neq l_5 }}  \Big( (GAG)_{l_1 l_2} G_{l_3 l_4} G_{l_5 l_6} -\frac{1}{2z}(GBG)_{l_1 l_2} G_{l_3 l_4} G_{l_5 l_6} \nonumber\\
& \qquad\qquad\qquad\qquad + \frac{1}{2}  \sum_{(a_1,\ldots,a_4) \in \mathcal{P}(2,1,1,1)} (G^{a_1} B G^{a_2})_{l_1 l_2} G^{a_3}_{l_3 l_4} G^{a_4}_{l_5 l_6}  \Big), \label{eq_parti_3}\\
&\frac{\partial^4 Q}{\partial x_{ij}^4}=24z^2 \sqrt{n}  \sum_{\substack{l_1, \cdots, l_8 \in \{i,j'\} \\ l_1 \neq l_8, l_2 \neq l_3, \\ l_4 \neq l_5, l_6 \neq l_7}}  \Big( (GAG)_{l_1 l_2} G_{l_3 l_4} G_{l_5 l_6} G_{l_7 l_8} -\frac{1}{2z}(GBG)_{l_1 l_2} G_{l_3 l_4} G_{l_5 l_6} G_{l_7 l_8} \nonumber\\
&\qquad\qquad\qquad\qquad+ \frac{1}{2}  \sum_{(a_1,\ldots,a_5) \in \mathcal{P}(2,1,1,1,1)} (G^{a_1} B G^{a_2})_{l_1 l_2} G^{a_3}_{l_3 l_4} G^{a_4}_{l_5 l_6}  G^{a_5}_{l_7l_8} \Big).\label{eq_parti_4}
\end{align}
\endgroup

\section{Proof of Lemmas \ref{lem.18091410} and  \ref{lem:key}}

\begin{proof}[Proof of Lemma \ref{lem.18091410}]
We focus our discussion on the first identity (\ref{18072530}). Differentiating $z$ on both sides of the equation
\begin{equation*}
G(H-z)=I,
\end{equation*}
we can get that 
\begin{equation*}
G'(H-z)+\frac{1}{2z}G(H-2z)=0.
\end{equation*}
The proof follows by multiplying $G$ on both sides of the above equation. For $G^3$ and $G^4,$ we can compute them recursively  by differentiating the following two equations respectively
\begin{equation*}
G^2(H-z)=G, \ G^3(H-z)=G^2. 
\end{equation*}
This completes the proof.
\end{proof}

\begin{proof}[Proof of Lemma \ref{lem:key}]

To prove Lemma \ref{lem:key}, we first  need  the following result  from \cite{Ding}.
\begin{lem} [Theorem 3.3 and 3.4 of \cite{Ding}] \label{lem:Ding}Under assumptions of (\ref{eq_meanva}), (\ref{eq_moment}), (\ref{eq_ratioassumption}) and Assumption \ref{assum_main}, for $i,j\in [r]$,  we have 
\begin{equation*}
|\mu_i-p(d_i)|=O_{\prec}(n^{-\frac12}).
\end{equation*}
In addition, for the singular vectors, we have 
\begin{equation*}
|\langle \mathbf{u}_i, \widehat{\mathbf{u}}_i \rangle^2-a_1(d_i)| =O_{\prec}(n^{-\frac12}), \qquad   |\langle \mathbf{v}_i, \widehat{\mathbf{v}}_i \rangle^2-a(d_i)|=O_{\prec}(n^{-\frac12}),
\end{equation*}
and for $1 \leq i \neq j \leq r,$
\begin{equation*}
| \langle \mathbf{u}_i, \widehat{\mathbf{u}}_j \rangle|^2=O_{\prec}(\frac{1}{n}), \qquad  | \langle\mathbf{v}_i, \widehat{\mathbf{v}}_j \rangle|^2=O_{\prec}(\frac{1}{n}).
\end{equation*}
\end{lem} 
With Lemma \ref{lem:Ding}, we can rewrite (\ref{18012401}) as
\begin{equation} \label{eq_reduce1}
L= \sum_{i=1}^r |\langle\widehat{\mathbf{u}}_i,\mathbf{u}_i\rangle|^2+O_{\prec}(\frac{1}{n}) \quad\text{and}\quad R=\sum_{i=1}^r |\langle\widehat{\mathbf{v}}_i,\mathbf{v}_i\rangle|^2+O_{\prec}(\frac{1}{n}).
\end{equation}
We next write the above quantities in terms of the Green functions. Recall from (\ref{180130110}) $\mathcal{Y}\equiv \mathcal{Y}(z)$ and denote by $\widehat{G}(z)=(\mathcal{Y}-z)^{-1}$.  By spectral decomposition, we write 
\begin{align}\label{main_representation}
\widehat{G}(z) = \sum_{i=1}^{M \wedge n} \frac{1}{\mu_i-z} \left( {\begin{array}{*{20}c}
   { \widehat{\mathbf u}_i \widehat{\mathbf u}_i^*} &  z^{-1/2} \sqrt{\mu}_i \widehat{\mathbf u}_i \widehat{\mathbf v}_i^*\\
   z^{-1/2} \sqrt{\mu}_i { \widehat{\mathbf v}_i \widehat{\mathbf u}_i^*} & { \widehat{\mathbf v}_i \widehat{\mathbf v}_i^*}  \\
\end{array}} \right)\nonumber\\
-\frac{1}{z}\sum_{i=M \wedge n}^{M \vee n}\left( \begin{array}{*{20}c} \mathbf{1}_{M>n} \widehat{\bm{u}}_i \widehat{\bm{u}}_i^* & \bm{0} \\
\bm{0} & \mathbf{1}_{M<n} \widehat{\bm{v}}_i \widehat{\bm{v}}_i^* \\
\end{array} \right).
\end{align}

For any $i\in [r]$, denote $\Gamma_i:=\partial B_{\rho}(d_i),$ where $B_{\rho}(d_i)$ is the open disc of radius $\rho$ around $d_i.$ Here $\rho$ is chosen to be a small but fixed positive number such that different discs corresponding to different $d_i$ do not have overlaps. This is achievable due to Assumption \ref{assum_main}. We start with the right singular vectors. Denote
\begin{equation*}
\widehat{\mathcal{G}}_1(z)=(YY^*-z)^{-1},  \ \widehat{\mathcal{G}}_2(z)=(Y^*Y-z)^{-1}.
\end{equation*}
Note that on one hand, we have for $i \leq r,$
\begin{equation*}
\langle \bm{v}_i,  \widehat{\mathcal{G}}_2(z) \bm{v}_i \rangle= \langle \bm{\mathsf{v}}_i, \widehat{G}(z) \bm{\mathsf{v}}_i \rangle, \quad \bm{\mathsf{v}}_i=(\bm{0}, \bm{v}_i^*)^*. 
\end{equation*}
One the other hand, by Lemma \ref{lem:Ding} and Cauchy's integral formula, with high probability, we have 
\begin{equation*}
\widehat{\bm{v}}_i \widehat{\bm{v}}_i^*= -\frac{1}{2 \pi \ii}\oint_{p(\Gamma_i)} \widehat{\mathcal{G}}_2(z) dz.
\end{equation*}
Together with (\ref{main_representation}), with high probability, we have the following integral representation
\begin{align*}
|\langle \mathbf{v}_i, \widehat{\mathbf{v}}_i\rangle|^2= \frac{1}{2d_i^2\pi \ii}\oint_{p(\Gamma_i)} \big(\big(\mathcal{D}^{-1}+\mathcal{U}^*G(z)\mathcal{U}\big)^{-1}\big)_{ii}\frac{{\rm d} z}{z},
\end{align*}
where we used the fact that
\begin{equation*}
\mathcal{U}^* \widehat{G}(z) \mathcal{U}=\mathcal{D}^{-1}-\mathcal{D}^{-1} (\mathcal{D}^{-1}+\mathcal{U}^*G(z) \mathcal{U})^{-1} \mathcal{D}^{-1}. 
\end{equation*}
Recall (\ref{eq_piz}) and denote 
\begin{align}
 \Psi(z)= -\mathcal{U}^* (\Pi_1(z)-G(z))\mathcal{U}. \label{18012501}
\end{align}
Using Lemma \ref{lem_localoutside},  we have
\begin{align}
\|\Psi(z)\|_{\text{op}}=O_{\prec}(n^{-\frac12}), \quad z \in \mathbf{S}_o. \label{18012910}
\end{align}
We can decompose $\mathcal{D}^{-1} (\mathcal{D}^{-1}+\mathcal{U}^*G(z) \mathcal{U})^{-1} \mathcal{D}^{-1}$ as
\begin{equation*}
\mathcal{D}^{-1}+\mathcal{U}^*G(z) \mathcal{U}= \mathcal{D}^{-1}+\mathcal{U}^*\Pi_1(z) \mathcal{U}+\Psi(z).
\end{equation*}
We further employ the resolvent expansion for $(\mathcal{D}^{-1}+\mathcal{U}^*G(z) \mathcal{U})^{-1}$ to write
\begin{align*}
|\langle \mathbf{v}_i, \widehat{\mathbf{v}}_i\rangle|^2= \frac{1}{d_i^2}(S_0+S_1)+O_{\prec}(\frac{1}{n}),
\end{align*}
where
\begin{align} \label{eq_s1}
&S_0=\frac{1}{2\pi \ii}\oint_{p(\Gamma_i)} \big(\big(\mathcal{D}^{-1}+\mathcal{U}^*\Pi_1(z)\mathcal{U}\big)^{-1}\big)_{ii}\frac{{\rm d} z}{z},   \nonumber\\
&S_1=\frac{1}{2\pi \ii} \oint_{p(\Gamma_i)} \Big( \big(\mathcal{D}^{-1}+\mathcal{U}^*\Pi_1(z)\mathcal{U}\big)^{-1}\Psi(z)\big(\mathcal{D}^{-1}+\mathcal{U}^*\Pi_1(z)\mathcal{U}\big)^{-1}\Big)_{ii} \frac{{\rm d}z}{ z}.
\end{align}
Here we used a discussion similar to Eq. (5.19) and Lemma 5.5 of \cite{Ding} and omit further details.  By the residual theorem, we have $S_0= d_i^2 a(d_i).$  Recall (\ref{eq_definw}) and denote 
\begin{equation*}
f_i(z):=-\text{Tr}\big(\Xi_1(z) \mathcal{U} W_i(z) \mathcal{U}^* \big).
\end{equation*}
We can then write
\begin{align*}
S_1=\frac{1}{2\pi \ii} \oint_{p(\Gamma_i)} \frac{zf_i(z)}{(zm_{1c}(z)m_{2c}(z)-d_i^{-2})^2} {\rm d} z.
\end{align*}
As $p(d)$ is a monotone function when $d> y^{1/4}$ and by Lemma \ref{lem_m1m2quanty},   we find that 
\begin{align*}
S_1=\frac{d_i^4}{2\pi \ii} \oint_{\Gamma_i} \frac{p(\zeta)f_i(p(\zeta))\zeta^4p'(\zeta)}{(d_i-\zeta)^2(d_i+\zeta)^2} {\rm d} \zeta.
\end{align*}
Then, by residue theorem, we obtain 
\begin{align}\label{18012503}
S_1 =  d_i^4  
\Big( f_i(p(\zeta)) \frac{\zeta^4p'(\zeta) p(\zeta)}{(d_i+\zeta)^2}\Big)' \Big|_{\zeta=d_i}
=d_i^2 \Tr \big(\Xi_1(p_i) A_i^R \big)+d_i^2 \Tr \big(\Xi_1^{\prime}(p_i) B_i^R \big),
\end{align}
where we recall (\ref{defn_h}) and  the definitions of $A_i^R$ and $B_i^R$ in \eqref{eq_ab}. The conclusion for $|\langle \mathbf{v}_i, \widehat{\mathbf{v}}_i\rangle|^2$ follows immediately.

The above discussion holds for all $i\in[r].$ Rearranging the terms of (\ref{18012503}) and using  Lemma \ref{lem_m1m2quanty}, we can conclude our proof for $R$ using (\ref{eq_reduce1}). Similar discussion yields the conclusion of $|\langle \mathbf{u}_i, \widehat{\mathbf{u}}_i\rangle|^2$ for each $i\in [r]$ and $L.$ This completes the proof of Lemma \ref{lem:key}. 
\end{proof}

\section{Proof of Proposition \ref{prop_iteration}} \label{s. proof of recursive estimate}

This section is devoted to the proof of  Proposition \ref{prop_iteration}. In Proposition \ref{prop_iteration}, we choose different parameters, $z$ and $z_0$,  for $Q$ and $\Delta$, separately. However, for brevity, we will omit both two parameters for simplicity
in the sequel. 

First of all, applying (\ref{eq_localoutside}) to the definition in (\ref{defn_q1}), we have
\begin{equation}\label{eq_boundq}
\mathcal{Q}=O_{\prec}(1).
\end{equation}

 Denote $(M+n) \times (M+n)$ diagonal matrices
\begin{align} \label{eq_diagonalmatrix}
\mathtt{I}^{\mathrm{u}}:= \begin{pmatrix} I_M & \\ & 0 \end{pmatrix} \quad\text{and}\quad 
\mathtt{I}^{\mathrm{l}}:=\begin{pmatrix} 0 & \\ & I_n \end{pmatrix}.
\end{align}
 We further define  $A_1=A\mathtt{I}^{\mathrm{u}}, \ A_2=A\mathtt{I}^{\mathrm{l}}$ and define $B_1,B_2$ analogously.  In addition, we set 
\begin{align}
&f_\alpha:= -m_{\alpha } \Tr H\Xi_1A_\alpha+(1+zm_{\alpha }) \Tr GA_\alpha,\nonumber\\
&g_\alpha:=-\frac{m_{\alpha }}{2} \Tr H \Xi_2 B_\alpha + \frac{1+zm_{\alpha }}{2} \Tr G^2 B_\alpha +\frac{zm_{\alpha }-1}{2z} \Tr G B_\alpha \nonumber\\
&\qquad\qquad - m'_{\alpha } \Tr B_\alpha +m_{\alpha }' \Tr H\Pi_1 B_\alpha, \qquad \alpha=1,2. \label{18072570}
\end{align}
The  proof of Proposition \ref{prop_iteration} is based on the following two lemmas. 
\begin{lem} \label{lem.rewrite of Q} 
Recall (\ref{defn_detministic1}) and (\ref{18072410}). For $z $ defined in (\ref{eq_zdefn}), we have 
\begin{align}\label{eq:anotherQ}
  Q &= \sqrt n \big( f_1 + f_2+g_1+g_2 \big)  +  \sqrt{n z} \sum_{(i,j)\in \mathcal{S}(\nu) } c_{ij}  x_{ij}  -\Delta_d.
\end{align}
\end{lem}
To state the second crucial lemma, Lemma \ref{lem. recursive estimate for each term}.  We first introduce some notations. Recall that $\Pi_a \ (1\le a \le 4)$ in \eqref{eq_piz} and \eqref{18072531} approximates $G^a$. We introduce the following matrices to approximate the powers of $G$ interacting with block diagonal matrices $\mathtt{I}^{\mathrm{u}}$ and $\mathtt{I}^{\mathrm{l}}$. For $1\le a_1,a_2\le 2$, define
\begin{align}\label{eq_piipi}
\Pi_{a_1,a_2}^{\mathrm{u}}:=\Pi_{a_1}  \mathtt{I}^{\mathrm{u}} \Pi_{a_2}
\quad\text{and}\quad
\Pi_{a_1,a_2}^{\mathrm{l}}:=\Pi_{a_1} \mathtt{I}^{\mathrm{l}} \Pi_{a_2}.
\end{align}
Note that they approximate $G^{a_1}  \mathtt{I}^{\mathrm{u}} \Pi_{a_2}$ and $G^{a_2} \mathtt{I}^{\mathrm{l}} \Pi_{a_2}$ respectively.  We further define
\begin{align}\label{eq_pi2ul}
\Pi_{2}^{\mathrm{u}}:= m_1' I_M \oplus (m_2' + \frac{1}{z}m_2) I_n
\quad\text{and}\quad
\Pi_{2}^{\mathrm{l}}:= (m_1' + \frac{1}{z} m_1) I_M \oplus m_2' I_n,
\end{align}
which approximate $G\mathtt{I}^{\mathrm{u}} G$ and $G\mathtt{I}^{\mathrm{l}} G$.

We need to introduce more notations. The first set of notations will show up in the calculation of $\Delta_d$, which is the mean value of $Q$. We set 
\begingroup
\allowdisplaybreaks
\begin{align}
&\mathfrak{d}_1^a:= \frac{2z}{n} \sum_{i,j} (\Pi_1)_{ii} (\Pi_1)_{j' j'} (\Pi_1 A_1)_{j' i},
\quad
\mathfrak{d}_2^a:=\frac{2z}{n} \sum_{i,j} (\Pi_1)_{ii} (\Pi_1)_{j' j'} (\Pi_1 A_2)_{i j'}, \nonumber \\
&\tilde{\mathfrak{d}}_1:= \frac{2z}{n} \sum_{(a_1,a_2,a_2)\in \mathcal P(2,1,1)} \sum_{i,j} (\Pi_{a_1})_{ii} (\Pi_{a_2})_{j' j'} (\Pi_{a_3} B_1)_{j' i}, \nonumber \\
&\tilde{\mathfrak{d}}_2:= \frac{2z}{n}  \sum_{(a_1,a_2,a_2)\in \mathcal P(2,1,1)} \sum_{i,j} (\Pi_{a_1})_{ii} (\Pi_{a_2})_{j' j'} (\Pi_{a_3} B_2)_{i j'}. \label{18091910}
\end{align}
\endgroup
And $\mathfrak{d}_1^b$ (resp. $\mathfrak{d}_2^b$) is defined by replacing $A_1$ (resp. $A_2$) to $B_1$ (resp. $B_2$) in the expression of $\mathfrak{d}_1^a$ (resp. $\mathfrak{d}_2^a$).  Using  (\ref{eq_diagonalmatrix}), we further set
\begingroup
\allowdisplaybreaks
\begin{align}
&\Pi_{3}^{\mathrm{u}}:= (m_1'' + \frac{1}{z} m_1' )I_M \oplus (m_2'' + \frac{2}{z} m_2') I_n,\nonumber\\
&\Pi_3^{\mathrm{l}}:=  (m_1'' + \frac{2}{z} m_1' ) I_M \oplus (m_2'' + \frac{1}{z} m_2') I_n, \nonumber\\
&\Pi_4^{\mathrm{u}}:= (\frac{2}{3}m_1^{(3)} + \frac{2}{z} m_1'' + \frac{1}{z^2} m_1') I_M \oplus (\frac{2}{3} m_2^{(3)} + \frac{2}{z} m_2'') I_n, \nonumber\\
&\Pi_4^{\mathrm{l}}:= (\frac{2}{3}m_1^{(3)} + \frac{2}{z} m_1'' ) I_M \oplus (\frac{2}{3} m_2^{(3)} + \frac{2}{z} m_2'' + \frac{1}{z^2} m_2') I_n. \label{18091566}
\end{align}
\endgroup
The next set of notations will appear in the derivation of the variance of $Q$. We denote
\begingroup
\allowdisplaybreaks
\begin{align}\label{def:missvar}
&\mathfrak{a}_{11}: = -(k-1)\sqrt{z} \Big( 2 \Tr (\Pi_2^{\mathrm{l}} - \Pi_{1,1}^{\mathrm{l}}) A_1 \Pi_1 A - \frac{1}{z} \Tr(\Pi_2^{\mathrm{l}} - \Pi_{1,1}^{\mathrm{l}}) A_1 \Pi_1 B \nonumber\\
&\hspace{3cm}+  \Tr (\Pi_{2}^{\mathrm{l}} - \Pi_{1,1}^{\mathrm{l}}) A_1 \Pi_{2} B +  \Tr (\Pi_{3}^{\mathrm{l}} - \Pi_{2,1}^{\mathrm{l}}) A_1 \Pi_{1} B \Big),\nonumber\\
&\tilde{\mathfrak{b}}_{11}: = -(k-1)\sqrt{z} \Big( 2\Tr (\Pi_3^{\mathrm{l}} - \Pi_{1,2}^{\mathrm{l}}) B_1 \Pi_1 A - \frac{1}{z} \Tr(\Pi_3^{\mathrm{l}} - \Pi_{1,2}^{\mathrm{l}}) B_1 \Pi_1 B \nonumber\\
&\hspace{3cm}+  \Tr (\Pi_{3}^{\mathrm{l}} - \Pi_{1,2}^{\mathrm{l}}) B_1 \Pi_{2} B +  \Tr (\Pi_{4}^{\mathrm{l}} - \Pi_{2,2}^{\mathrm{l}}) B_1 \Pi_{1} B \Big).
\end{align}
\endgroup
In addition,  $\mathfrak{a}_{12}$ is defined via replacing $A_1$ with $A_2$ and $\Pi_{a}^{\mathrm{l}}, \Pi_{a_1,a_2}^{\mathrm{l}}$ with $\Pi_{a}^{\mathrm{u}}, \Pi_{a_1,a_2}^{\mathrm{u}}$ in the  definition of $\mathfrak{a}_{11}$ . We further define $\mathfrak{b}_{11}$ (resp. $\mathfrak{b}_{12}$) via replacing $A_1$ (resp. $A_2$) with $B_1$ (resp. $B_2$) in the definition of  $\mathfrak{a}_{11}$ (resp. $\mathfrak{a}_{12}$). Similarly, $\tilde{\mathfrak{b}}_{12}$ is obtained by replacing $B_1$ with $B_2$ and $\Pi_{a}^{\mathrm{l}}, \Pi_{a_1,a_2}^{\mathrm{l}}$ with $\Pi_{a}^{\mathrm{u}}, \Pi_{a_1,a_2}^{\mathrm{u}}$ in the definition of $\tilde{\mathfrak{b}}_{11}$. 

Next, recall $c_{ij}$ defined in \eqref{defn_cij1} and set
\begin{align}
&\mathfrak{a}_{21}:= -\frac{ (k-1)z}{\sqrt{n}} \sum_{(i,j)\in \mathcal S(\nu)} (\Pi_1)_{j' j'} (\Pi_1 A_1)_{ii} c_{ij},\nonumber\\
&\tilde{\mathfrak{b}}_{21}:= -\frac{ (k-1)z}{\sqrt{n}} \sum_{(i,j)\in \mathcal S(\nu)} \big( (\Pi_1)_{j' j'} (\Pi_2 B_1)_{ii} + (\Pi_2)_{j' j'} (\Pi_1 B_1)_{ii}  \big) c_{ij}. \label{18091930}
\end{align}
Further,  $\mathfrak{a}_{22}$ (resp. $\tilde{\mathfrak{b}}_{22}$) is defined  by replacing $(A_1)_{ii}$ (resp. $(B_1)_{ii}$) with $(A_2)_{j' j'}$ (resp. $(B_2)_{j' j'}$) in the definition of $\mathfrak{a}_{21}$ (resp. $\tilde{\mathfrak{b}}_{21}$). Then we recall $s_{ij}$ in \eqref{defn_sij1} and set
\begin{align}
&\mathfrak{a}_{31}:=-\frac{2(k-1)z^{3/2}}{n} \sum_{i,j}  (\Pi_1)_{j'j'} (\Pi_1 A_1)_{ii} s_{ij}, \nonumber\\  
&\tilde{\mathfrak{b}}_{31}=-\frac{2(k-1)z^{3/2}}{n} \sum_{i,j} \big(  (\Pi_1)_{j'j'} (\Pi_2 B_1)_{ii} +  (\Pi_2)_{j'j'} (\Pi_1B_1)_{ii} \big) s_{ij}. \label{18091931}
\end{align}
Further,  $\mathfrak{a}_{32}$ (resp. $\tilde{\mathfrak{b}}_{32}$) is defined  via replacing $(A_1)_{ii}$ (resp. $(B_1)_{ii}$) with $(A_2)_{j' j'}$ (resp. $(B_2)_{j' j'}$) in the definition of the $\mathfrak{a}_{31}$ (resp. $\tilde{\mathfrak{b}}_{31}$). Also,  $\mathfrak{b}_{31}$ (resp. $\mathfrak{b}_{32}$) is defined by replacing $A_1$ (resp. $A_2$) with $B_1$ (resp. $B_2$) in the definition of  $\mathfrak{a}_{31}$ (resp. $\mathfrak{a}_{32}$).

For $\alpha=1,2$, we further write 
\begin{align}
\mathfrak{a}_{0 \alpha} &:=\mathfrak{a}_{1\alpha}+\kappa_3 \mathfrak{a}_{2\alpha}+\frac{\kappa_4}{2} \mathfrak{a}_{3\alpha},\nonumber\\
\mathfrak{b}_{0\alpha} &:=\frac{m_\alpha}{2} \tilde{\mathfrak{b}}_{1\alpha}+m_{\alpha}'\mathfrak{b}_{1\alpha}+\frac{\kappa_3 m_\alpha  }{2} \tilde{\mathfrak{b}}_{2\alpha}+\kappa_3  m_\alpha'  \mathfrak{b}_{2\alpha}+\frac{ \kappa_4m_\alpha }{4} \tilde{\mathfrak{b}}_{3\alpha}+\frac{ \kappa_4 m_\alpha' }{2} \mathfrak{b}_{3\alpha}. \label{180920150}
\end{align}
For brevity, we also adopt the notation 
\begin{align*}
\mathfrak{q}^{(l)}=Q^l(z)e^{\ii t \Delta(z_0)}. 
\end{align*}

Recall the notations in (\ref{18072570}). With the above notations, we now state  the following lemma. 
\begin{lem} \label{lem. recursive estimate for each term} Under the assumptions of Theorem \ref{thm_mainthm},  we have for $\alpha=1,2$,
\begin{align}
&\sqrt n \E  f_\alpha  \mathfrak{q}^{(k-1)}
=-\sqrt{z} m_{\alpha} \mathbb{E} \Big( \frac{\kappa_3}{2} \mathfrak{d}_\alpha^a \mathfrak{q}^{(k-1)} + \mathfrak{a}_{0\alpha} \mathfrak{q}^{(k-2)} \Big) +O_{\prec}(n^{-\frac12+4\nu}), \label{eq_parta1}  \\
&\sqrt{n}\mathbb{E} g_{\alpha} \mathfrak{q}^{(k-1)}= - \sqrt{z} \mathbb{E} \Big(\frac{ \kappa_3 }{4} \Big(m_{\alpha}\tilde{\mathfrak{d}}_\alpha+2m_{\alpha}' \mathfrak{d}_\alpha^b \Big)\mathfrak{q}^{(k-1)} + \mathfrak{b}_{0\alpha}\mathfrak{q}^{(k-2)}  \Big) +O_{\prec}(n^{-\frac12+4\nu}),  \label{eq_partb1} 
\end{align}
In addition, we also have
\begin{align}
& \sqrt{n z} \sum_{(i,j)\in \mathcal{S}(\nu)}  c_{ij}  \E  x_{ij} \mathfrak{q}^{(k-1)}=(k-1) \Big( z \sum_{(i,j)\in \mathcal{S}(\nu)} c^2_{ij}+  \frac{z^{\frac32} \kappa_3 }{\sqrt{n}} \sum_{(i,j)\in \mathcal{S}(\nu)} s_{ij} c_{ij} \Big) \E \mathfrak{q}^{(k-2)}\nonumber\\
&\hspace{35ex}+O_{\prec}(n^{-\frac12+4\nu}).  \label{180920101}
\end{align}
\end{lem}

With Lemmas \ref{lem.rewrite of Q} and \ref{lem. recursive estimate for each term}, we can now prove Proposition \ref{prop_iteration}.  

\begin{proof}[Proof of Proposition \ref{prop_iteration}]
By simply combining Lemma \ref{lem.rewrite of Q} and \ref{lem. recursive estimate for each term}, we can write 
\begin{align*}
\mathbb{E} \mathfrak{q}^{(k)}= \mathfrak{c}_1 \E \mathfrak{q}^{(k-1)}+ \mathfrak{c}_2 \E \mathfrak{q}^{(k-2)}-\Delta_d \E \mathfrak{q}^{(k-1)}+O_{\prec}(n^{-\frac12+4\nu}),
\end{align*}
where
\begin{align*}
\mathfrak{c}_1 &=- \sqrt{z} \kappa_3 \sum_{\alpha=1,2} \Big( \frac{1}{2}
m_\alpha \mathfrak{d}_\alpha^a+\frac{1}{4} m_\alpha \tilde{\mathfrak{d}}_\alpha+\frac{1}{2} m_\alpha'  \mathfrak{d}_\alpha^b \Big),\nonumber\\
\mathfrak{c}_2 &=-\sqrt{z} \sum_{\alpha=1,2} \Big( m_\alpha \mathfrak{a}_{1\alpha}+ \kappa_3 m_\alpha  \mathfrak{a}_{2\alpha}+\frac{\kappa_4m_\alpha }{2} \mathfrak{a}_{3\alpha}+\frac{m_\alpha}{2} \tilde{\mathfrak{b}}_{1\alpha}+m_\alpha' \mathfrak{b}_{1\alpha} \\
&\qquad\qquad+\frac{\kappa_3m_\alpha }{2} \tilde{\mathfrak{b}}_{2\alpha}+\kappa_3 m_\alpha' \mathfrak{b}_{2\alpha}+ \frac{\kappa_4 m_\alpha}{4} \tilde{\mathfrak{b}}_{3\alpha}+\frac{\kappa_4m_\alpha'  }{2}\mathfrak{b}_{3\alpha}  \Big) \\
& + z \sum_{(i,j)\in \mathcal{S}(\nu)} c^2_{ij}+  \frac{z^{\frac32} \kappa_3 }{\sqrt{n}} \sum_{(i,j)\in \mathcal{S}(\nu)} s_{ij} c_{ij}.
\end{align*}
Also recall $\Delta_d$ from  \eqref{defn_detministic1} and $V$ from (\ref{def:variance}).  By substituting the definitions of the notations in (\ref{18091910}), (\ref{def:missvar}), (\ref{18091930}), (\ref{18091931}), and also their analogues, it is elementary to check
\begin{align}
\mathfrak{c}_1=\Delta_d, \qquad \mathfrak{c}_2=V.  \label{18091950}
\end{align}
This completes the proof of (\ref{eq_qdeltarecusie}). Further we can regard (\ref{18092701}) as a degenerate case of (\ref{eq_qdeltarecusie}). The proof can be done in the same way. We thus conclude the proof of Proposition \ref{prop_iteration}. 
\end{proof}

Therefore, what remains is to prove Lemmas \ref{lem.rewrite of Q} and \ref{lem. recursive estimate for each term}. We prove Lemma \ref{lem.rewrite of Q} in the rest of this section, and state the proof of Lemma \ref{lem. recursive estimate for each term} in Section \ref{s. recursive estimate for each term}.

\begin{proof}[Proof of Lemma \ref{lem.rewrite of Q}]

Recall from (\ref{18072502}) and (\ref{18072501}) that
\begin{align}
Q=\mathcal{Q}-\Delta_r-\Delta_d. \label{180727100}
\end{align}
For brevity, we also write  
\begin{equation}\label{defn_f1f2}
F_{1}=1+zm_{1}, \ F_{2}=1+zm_{2}.
\end{equation}
By (\ref{eq_self1}) and (\ref{eq_self2}), it is easy to check that
\begin{equation} \label{eq_f1mm2}
F_1=-zm_1m_2,  \ F_2=-zym_1m_2.
\end{equation}

Note that by definition $\Tr G A = \Tr G A_1 + \Tr G A_2$ and $\Tr \Pi_1 A = m_{1} \Tr A_1 + m_{2} \Tr A_2$. Thus using \eqref{defn_f1f2}, we have
\begin{align}
\Tr \Xi_1 A &=  \Tr G A_1 + \Tr G A_2 - m_{1} \Tr A_1 - m_{2} \Tr A_2 \nonumber \\
&=-m_{1} \Tr H G A_1 - m_{2} \Tr H G A_2 + F_1 \Tr G A_1 + F_2 \Tr G A_2, \label{18072560}
\end{align}
where in the last step, we used  the fact $zG=HG-I$. 

Using (\ref{18072530}) and (\ref{18072531}), one can write
\begin{align*}
\Tr \Xi_1' B &= \frac{1}{2} \Tr G^2 B_1 + \frac{1}{2} \Tr G^2 B_2 - \frac{1}{2z} \Tr G B_1- \frac{1}{2z} \Tr G B_2- m_{1}' \Tr B_1 -m_{2}' \Tr B_2.
\end{align*}
By further using the identity $zG^2 = H G^2 -G$, it is not difficult to check
\begin{align}
\Tr \Xi_1' B
&= -\frac{m_{1}}{2} \Tr H G^2 B_1 + \frac{F_1}{2} \Tr G^2 B_1 +\frac{1}{2}(m_{1}-\frac{1}{z}) \Tr GB_1 -  m'_{1} \Tr B_1 \nonumber \\
&\quad -\frac{m_{2}}{2} \Tr H G^2 B_2 + \frac{F_2}{2} \Tr G^2 B_2 +\frac{1}{2}(m_{2}-\frac{1}{z}) \Tr GB_2 -  m'_{2} \Tr B_2.  \label{18072561}
\end{align}
Recall the definition (\ref{defn_q1}). Putting (\ref{18072560}) and (\ref{18072561}) together, we get
\begin{align} \label{eq_q1expansion}
\mathcal{Q}&=\sqrt{n} \Big( -m_{1} \Tr HGA_1+F_{1} \Tr GA_1-m_{2}\Tr HGA_2+F_{2} \Tr GA_2 \nonumber  \\
&\quad -\frac{m_{1}}{2} \Tr H G^2 B_1 + \frac{F_1}{2} \Tr G^2 B_1 +\frac{1}{2}(m_{1}-\frac{1}{z}) \Tr GB_1 -  m'_{1} \Tr B_1 \nonumber\\
&\quad -\frac{m_{2}}{2} \Tr H G^2 B_2 + \frac{F_2}{2} \Tr G^2 B_2 +\frac{1}{2}(m_{2}-\frac{1}{z}) \Tr GB_2 -  m'_{2} \Tr B_2 \Big).
\end{align}
Recall the definition of $\Delta_r$ from (\ref{18072410}). We write 
\begin{align*}
\Delta_r  =\sqrt{n z}\sum_{i,j} x_{ij} c_{ij}-\sqrt{n z}\sum_{(i,j)\in \mathcal{S}(\nu) } x_{ij} c_{ij}.
\end{align*}
Further recall the definition of $c_{ij} $ from \eqref{defn_cij1}. 
It is elementary to check that
\begin{align}
\sqrt{nz} \sum_{i,j} x_{ij} c_{ij} = -\sqrt n &\Big( m_{1} \text{Tr} H\Pi_1 A_1 + m_{2} \text{Tr} H\Pi_1 A_2 + \frac{m_{1}}{2} \Tr H \Pi_2 B_1 \nonumber\\
&\quad + \frac{m_{2} }{2}  \Tr H \Pi_2 B_2  + m_{1}' \Tr H\Pi_1 B_1 + m_{2}' \Tr H\Pi_1 B_2 \Big). \label{18072571}
\end{align}
Using (\ref{eq_q1expansion}) and (\ref{18072571}), with the notations defined in  (\ref{18072570}), we can write
\begin{align}
\mathcal{Q} - \sqrt{nz} \sum_{i,j} x_{ij} c_{ij}=\sqrt n \big(f_1 + f_2 + g_1 + g_2 \big). \label{18091960}
\end{align}
Combining (\ref{18072410}), (\ref{180727100}) and (\ref{18091960}) we can conclude the proof. 
\end{proof}

\section{Proof of Lemma \ref{lem. recursive estimate for each term}}\label{s. recursive estimate for each term}

To prove Lemma \ref{lem. recursive estimate for each term}, we need the 
 following lemma summarizing  some estimates on the derivative of $Q$ w.r.t $x_{ij}$'s, which will be frequently used in the subsequent discussion.  We first write $\frac{\partial Q}{\partial x_{ij}}$ in terms of Green functions.  Recall the definition of $Q$ in (\ref{18072501}) that
 \begin{align*}
 Q= \sqrt{n} \Big( \Tr \big(\Xi_1A\big)+ \Tr \big(\Xi'_1B\big) \Big) - \sqrt{nz}\sum_{(i,j)\in \mathcal{B}(\nu) } x_{ij} c_{ij} - \Delta_d,
 \end{align*}
where $\Xi_1=G-\Pi_1$ and $\Delta_d$ is a deterministic quantity in \eqref{defn_detministic1}. 
Using $G'=\frac{1}{2}(G^2 - z^{-1} G)$ in Lemma \ref{lem.18091410}, 
we find that  
\begin{equation*}
\frac{\partial Q}{\partial x_{ij}}=\sqrt{n} \left( \Tr \frac{\partial G}{\partial x_{ij}} A+\frac{1}{2} \Tr \Big( \frac{\partial G^2}{\partial x_{ij}}B-z^{-1}\frac{\partial G}{\partial x_{ij}}B \Big) \right)-\mathbf{1}\Big( (i,j) \in \mathcal{B}(\nu) \Big)\sqrt{nz} c_{ij}. 
\end{equation*}
By Lemma \ref{lem_greenderivative}, it can be further seen that
\begin{align}\label{eq_parti_1}
\frac{\partial Q}{\partial x_{ij}}&=-\sqrt{nz} \sum_{\substack{ l_1,l_2\in \{ i, j'\} \\ l_1\neq l_2 }}\Big( (GAG)_{l_1 l_2} -\frac{1}{2z} (GBG)_{l_1 l_2}  + \frac{1}{2}  (G B G^{2})_{l_1 l_2}+  \frac{1}{2}  (G^2 B G)_{l_1 l_2}\Big)\nonumber\\
&\quad -\mathbf{1}\Big( (i,j) \in \mathcal{B}(\nu) \Big)\sqrt{nz} c_{ij}.
\end{align}

\begin{lem} \label{lem. deri of Q} Under the assumptions of Proposition \ref{lem:main}, we have 
\begin{align}\label{eq:derivativeQ1}
\frac{\partial Q}{\partial x_{ij} }& =\sqrt{n z} \mathbf{1} \Big((i,j) \in \mathcal{S}(\nu) \Big)c_{ij} + O_{\prec}(1).
\end{align}
Consequently, we have the bounds
\begin{align}
\frac{\partial Q}{\partial x_{ij}}=
\left\{
\begin{array}{lll}
O_{\prec}(1), & \forall \ (i,j) \in \mathcal{B}(\nu)   \\\\
O_{\prec}(n^{\frac12-\nu}), & \forall  (i,j) \in \mathcal{S}(\nu).
\end{array}
\right. \label{18072549} 
\end{align}
\end{lem}

\begin{proof}[Proof of Lemma \ref{lem. deri of Q}] First, recall the definitions in (\ref{18072531}) and (\ref{18072530}). By (\ref{eq_localoutside}),
we have that for $a_1, a_2=1,2$,
\begin{equation*}
(G^{a_1}AG^{a_2})_{l_1 l_2}=(\Pi_{a_1} A \Pi_{a_2})_{l_1 l_2}+O_{\prec}(n^{-\frac12}). 
\end{equation*}
Applying the above estimates to (\ref{eq_parti_1}), we find that 
\begin{align}
\frac{\partial Q}{\partial x_{ij}}
&= -\sqrt{z n }  \sum_{\substack{ l_1,l_2\in \{ i, j'\} \\ l_1\neq l_2 }}\Big( (\Pi_1 A \Pi_1)_{l_1 l_2} -\frac{1}{2z} (\Pi_1 B \Pi_1)_{l_1 l_2}+ \frac{1}{2}   (\Pi_{1} B \Pi_{2})_{l_1 l_2}+\frac{1}{2}   (\Pi_{2} B \Pi_{1})_{l_1 l_2} \Big) \nonumber\\
& \quad -\mathbf{1} \Big((i,j) \in \mathcal{B}(\nu) \Big)\sqrt{n z} c_{ij}+O_\prec(1). \label{18091570}
\end{align}
Comparing \eqref{18091570} with the definition of $c_{ij}$ in \eqref{defn_cij1}, we prove  (\ref{eq:derivativeQ1}) and the first case of (\ref{18072549}).

Next, by the definitions of $A,B$ in \eqref{eq:new-ab} and the set $\mathcal S(\nu)$ in (\ref{18072542}), it follows immediately that there exists some constant $C>0,$ such that 
\begin{equation*}
 |A_{ij^{\prime}}| \leq C n^{-\nu},\quad |B_{ij^{\prime}}| \leq C n^{-\nu}, \qquad \forall  (i,j) \in \mathcal{S}(\nu). 
\end{equation*}
By the estimates in (\ref{eq_boundm1m2}),  we get the second case of (\ref{18072549}). This concludes the proof of Lemma \ref{lem. deri of Q}. 
 \end{proof}
The remaining of the section is devoted to the proof of Lemma \ref{lem. recursive estimate for each term}. 
\begin{proof}[Proof of Lemma \ref{lem. recursive estimate for each term}] We will focus on  the proof of (\ref{eq_parta1}). Since the proof of (\ref{eq_partb1}) is analogous, we shall only outline the main steps.   Recall from the definition in (\ref{18072570}) and (\ref{defn_f1f2}) that
\begin{align}\label{eq_defnfa1}
\sqrt{n}  \E f_1 \mathfrak{q}^{(k-1)}=\E \Big(-m_{1} \sqrt{zn} \sum_{i,j} x_{ij} (\Xi_1A_1)_{j^{\prime}i}+\sqrt{n} F_1 \text{Tr}GA_1 \Big) \mathfrak{q}^{(k-1)}.
\end{align}  
 For brevity, we use the notations
\begin{equation}\label{eq_defnfunction}
h_1=\big(\Xi_1 A_1 \big)_{j^{\prime}i}, \quad h_2=Q^{k-1}, \quad h_3=e^{\ii t \Delta}.
\end{equation}
Note that $h_1$ actually depends on the index $(j^{\prime},i)$. However, we drop this dependence from notation for brevity. 
By Lemma \ref{lem_cumu}, one has 
\begin{align}\label{eq:A1main}
&\sqrt{n} \sum_{i,j} \E  x_{ij} (\Xi_1A_1)_{j^{\prime}i} \mathfrak{q}^{(k-1)} =\sqrt{n} \sum_{i,j} \E  x_{ij} (h_1 h_2 h_3)\nonumber\\
&\quad  = \sum_{l=1}^3  \frac{\kappa_{l+1}}{l! n^{l/2}} \sum_{i,j}\E \Big( \frac{\partial^l}{\partial x_{ij}^l} (h_1 h_2 h_3) \Big)+\mathbb{E}\mathcal{R}_1, 
\end{align}
where $\mathcal{R}_1$ satisfies that, for any sufficiently small $\epsilon>0$ and sufficiently large $K>0$, 
\begin{equation}\label{eq:errorA1}
|\E \mathcal{R}_1| \le  \sum_{i,j} \E \left( n^{-2} \sup_{|x_{ij}|\leq n^{-\frac12+\epsilon} } \bigg| \frac{\partial^4}{\partial x_{ij}^4} (h_1 h_2 h_3) \bigg|+ n^{-K}\sup_{x_{ij} \in \R}  \bigg| \frac{\partial^4}{\partial x_{ij}^4} (h_1 h_2 h_3) \bigg| \right). 
\end{equation}

Here we used the assumption that $\mathbb{E}|\sqrt{n} x_{ij}|^p \leq C_p$ for all $p\ge 3$. Therefore, the main technical estimates are the first four  derivatives of $h_1 h_2 h_3.$  By product rule, for each $l \in \N$, we have
\begin{align}\label{eq_chainruleproduct}
 \frac{\partial^l}{\partial x_{ij}^l} (h_1 h_2 h_3) = \sum_{l_1+l_2+l_3=l} {l \choose l_1, l_2, l_3} \frac{\partial^{l_1} h_1}{\partial x_{ij}^{l_1}} \frac{\partial^{l_2} h_2}{\partial x_{ij}^{l_2}} \frac{\partial^{l_3} h_3}{\partial x_{ij}^{l_3}}.  
\end{align}
First, it is elementary to verify 
\begin{equation}\label{eq:derivativeh3}
\frac{\partial^l h_3}{\partial x_{ij}^l}= \mathbf{1}((i,j) \in \mathcal{B}\left(\nu)\right) \Big( \ii t\sqrt{n z} c_{ij} \Big)^l e^{\ii t \Delta},
\end{equation}
and 
\begin{equation*}
\frac{\partial^l h_1 }{\partial x^l_{ij}}= \Big( \frac{\partial^l G}{\partial x_{ij}^l} A_1 \Big)_{j^{\prime} i}.
\end{equation*}
The derivatives of $h_2$ can be computed using Fa\`a di Bruno's formulas. For the reader's convenience, we list them here.   The first derivative of $h_2$ is 
\begin{equation*}
\frac{\partial h_2}{\partial x_{ij}}=(k-1) \frac{\partial Q}{\partial x_{ij}} Q^{k-2}.
\end{equation*}
The second derivative of $h_2$ is 
\begin{align*}
\frac{\partial^2 h_2}{\partial x_{ij}^2}=\frac{(k-1)!}{(k-3)!} Q^{k-3} \Big(\frac{\partial Q}{\partial x_{ij}} \Big)^2+(k-1) Q^{k-2} \frac{\partial^2 Q}{\partial x_{ij}^2}.
\end{align*}
The third derivative of $h_2$ is 
\begin{equation*}
\frac{\partial^3 h_2}{\partial x_{ij}^3}=\frac{(k-1)!}{(k-4)!} Q^{k-4} \Big(\frac{\partial Q}{\partial x_{ij}} \Big)^3+3\frac{(k-1)!}{(k-3)!}Q^{k-3} \frac{\partial Q}{\partial x_{ij}} \frac{\partial^2 Q}{\partial x_{ij}^2}+(k-1)Q^{k-2} \frac{\partial^3 Q}{\partial x_{ij}^3}.
\end{equation*}
The fourth derivative of $h_2$ is 
\begin{align*}
\frac{\partial^4 h_2}{\partial x_{ij}^4}&=\frac{(k-1)!}{(k-5)!} Q^{k-5} \Big(\frac{\partial Q}{\partial x_{ij}} \Big)^4+6\frac{(k-1)!}{(k-4)!} Q^{k-4} \Big(\frac{\partial Q}{\partial x_{ij}} \Big)^2 \frac{\partial^2 Q}{\partial x_{ij}^2} \\
&\quad+\frac{(k-1)!}{(k-3)!} Q^{k-3} \Big(4 \frac{\partial Q}{\partial x_{ij}} \frac{\partial^3 Q}{\partial x_{ij}^3}+3 \Big( \frac{\partial^2 Q}{\partial x_{ij}^2} \Big)^2 \Big)+(k-1)Q^{k-2} \frac{\partial Q^4}{\partial x_{ij}^4}.
\end{align*}
 As we can see from the above identities, the key ingredients are the partial derivatives of $Q$ and $GA_1.$ 

For brevity, we introduce the notation
\begin{align}
\hslash(l_1,l_2,l_3):=n^{-\frac{l_1+l_2+l_3}{2}}\sum_{i,j}\frac{\partial^{l_1} h_1}{\partial x_{ij}^{l_1}} \frac{\partial^{l_2} h_2}{\partial x_{ij}^{l_2}} \frac{\partial^{l_3} h_3}{\partial x_{ij}^{l_3}}.  \label{18091411} 
\end{align}
In the following two lemmas, we summarize the estimates of $\hslash(l_1,l_2,l_3)$ for $l_1+l_2+l_3\leq 4$. The proofs of the two lemmas will be given in Sections  \ref{s.18092001} and \ref{s.18092002}. 
\begin{lem}\label{lem_estimateh1h2h3} 
For the first derivative of $h_1h_2h_3$,  we have that 
\begingroup
\allowdisplaybreaks
\begin{align}
&\hslash(1,0,0)=-\sqrt{nz} m_2 \Tr (GA_1) \mathfrak{q}^{(k-1)} +O_{\prec}(n^{-\frac12}), \label{eq_parth1h2h3}\\
&\hslash(0,1,0)=\mathfrak{a}_{11} \mathfrak{q}^{(k-2)} +O_{\prec}(n^{-\frac12+4\nu}),\label{eq_h1partih2h3}\\
&\hslash(0,0,1) =O_{\prec}(n^{-\frac12+4\nu}).\label{eq_h1h2parih3}
\end{align}
\endgroup
\end{lem} 
\begin{lem}\label{lem_24h1h2h3} On higher order derivatives of $h_1h_2h_3$, we have the following estimates.

\noindent(1). For the second  derivative, we have  
\begingroup
\allowdisplaybreaks
\begin{align}
&\hslash(2,0,0)=\mathfrak{d}_1^a \mathfrak{q}^{(k-1)}+O_{\prec}(n^{-\frac12}), \label{eq_h200}\\
&\hslash(1,1,0)=\mathfrak{a}_{21} \mathfrak{q}^{(k-2)}+O_{\prec}(n^{-\frac12+4\nu}), \label{eq_h110}\\
&\hslash(1,0,1)=O_{\prec}(n^{-\frac12+4\nu}),
\quad 
\hslash(0,2,0)=O_{\prec}(n^{-\frac12}), \nonumber \\
&\hslash(0,1,1)=O_{\prec}(n^{-\frac12+4\nu}), \quad \hslash(0,0,2)=O_{\prec}(n^{-1+4\nu}). \nonumber 
\end{align}
\endgroup
(2). For the third derivative, we have 
\begingroup
\allowdisplaybreaks
\begin{align}
&\hslash(1,2,0)=\mathfrak{a}_{31}\mathfrak{q}^{(k-2)}+O_{\prec}(n^{-\frac12}), \label{eq_120}\\
&\hslash(3,0,0)=O_{\prec}(n^{-\frac12}), 
\qquad\quad \;
\hslash(0,3,0)=O_{\prec}(n^{-\frac12}),\nonumber \\
&\hslash(2,1,0)=O_{\prec}(n^{-1}),
\qquad\qquad
\hslash(2,0,1) =O_{\prec}(n^{-\frac32+4\nu}), \nonumber \\
&\hslash(1,1,1)=O_{\prec}(n^{-1+4\nu}), 
\qquad \; \; 
\hslash(1,0,2)=O_{\prec}(n^{-\frac32+4\nu}), \nonumber \\
&\hslash(0,2,1)=O_{\prec}(n^{-\frac12+4\nu}). \nonumber
\end{align}
\endgroup 
(3). For the fourth  derivative, all the terms in the RHS of (\ref{eq_chainruleproduct}) can be bounded by $O_{\prec}(n^{-\frac12+4\nu}).$
\end{lem}

By Lemma \ref{lem_estimateh1h2h3} and Lemma \ref{lem_24h1h2h3},  the first term in \eqref{eq:A1main} is estimated by 
\begin{align*}
&\sum_{l=1}^3  \frac{\kappa_{l+1}}{l! n^{l/2}} \sum_{i,j}\E \Big( \frac{\partial^l}{\partial x_{ij}^l} (h_1 h_2 h_3) \Big) = \sum_{l=1}^3  \sum_{i,j} \sum_{l_1 + l_2 +l_3=l} \frac{\kappa_{l+1}}{l_1! l_2! l_3!} \hslash(l_1,l_2,l_3)\\
&=-\sqrt{nz} m_2 \Tr (G A_1) \mathfrak{q}^{(k-1)} + \frac{\kappa_3}{2}\mathfrak{d}_1^a \mathfrak{q}^{(k-1)}+ \Big(\mathfrak{a}_{11}+\kappa_3 \mathfrak{a}_{21}+\frac{\kappa_4}{2} \mathfrak{a}_{31}\Big) \mathfrak{q}^{(k-2)}.
\end{align*}
For the second term in \eqref{eq:A1main}, we claim that
\begin{equation}\label{eq_a1final}
|\mathbb{E} \mathcal{R}_1| \leq n^{-1/2+4\nu}.
\end{equation}
To prove \eqref{eq_a1final}, it is enough to bound the two terms on the right hand side of \eqref{eq:errorA1}. We apply Lemma \ref{lem_24h1h2h3} to the first term on the right hand side of \eqref{eq:errorA1} to get
\begin{align*}
 \sum_{i,j} \E n^{-2} \sup_{|x_{ij}|\leq n^{-\frac12+\epsilon} } \bigg| \frac{\partial^4}{\partial x_{ij}^4} (h_1 h_2 h_3) \bigg| \le n^{-1/2+4\nu}.
\end{align*}
A minor issue with the above step is that Lemma  \ref{lem_24h1h2h3} is proved for the matrix $X$ with all entries random variables. In our application of Lemma \ref{lem_24h1h2h3}, for each pair of fixed indices $(i,j)$, we actually consider a random matrix $X$ whose $(i,j)$th entry is a deterministic number with small magnitude and all the others random variables. However, this can be justified by a perturbation argument with the aid of resolvent expansion. Indeed, replacing one random entry $x_{ij}$ by any deterministic number bounded by $n^{-1/2+\varepsilon}$ and keeping the other $X$ entries random will not change the isotropic local law. Thus Lemma \ref{lem_24h1h2h3} holds for such random matrix $X$. 

For the second term on the right hand side of \eqref{eq:errorA1}, we use the trivial bounds for $G$ and its derivatives to obtain 
$$\sum_{i,j} \E n^{-K}\sup_{x_{ij} \in \R}  \bigg| \frac{\partial^4}{\partial x_{ij}^4} (h_1 h_2 h_3) \bigg| \le n^{-K+2+C}$$ for a positive constant $C$. By taking $K$ sufficiently large, we conclude \eqref{eq_a1final}. 

Plugging \eqref{eq:A1main} into \eqref{eq_defnfa1}, we finally get 
\begin{align*}
\sqrt{n} \mathbb{E} f_1 \mathfrak{q}^{(k-1)}
=-m_1\sqrt{z} \mathbb{E} \Big(  \frac{\kappa_3}{2}\mathfrak{d}_1^a \mathfrak{q}^{(k-1)}+ \Big(\mathfrak{a}_{11}+\kappa_3 \mathfrak{a}_{21}+\frac{\kappa_4}{2} \mathfrak{a}_{31}\Big) \mathfrak{q}^{(k-2)} \Big) 
 +O_{\prec}(n^{-\frac12+4\nu}).
\end{align*}
Note that by (\ref{eq_f1mm2}), the term $\sqrt{n} z m_1 m_2 \Tr G A_1 \mathfrak{q}^{(k-1)}$ is cancelled with $F_1\text{Tr} GA_1\mathfrak{q}^{(k-1)}$ in (\ref{eq_defnfa1}). This verifies (\ref{eq_parta1}) in case of $\alpha=1$ by recalling the definition in (\ref{180920150}).

Next, we turn to \eqref{eq_partb1} for $\alpha=1$.  Recall the definition of $g_1$ in (\ref{18072570}). We have
\begin{align}\label{eq_defnb1}
\sqrt{n} \mathbb{E} g_1 \mathfrak{q}^{(k-1)}&=\sqrt{n}\mathbb{E} \Big(-\frac{m_1}{2} \sqrt{z} \sum_{i,j} x_{ij} (\Xi_2 B_1)_{j'i} +\frac{F_1}{2} \text{Tr} G^2 B_1 \nonumber \\
&\quad+\frac{zm_1-1}{2z} \text{Tr}GB_1-m_1' \text{Tr} B_1+m_1' \text{Tr} H\Pi_1 B_1  \Big) \mathfrak{q}^{(k-1)}.
\end{align}
The main task is to estimate the cumulant expansion of the term 
$$\sqrt{n}\sum_{i,j} \mathbb{E} x_{ij} (\Xi_2B_1)_{j'i} \mathfrak{q}^{(k-1)},$$ which is analogous to \eqref{eq:A1main}. Recall $h_2$ and $h_3$ in (\ref{eq_defnfunction}) and denote  
 \begin{equation}\label{eq_htitle}
\tilde h_1=\big(\Xi_2 B_1 \big)_{j^{\prime}i}.
\end{equation}
Note that $\tilde h_1$ depends on the indices $i,j$. However, we drop these dependence from the notation for brevity.  Similarly to (\ref{18091411}), we introduce the notation 
\begin{align}
\tilde{\hslash}(l_1,l_2,l_3):=n^{-\frac{l_1+l_2+l_3}{2}}\sum_{i,j}\frac{\partial^{l_1} \tilde{h}_1}{\partial x_{ij}^{l_1}} \frac{\partial^{l_2} h_2}{\partial x_{ij}^{l_2}} \frac{\partial^{l_3} h_3}{\partial x_{ij}^{l_3}}. \label{18091420}
\end{align}
We collect the estimates of $\tilde{\hslash}(l_1,l_2,l_3)$ for $l_1+l_2+l_3 \le 4$ in the following two lemmas, whose proofs are postponed to Section \ref{subsec:proofofotherlemma}.
\begin{lem}\label{lem_bestimate} 
For the first derivative of $\tilde h_1h_2h_3$, we have 
\begingroup
\allowdisplaybreaks
\begin{align*}
&\tilde{\hslash}(1,0,0)=-\sqrt{nz} \big((2m_2'+\frac{m_2}{z})\operatorname{Tr}(GB_1)+m_2 \operatorname{Tr}(G^2B_1) \big) \mathfrak{q}^{(k-1)}+O_{\prec}(n^{-\frac12}),\\
&\tilde{\hslash}(0,1,0)= \tilde{\mathfrak{b}}_{11} \mathfrak{q}^{(k-2)}+O_{\prec}(n^{-\frac12+4\nu}),\\
&\tilde{\hslash}(0,0,1) =O_{\prec}(n^{-\frac12+4\nu}).
\end{align*}
\endgroup
\end{lem}

\begin{lem}\label{lem_bestimate2} For higher order derivatives of $\tilde h_1h_2h_3$, we have the following estimates. 

\noindent(1). For the second  derivative, we have  
\begingroup
\allowdisplaybreaks
\begin{align*}
&\tilde{\hslash}(2,0,0)
=\tilde{\mathfrak{d}}_1 \mathfrak{q}^{(k-1)} +O_{\prec}(n^{-\frac12}),\\
&\tilde{\hslash}(1,1,0)=\tilde{\mathfrak{b}}_{21} \mathfrak{q}^{(k-2)}+O_{\prec}(n^{-\frac12}),\\
&\tilde{\hslash}(0,2,0)=O_{\prec}(n^{-\frac12}).
\end{align*}
\endgroup
All the other terms with $l_3\geq 1$ can be bounded by $O_{\prec}(n^{-\frac12+4\nu}).$

\noindent(2). For the third  derivative, we have 
\begingroup
\allowdisplaybreaks
\begin{align*}
&\tilde{\hslash}(1,2,0) = \tilde{\mathfrak{b}}_{31} \mathfrak{q}^{(k-2)}+O_{\prec}(n^{-\frac12}),\\
& \tilde{\hslash}(3,0,0)=O_{\prec}(n^{-\frac12}), 
\quad 
\tilde{\hslash}(0,3,0)=O_{\prec}(n^{-\frac12}), \\
&\tilde{\hslash}(2,1,0)=O_{\prec}(n^{-1}).
\end{align*}
\endgroup
All the other terms with $l_3\geq 1$ can be bounded by  $O_{\prec}(n^{-\frac12+4\nu}).$

\noindent(3).  For the fourth  derivative, all the terms can be bounded by $O_{\prec}(n^{-\frac12+4\nu}).$
\end{lem}

With these preparations, using arguments similar to those of \eqref{eq_defnfa1}, we find that 
\begin{align*}
 \sqrt{n}\mathbb{E} g_1 \mathfrak{q}^{(k-1)}
&=-\frac{m_1 \sqrt{z}}{2} \mathbb{E} \Big(  \frac{\kappa_3}{2}\tilde{\mathfrak{d}}_1 \mathfrak{q}^{(k-1)}+\Big( \tilde{\mathfrak{b}}_{11}+\kappa_3\tilde{\mathfrak{b}}_{21}+\frac{\kappa_4}{2}\tilde{\mathfrak{b}}_{31} \Big) \mathfrak{q}^{(k-2)}  \Big)  \\
&\quad+\sqrt{n} \mathbb{E} \Big( \frac{m_1'}{m_1} \text{Tr} GB_1-m_1' \text{Tr} B_1 +m_1' \text{Tr} H\Pi_1 B_1 \Big) \mathfrak{q}^{(k-1)}+O_{\prec}(n^{-\frac12+4\nu}).
\end{align*}
In the above, we use  (\ref{eq_f1mm2}) and an identity
\begin{align}
\frac{m_1}{2} \Big(zm_1 (2m_2'+\frac{m_2}{z})+m_1-\frac{1}{z} \Big)=m_1', \label{180920100}
\end{align}
which can be checked from  (\ref{eq_self1}) and (\ref{eq_self2}). 
Next, observe that
\begin{align*}
& \sqrt{n} \mathbb{E} \Big( \frac{m_1'}{m_1} \text{Tr} GB_1+m_1' \text{Tr} H\Pi_1 B_1-m_1' \text{Tr} B_1 \Big) \mathfrak{q}^{(k-1)} \\
&= \sqrt{n} \mathbb{E} \Big(-zm_1' \text{Tr} GB_1+\frac{m_1' F_1}{m_1} \text{Tr} GB_1 +m_1' \text{Tr} H \Pi_1 B_1-m_1' \text{Tr} B_1 \Big) \mathfrak{q}^{(k-1)} , \\
&=\frac{m_1'}{m_1} \sqrt{n} \mathbb{E} \Big( -m_1 \text{Tr} H\Xi_1B_1+F_1 \text{Tr} GB_1 \Big) \mathfrak{q}^{(k-1)}.
\end{align*}
In the first step above, we simply use the definition of $F_1$ in (\ref{defn_f1f2}). In the second step, we use the fact $zG=HG-I$. Note that the remaining derivation can be done via replacing $A_1$ with $B_1$ (mutatis mutandis)  in the counterpart for $f_1$. Therefore, we finally get 
\begin{align*}
 \sqrt{n}\mathbb{E} g_1\mathfrak{q}^{(k-1)}
&=- \sqrt{z} \mathbb{E} \Big( \frac{m_1 \kappa_3 }{4} \Big(\tilde{\mathfrak{d}}_1+2\frac{m_1'}{m_1} \mathfrak{d}_1^b \Big)\mathfrak{q}^{(k-1)} +\Big(\frac{m_1}{2} \tilde{\mathfrak{b}}_{11} + m_1'\mathfrak{b}_{11}\Big) \mathfrak{q}^{(k-2)}\\
&\qquad+ \Big(\frac{m_1  \kappa_3}{2} \tilde{\mathfrak{b}}_{21}+ \kappa_3  m_1'  \mathfrak{b}_{21} + \frac{m_1 \kappa_4}{4} \tilde{\mathfrak{b}}_{31} + \frac{\kappa_4 m_1' }{2} \mathfrak{b}_{31} \Big) \mathfrak{q}^{(k-2)}  \Big)+O_{\prec}(n^{-\frac12+4\nu}).  
\end{align*}
This verifies (\ref{eq_partb1}) in case of $\alpha=1$ by recalling the definition in (\ref{180920150}).

The proofs of (\ref{eq_parta1}) and (\ref{eq_partb1}) in case of $\alpha=2$  are analogous to those of (\ref{eq_defnfa1}) and  (\ref{eq_defnb1}). We outline the main steps. 
First observe that 
\begin{align*}
\sqrt{n} \mathbb{E} f_2\mathfrak{q}^{(k-1)} &=\mathbb{E} \Big(-m_2 \sqrt{nz} \sum_{i,j} x_{ij} (\Xi_1A_2)_{ij'}+F_2 \text{Tr} GA_2\Big) \mathfrak{q}^{(k-1)}, \nonumber\\
\sqrt{n} \mathbb{E} g_2 \mathfrak{q}^{(k-1)}&=\mathbb{E} \Big(-\frac{m_2}{2} \sqrt{nz} \sum_{i,j} x_{ij} (\Xi_2 B_2)_{ij'} +\frac{F_2}{2} \text{Tr} G^2 B_2  \nonumber\\
&\qquad+\frac{zm_2-1}{2z} \text{Tr}GB_2-m_1' \text{Tr} B_2+m_1' \text{Tr} H\Pi_1 B_2  \Big) \mathfrak{q}^{(k-1)}.
\end{align*}
Recall $h_2$ and $h_3$ in (\ref{eq_defnfunction}) and denote 
\begin{equation*}
\mathbbm{h}_1=(\Xi_1A_2)_{ij'}, \quad \tilde{\mathbbm{h}}_1=(\Xi_2B_2)_{ij'}. 
\end{equation*} 
 Analogously to (\ref{18091411}) and (\ref{18091420}), we introduce the notations
\begin{align*}
&\mathfrak{h}(l_1,l_2,l_3):=n^{-\frac{l_1+l_2+l_3}{2}}\sum_{i,j}\frac{\partial^{l_1} \mathbbm{h}_1}{\partial x_{ij}^{l_1}} \frac{\partial^{l_2} h_2}{\partial x_{ij}^{l_2}} \frac{\partial^{l_3} h_3}{\partial x_{ij}^{l_3}},
\end{align*}
and $\tilde{\mathfrak{h}}(l_1,l_2,l_3)$ which is defined via replacing $\mathbbm{h}_1$ by $\tilde{\mathbbm{h}}_1$ in the above definition. 

Then we have the estimates for the first order derivatives involving $\mathbbm{h}_1$ and $\tilde{\mathbbm{h}}_1.$
\begin{lem} \label{lem_a2b2estimate} For ${\mathfrak{h}}$,  we have  
\begingroup
\allowdisplaybreaks
\begin{align}
&\mathfrak{h}(1,0,0)= -\sqrt{nz} y m_1 (\operatorname{Tr} GA_2) \mathfrak{q}^{(k-1)} +O_{\prec}(n^{-\frac12}),\nonumber\\
&\mathfrak{h}(0,1,0)=	\mathfrak{a}_{12} \mathfrak{q}^{(k-2)}+O_{\prec}(n^{-\frac12+4\nu}), \label{eq_partifrakh1h2h3}\\
&\mathfrak{h}(0,0,1) =O_{\prec}(n^{-\frac12+4\nu}). \nonumber
\end{align}
\endgroup
Similarly, for $\tilde{\mathfrak{h}}$, we have 
\begin{align*}
&\tilde{\mathfrak{h}}(1,0,0)=-\sqrt{nz} y \Big((2m_1'+\frac{m_1}{z}) \Tr GB_2+m_1 \operatorname{Tr} G^2 B_2 \Big) \mathfrak{q}^{(k-1)}+O_{\prec}(n^{-\frac12}),\nonumber\\
&\tilde{\mathfrak{h}}(0,1,0)=\tilde{\mathfrak{b}}_{12}\mathfrak{q}^{(k-2)} +O_{\prec}(n^{-\frac12+4\nu}),\\
&\tilde{\mathfrak{h}}(0,0,1) =O_{\prec}(n^{-\frac12+4\nu}).
\end{align*}
\end{lem}
For the higher order derivatives, we have the following lemma. 
\begin{lem} \label{lem_a2b2estimatehigh} We have the following estimates in case $l_1+l_2+l_3\geq 2$. 

\noindent(1). For ${\mathfrak{h}}(l_1,l_2,l_3)$,  we have 
\begingroup
\allowdisplaybreaks
\begin{align*}
&{\mathfrak{h}}(2,0,0)=\mathfrak{d}_2^a \mathfrak{q}^{(k-1)} +O_{\prec}(n^{-\frac12}),\\
&{\mathfrak{h}}(1,1,0)=\mathfrak{a}_{22} \mathfrak{q}^{(k-2)}+O_{\prec}(n^{-\frac12+4\nu}),\\
&{\mathfrak{h}}(1,2,0)=\mathfrak{a}_{32} \mathfrak{q}^{(k-2)}+O_{\prec}(n^{-\frac12}).
\end{align*}
\endgroup
All the other terms with $l_1+l_2+l_3\geq 2$ can be bounded by $O_{\prec}(n^{-\frac12+4\nu}).$ 

\noindent(2). For  $\tilde{\mathfrak{h}}(l_1,l_2,l_3)$ we have 
\begingroup
\allowdisplaybreaks
\begin{align*}
& \tilde{\mathfrak{h}}(2,0,0)=\tilde{\mathfrak{d}}_2 \mathfrak{q}^{(k-1)}+O_{\prec}(n^{-\frac12}),\\
&\tilde{\mathfrak{h}}(1,1,0)=\tilde{\mathfrak{b}}_{22} \mathfrak{q}^{(k-2)}+O_{\prec}(n^{-\frac12}),\\
&\tilde{\mathfrak{h}}(1,2,0) =\tilde{\mathfrak{b}}_{32} \mathfrak{q}^{(k-2)}+O_{\prec}(n^{-\frac12}).
\end{align*}
\endgroup
All the other terms with $l_1+l_2+l_3\geq 2$ can be bounded by $O_{\prec}(n^{-\frac12+4\nu}).$
\end{lem}

The proofs of the above lemmas will be given in Section \ref{subsec:proofofotherlemma}. The remaining estimates for $\sqrt{n} \mathbb{E} f_2 \mathfrak{q}^{(k-1)}$ and $\sqrt{n} \mathbb{E} g_2 \mathfrak{q}^{(k-1)}$ follow the same arguments as those of (\ref{eq_defnfa1}) and  (\ref{eq_defnb1}), and are therefore omitted. As a side note, we mention an identity (comparable to (\ref{180920100})) 
\begin{equation*}
\frac{m_2}{2} \Big(z y {m_2}(2m_1'+\frac{m_1}{z})+m_2-\frac{1}{z} \Big)=m_2'
\end{equation*}
used in the derivation of the $g_2$ term.

Lastly, we prove (\ref{180920101}).  Recall $h_2=Q^{k-1}$ and $h_3=e^{\ii t \Delta}$. By Lemma \ref{lem_cumu}, we have 
\begin{align}
\sqrt{nz} \sum_{(i,j) \in \mathcal{S}(\nu)} c_{ij} \mathbb{E} x_{ij} \mathfrak{q}^{(k-1)}=\sqrt{z} \sum_{(i,j) \in \mathcal{S}(\nu)} c_{ij}  \mathbb{E} \Big( \frac{1}{\sqrt{n}}\frac{\partial (h_2 h_3)}{\partial x_{ij}}+\frac{\kappa_3}{2n} \frac{\partial^2(h_2 h_3) }{\partial x_{ij}^2} \Big)+\mathbb{E} \mathcal{R}, \label{180920110}
\end{align}
where $\mathcal{R}$ satisfies that, for any sufficiently small $\epsilon>0$ and sufficiently large $K>0$,  
\begin{equation*}
|\mathbb{E} \mathcal{R}| \leq  \sum_{i,j} \E  \Big( n^{-\frac32} \sup_{|x_{ij}| \leq n^{-\frac32+\epsilon}} \Big| c_{ij} \frac{\partial^3 (h_2 h_3)}{\partial x_{ij}^3} \Big|+n^{-K}  \sup_{|x_{ij}| \in \mathbb{R}} \Big| c_{ij} \frac{\partial^3 (h_2 h_3)}{\partial x_{ij}^3} \Big| \Big). 
\end{equation*}
We first show that 
\begin{align}
|\mathbb{E}\mathcal{R}|=O_{\prec}(n^{-\frac12+4\nu}). \label{180920111}
\end{align}
Similar to the discussion of \eqref{eq_a1final}, the proof boils down to estimate the third order derivative of $h_2 h_3$. Using the same proof as \eqref{eq_h1h2parih3} in Lemma \ref{lem_24h1h2h3} (given in Section \ref{s.18092001}), we observe that in the derivatives of $h_2 h_3$, any term containing the derivatives of $h_3$ can bounded by $O_{\prec}(n^{-\frac12+4\nu}).$ Thus, by product rule,
\begin{align*}
\frac{\partial ^3 (h_2 h_3)}{\partial x_{ij}^3 }= \frac{\partial ^3 h_2 }{\partial x_{ij}^3 } h_3 + O_{\prec}(n^{-\frac12+4\nu})=O_{\prec}\Big( \mathbf{u}(i)\mathbf{v}(j)+n^{-\frac12+4\nu} \Big).
\end{align*}
The last step is obtained analogously to \eqref{eq_120}. We omit the details. To conclude \eqref{180920111}, we also use $c_{ij} = O_{\prec} (\mathbf{u}(i)\mathbf{v}(j))$ by recalling its definition \eqref{defn_cij1} and the fact that $\mathbf u$, $\mathbf v$ are both unit vectors. 

Next, using arguments similar to \eqref{eq_h110} and \eqref{eq_120}, we get
\begin{align}
\frac{1}{\sqrt{n}} \frac{\partial (h_2 h_3)}{\partial x_{ij}}=\frac{1}{\sqrt{n}} \frac{\partial h_2 }{\partial x_{ij}}h_3 +O_{\prec}(n^{-\frac12+4\nu})=(k-1)\sqrt{z} c_{ij} \mathfrak{q}^{(k-2)}+O_{\prec}(n^{-\frac12+4\nu}), \label{180920112}
\end{align}
and  
\begin{align}\label{180920113}
\frac{1}{n} \frac{\partial^2 (h_2 h_3)}{\partial x^2_{ij}} = \frac{1}{n} \frac{\partial^2 h_2 }{\partial x^2_{ij}} h_3+O_{\prec}(n^{-\frac12+4\nu}) = 2\frac{(k-1)z}{\sqrt{n}} s_{ij}\mathfrak{q}^{(k-2)} +O_{\prec}(n^{-\frac12+4\nu}).
\end{align}
Plugging  (\ref{180920111})-(\ref{180920113}) into (\ref{180920110}), we obtain (\ref{180920101}). The proof of  Lemma \ref{lem. recursive estimate for each term} is now complete.
\end{proof}

\subsection{Proof of Lemma \ref{lem_estimateh1h2h3}}\label{s.18092001} We start with a simple identity which will be frequently referred to later.   
For any deterministic matrix $W \in \mathbb{R}^{(M+n) \times (M+n)},$ it is elementary to check that  
\begin{equation}\label{eq_firstgm}
\Big( \frac{\partial G}{\partial x_{ij}} W\Big)_{ab}=-\sqrt{z} \Big( G_{aj^{\prime}}(GW)_{ib}+G_{ai}(GW)_{j^{\prime}b} \Big).
\end{equation}
We emphasize that both \eqref{eq_boundm1m2} and a basic fact (as a consequence of \eqref{eq_boundq})  
\begin{align*}
\mathfrak{q}^{(l)}=Q^{l} e^{\ii t \Delta}=O_{\prec}(1) \quad \text{for } l\ge 1 
\end{align*}
will be applied to bound the error terms throughout the proofs of Lemma \ref{lem_estimateh1h2h3}-Lemma \ref{lem_a2b2estimatehigh}.

For convenience, we denote the blocks of $A$ and $B$ (c.f. \eqref{eq:new-ab} ) by $\mathcal A_k$'s and $\mathcal B_k$'s, i.e., 
\begin{align}\label{eq_abmatrix}
A= \begin{pmatrix} \omega_1 \mathbf{u} \mathbf{u}^* & \omega_2 \mathbf{u} \mathbf{v}^* \\ \omega_3 \mathbf{v} \mathbf{u}^* & \omega_4 \mathbf{v} \mathbf{v}^* \end{pmatrix}:= \begin{pmatrix} \mathcal{A}_1 & \mathcal{A}_2 \\ \mathcal{A}_3 & \mathcal{A}_4 \end{pmatrix},
\quad
B= \begin{pmatrix} \varpi_1 \mathbf{u} \mathbf{u}^* & \varpi_2 \mathbf{u} \mathbf{v}^* \\ \varpi_3 \mathbf{v} \mathbf{u}^* & \varpi_4 \mathbf{v} \mathbf{v}^* \end{pmatrix}:= \begin{pmatrix} \mathcal{B}_1 & \mathcal{B}_2 \\ \mathcal{B}_3 & \mathcal{B}_4 \end{pmatrix}.
\end{align}

With the above preparation, we now prove Lemma \ref{lem_estimateh1h2h3}.
\begin{proof}[Proof of Lemma \ref{lem_estimateh1h2h3}] First, by recalling the notations in (\ref{eq_defnfunction}) and (\ref{18091411}), and using \eqref{eq_firstgm}, we have
\begin{align*}
\hslash(1,0,0)&= \frac{1}{\sqrt n} \sum_{i,j} \Big( \frac{\partial G}{\partial x_{ij}} A_1 \Big)_{j'i}\mathfrak{q}^{(k-1)}\\
&=-\sqrt{nz} \frac{1}{n} \sum_{i,j} \Big( G_{j'j^{\prime}}(GA_1)_{ii}+G_{j'i}(GA_1)_{j'i} \Big) \mathfrak{q}^{(k-1)}.
\end{align*}
Moreover, by \eqref{18072401} and \eqref{18091401}, we further get 
\begin{align}\label{eq_parh1h2h3}
\hslash(1,0,0)
&=-\sqrt{nz} m_{2n}(\text{Tr} GA_1) \mathfrak{q}^{(k-1)}-\sqrt{\frac{z}{n}} (\text{Tr} GA_1G) \mathfrak{q}^{(k-1)}\nonumber \\
&= -\sqrt{nz} m_2 (\text{Tr} GA_1) \mathfrak{q}^{(k-1)} +O_{\prec}(n^{-\frac12}),
\end{align}
where the last step follows from the property of trace and (\ref{eq_genrealcontrol}).

Next, using the fact $|\mathcal{B}(\nu)| \leq Cn^{4\nu}$ together with the definition of $c_{ij}$ in \eqref{defn_cij1} and \eqref{eq_localoutside}, we obtain 
\begin{align}\label{eq:proof001}
\hslash(0,0,1) &=\sqrt{z} \sum_{(i,j) \in \mathcal{B}(\nu)} \ii t c_{ij} (\Xi_1)_{j'i} \mathfrak{q}^{(k-1)}=O_{\prec}(n^{-\frac12+4\nu}).
\end{align} 
The main task is the estimate of 
\begin{align*}
\hslash(0,1,0)=\frac{k-1}{\sqrt{n}} \sum_{i,j} (\Xi_1A_1)_{j'i}  \frac{\partial Q}{\partial x_{ij}} \mathfrak{q}^{(k-2)}. 
\end{align*}
In light  of the expression of $\partial Q/\partial x_{ij}$ in \eqref{eq_parti_1}, by symmetry, we get
\begin{align}\label{180920200}
\hslash(0,1,0)=&-(k-1)\sqrt{z} \sum_{i,j}(\Xi_1 A_1)_{j' i} \big[ 2(GAG)_{j'i} -\frac{1}{z}(GBG)_{j' i} + (GBG^2)_{j'i} + (G^2BG)_{j' i} \big] \mathfrak{q}^{(k-2)} \nonumber\\
&-(k-1)\sqrt{z} \sum_{(i,j) \in \mathcal{B}(\nu)}(\Xi_1 A_1)_{j' i} c_{ij} \mathfrak{q}^{(k-2)}.
\end{align}
The last term on the right hand side of \eqref{180920200} is bounded by $O_{\prec}(n^{-\frac12+4\nu})$, by exactly the same estimate of \eqref{eq:proof001}. Now we turn towards the first term on the right hand side of \eqref{180920200}. 
We first claim that
\begin{align}
\sum_{i,j} (\Xi_1A_1)_{j'i} (GAG)_{j'i}=\Tr(\Pi_2^{\mathrm{l}}-\Pi_{1,1}^{\mathrm{l}})A_1 \Pi_1 A +O_{\prec}(n^{-\frac12}).  \label{180920201}
\end{align}
To derive the above statement, a key observation is that the summation on the left hand side of \eqref{180920201} can be written in terms of a trace, with the aid of the block diagonal matrices $\mathtt{I}^{\mathrm u}$ and $\mathtt{I}^{\mathrm l}$ in \eqref{eq_diagonalmatrix}. Indeed, we find
\begin{align*}
\sum_{i,j} (\Xi_1A_1)_{j'i} (GAG)_{j'i} = \Tr( \mathtt{I}^{\mathrm l} \Xi_1 A_1 \mathtt{I}^{\mathrm u} G AG \mathtt{I}^{\mathrm l}) = \Tr (G \mathtt{I}^{\mathrm l} G - G \mathtt{I}^{\mathrm l} \Pi_1) A_1 G A.  
\end{align*}
Thus  $\Pi_{1,1}^{\mathrm{l}}$ and $\Pi_2^{\mathrm{l}}$ (c.f. \eqref{eq_piipi} and \eqref{eq_pi2ul}) appear naturally in \eqref{180920201}.

To prove \eqref{180920201}, using the expressions of $G$ in (\ref{green2}) and $A$ in \eqref{eq_abmatrix}, we have that
\begin{align}
(\Xi_1A)_{j'i} &=(\frac{1}{\sqrt{z}} X^* \mathcal{G}_1 \mathcal{A}_1+(\mathcal{G}_2-m_2)\mathcal{A}_3)_{ji},  \label{eq_g-pia}\\
(GAG)_{j'i} &=(\frac{1}{\sqrt{z}} X^* \mathcal{G}_1 \mathcal{A}_1 \mathcal{G}_1+\mathcal{G}_2 \mathcal{A}_3 \mathcal{G}_1+\frac{1}{z} X^* \mathcal{G}_1 \mathcal{A}_2 X^* \mathcal{G}_1+\frac{1}{\sqrt{z}} \mathcal{G}_2 \mathcal{A}_4 X^* \mathcal{G}_1)_{ji}.  \label{eq_gag}
\end{align}
Expanding the left hand side of \eqref{180920201} with the above expressions, we shall show that there are two main terms and all others are negligible. 

The first contributing term is
\begin{align*}
&\sum_{i,j} ((\mathcal{G}_2-m_2)\mathcal{A}_3)_{j i} (\mathcal{G}_2 \mathcal{A}_3 \mathcal{G}_1)_{j i}= \omega^2_3 \text{Tr} ((\mathcal{G}_2-m_2) \mathbf{v} \mathbf{u}^* \mathcal{G}_1 \mathbf{u} \mathbf{v}^* \mathcal{G}_2 ) \\
&=\omega_3^2 (\mathbf{u}^* \mathcal{G}_1 \mathbf{u})(\mathbf{v}^* \mathcal{G}_2 (\mathcal{G}_2-m_2) \mathbf{v})=\sum_{i, j} ((\Pi_2^{\mathrm{l}}-\Pi_{1,1}^{\mathrm{l}})A_1 \Pi_1)_{j'i} A_{ij'}+  O_{\prec}(n^{-\frac12}), 
\end{align*} 
where in the last step we use $\mathcal{G}_2^2=\mathcal{G}_2'$ and the definition of $A$ in \eqref{eq_abmatrix}, followed by (\ref{eq_localoutside}) and (\ref{eq_genrealcontrol}). 

The second contributing term is
\begin{align*}
\frac{1}{z}\sum_{i,j} (X^* \mathcal{G}_1 \mathcal{A}_1)_{j i} (X^* \mathcal{G}_1 \mathcal{A}_1 \mathcal{G}_1)_{j i}&={\omega_1^2} (\mathbf{u}^* \mathcal{G}_1 \mathbf{u})\big( \frac{1}{z} \mathbf{u}^* \mathcal{G}_1 XX^* \mathcal{G}_1 \mathbf{u} \big).
\end{align*}
Let $\bar{\mathbf{v}}=(\mathbf{0}, \mathbf{v})^*$ and $\bar{\mathbf{u}}=(\mathbf{u}, \mathbf{0})^*$ denote the augmented vectors in $\mathbb{R}^{M+n}.$ Note that by \eqref{eq_genrealcontrol}, we first have
\begin{align*}
\bar{\mathbf{u}}^* G^2 \bar{\mathbf{u}}=\mathbf{u}^* \mathcal{G}_1^2 \mathbf{u}+\frac{1}{z} \mathbf{u}^* \mathcal{G}_1 XX^* \mathcal{G}_1 \mathbf{u}=2m_1'+\frac{m_1}{z}+O_{\prec}(n^{-\frac12}).
\end{align*}
Further observe that 
\begin{equation*}
\mathbf{u}^* \mathcal{G}_1^2 \mathbf{u}=\bar{\mathbf{u}}^* G' \bar{\mathbf{u}}=m_1'+O_{\prec}(n^{-\frac12}),
\end{equation*}
where the last equation follows from  (\ref{eq_consequenceoflcoal}). Putting them together, we conclude that 
\begin{equation*}
\frac{1}{z} \mathbf{u}^* \mathcal{G}_1 XX^* \mathcal{G}_1 \mathbf{u}=m_1'+\frac{m_1}{z}+O_{\prec}(n^{-\frac12}).
\end{equation*}
As a consequence, 
\begin{equation*}
\frac{1}{z}\sum_{i,j} (X^* \mathcal{G}_1 \mathcal{A}_1)_{j i} (X^* \mathcal{G}_1 \mathcal{A}_1 \mathcal{G}_1)_{j i}=\sum_{i, j} ((\Pi_2^{\mathrm{l}}-\Pi_{1,1}^{\mathrm{l}})A_1 \Pi_1)_{ij'} A_{j' i}+O_{\prec}(n^{-\frac12}).
\end{equation*}
Note that 
\begin{align*}
\sum_{i, j} ((\Pi_2^{\mathrm{l}}-\Pi_{1,1}^{\mathrm{l}})A_1 \Pi_1)_{ij'} A_{j' i} + \sum_{i, j} ((\Pi_2^{\mathrm{l}}-\Pi_{1,1}^{\mathrm{l}})A_1 \Pi_1)_{j'i} A_{ij'} = \Tr(\Pi_2^{\mathrm{l}}-\Pi_{1,1}^{\mathrm{l}})A_1 \Pi_1 A.
\end{align*}
What remains is to show all other terms in the expansion of the left hand side of \eqref{180920201} with \eqref{eq_g-pia} and \eqref{eq_gag} are negligible. Let us concentrate on the following term. All other remaining terms are estimated similarly; we omit the details.
\begin{align*}
&\frac{1}{\sqrt{z}}\sum_{i,j} (X^* \mathcal{G}_1 \mathcal{A}_1)_{j'i} (\mathcal{G}_2 \mathcal{A}_3 \mathcal{G}_1)_{j'i} =\frac{\omega_1 \omega_3}{\sqrt{z}} \text{Tr} (X^*\mathcal{G}_1 \mathbf{u}\mathbf{u}^* \mathcal{G}_1 \mathbf{u} \mathbf{v}^* \mathcal{G}_2) \\ 
 &= \frac{\omega_1 \omega_3}{\sqrt{z}} \text{Tr} (\mathbf{v}^*\mathcal{G}_2^2 X^* \mathbf{u} \mathbf{u}^* \mathcal{G}_1 \mathbf{u})= \frac{\omega_1 \omega_3}{\sqrt{z}} (\mathbf{u}^* \mathcal{G}_1 \mathbf{u}) (\mathbf{v}^* \mathcal{G}_2^2X^* \mathbf{u} ).
\end{align*} 
In the second step above, we use the fact $X^* \mathcal{G}_1=\mathcal{G}_2X^*$ which can be checked easily via the singular value decomposition. Therefore, using $\mathcal{G}_2^2 =\mathcal{G}_2'$ and $G' = (G^2 - z^{-1}G)/2$, together with (\ref{eq_orginallocal}) and (\ref{eq_consequenceoflcoal}),  we get that
\begin{equation*}
\mathbf{v}^* \mathcal{G}_2^2 X^* \mathbf{u}=( \bar{\mathbf{v}}^* \sqrt{z} G \bar{\mathbf{u}})'=\frac{1}{2\sqrt{z}} \bar{\mathbf{v}}^* G \bar{\mathbf{u}}+\sqrt{z} \bar{\mathbf{v}}^* G' \bar{\mathbf{u}}=O_{\prec}(n^{-\frac12}). 
\end{equation*}
Hence, we conclude that
\begin{equation*}
\frac{1}{\sqrt{z}}\sum_{i,j} (X^* \mathcal{G}_1 \mathcal{A}_1)_{j'i} (\mathcal{G}_2 \mathcal{A}_3 \mathcal{G}_1)_{j'i} =O_{\prec}(n^{-\frac12}).
\end{equation*}
The proof of \eqref{180920201} is complete. 

 Next, analogously, we shall show that 
\begin{align}\label{eq_piag2bg}
\sum_{i,j} (\Xi_1A_1)_{j'i} (G^2 BG)_{j'i}=\Tr (\Pi_3^{\mathrm{l}}-\Pi_{2,1}^{\mathrm{l}})A_1 \Pi_1 B+O_{\prec}(n^{-\frac12}).
\end{align}
A simple calculation using \eqref{green2} and \eqref{eq:new-ab} yields  
\begin{align} \label{eq_g2bg}
(G^2BG)_{j'i} 
=&(\frac{1}{\sqrt{z}} X^* \mathcal{G}_1^2 \mathcal{B}_1 \mathcal{G}_1+\frac{1}{\sqrt{z}}\mathcal{G}_2 X^* \mathcal{G}_1 \mathcal{B}_1 \mathcal{G}_1+\frac{1}{z}X^* \mathcal{G}_1^2 X \mathcal{B}_3 \mathcal{G}_1+\mathcal{G}_2^2 \mathcal{B}_3 \mathcal{G}_1+\frac{1}{z} X^* \mathcal{G}_1^2 \mathcal{B}_2 X^* \mathcal{G}_1 \nonumber \\
&+\frac{1}{z}\mathcal{G}_2X^* \mathcal{G}_1 \mathcal{B}_2 X^* \mathcal{G}_1+\frac{1}{z^{\frac32}} X^*\mathcal{G}_1^2X \mathcal{B}_4 X^* \mathcal{G}_1+\frac{1}{\sqrt{z}}\mathcal{G}_2^2 \mathcal{B}_4 X^* \mathcal{G}_1)_{ji}.
\end{align}
In a similar way to the discussion of (\ref{180920201}), we expand $(\Xi_1A)_{j'i}(G^2BG)_{j'i} $ using (\ref{eq_g-pia}) and \eqref{eq_g2bg}. There are only four non-negligible terms in the expansion. 

Recall $\mathcal{A}_1$ and $\mathcal{B}_1$ in \eqref{eq:new-ab}. The first non-negligible term is
\begin{align*}
\frac{1}{z} \sum_{i,j} (X^* \mathcal{G}_1 \mathcal{A}_1)_{j'i} (X^* \mathcal{G}_1^2 \mathcal{B}_1 \mathcal{G}_1)_{j'i}=\frac{\omega_1 \varpi_1 }{z} (\mathbf{u}^* \mathcal{G}_1 \mathbf{u}) (\mathbf{u}^* \mathcal{G}_1^2 XX^* \mathcal{G}_1 \mathbf{u}).
\end{align*}
To estimate $\mathbf{u}^* \mathcal{G}_1^2 XX^* \mathcal{G}_1 \mathbf{u}$ in the above,  we observe that (via elementary calculations and the fact $\mathcal{G}_2 X^*=X^* \mathcal{G}_1$) 
\begin{align*}
\bar{\mathbf{u}}^* G^3 \bar{\mathbf{u}}=\mathbf{u}^* \mathcal{G}_1^3 \mathbf{u}+\frac{3}{z} (\mathbf{u}^* \mathcal{G}_1^2 XX^* \mathcal{G}_1 \mathbf{u}).
\end{align*} 
Moreover, by $\mathcal{G}_1^3=\frac{1}{2}\mathcal{G}_1^{\prime \prime}$, (\ref{eq_consequenceoflcoal}) and (\ref{eq_genrealcontrol}), we find 
\begin{align*}
&\mathbf{u}^* \mathcal{G}_1^3 \mathbf{u}=\frac{1}{2}\bar{\mathbf{u}}^* G^{\prime \prime} \bar{\mathbf{u}}=\frac{1}{2}m_1^{\prime \prime}+O_{\prec}(n^{-\frac12}),\\
&\bar{\mathbf{u}}^* G^3 \bar{\mathbf{u}}= 2 m_1'' + \frac{3}{z} m_1' + O_{\prec}(n^{-\frac12}).
\end{align*}
Hence,
\begin{equation}\label{eq_g2xxg}
\mathbf{u}^* \mathcal{G}_1^2 XX^* \mathcal{G}_1 \mathbf{u}=\frac{z}{2} m_{1}^{\prime \prime}+m_1'+O_{\prec}(n^{-\frac12}).
\end{equation}
We conclude that
\begin{equation*}
\frac{1}{z} \sum_{i,j} (X^* \mathcal{G}_1 \mathcal{A}_1)_{j i} (X^* \mathcal{G}_1^2 \mathcal{B}_1 \mathcal{G}_1)_{j i}=\frac{1}{2}\sum_{i,j} ((\Pi_3^{\mathrm{l}}-\Pi_{2,1}^{\mathrm{l}})A_1 \Pi_1)_{ij'} B_{ij'}+O_{\prec}(n^{-\frac12}).
\end{equation*}
Using the fact $X\mathcal{G}_2= \mathcal{G}_1X$ and the same arguments as above, we can show the second non-negligible term is  
\begin{equation*}
\frac{1}{z} \sum_{i,j} (X^* \mathcal{G}_1 \mathcal{A}_1)_{j i} (\mathcal{G}_2 X^* \mathcal{G}_1 \mathcal{B}_1 \mathcal{G}_1)_{j i} =\frac{1}{2}\sum_{i,j} ((\Pi_3^{\mathrm{l}}-\Pi_{2,1}^{\mathrm{l}})A_1 \Pi_1)_{ij'} B_{ij'}+O_{\prec}(n^{-\frac12}),
\end{equation*}
The third  non-negligible term is
\begin{align*}
&\frac{1}{z}\sum_{i,j} ((\mathcal{G}_2-m_2)\mathcal{A}_3)_{j'i} (X^* \mathcal{G}_1^2 X \mathcal{B}_3 \mathcal{G}_1)_{j'i}=\frac{1}{z} (\mathbf{u}^* \mathcal{G}_1 \mathbf{u}) (\mathbf{v}^* X^*\mathcal{G}_1^2 X (\mathcal{G}_2-m_2) \mathbf{v}) \\
&=\omega_3 \varpi_3m_1(\frac{m_2^{\prime \prime}}{2}+\frac{m_2'}{z}-\frac{m_2^2+zm_2m_2'}{z})  +O_{\prec}(n^{-\frac12}),
\end{align*}
where we used the facts $X^* \mathcal{G}_1^2 X \mathcal{G}_2=X^* \mathcal{G}_1^3 X$ and $\mathcal{G}_2^3=\frac{1}{2} \mathcal{G}_2^{\prime \prime}$, as well as
\begin{equation*}
\bar{\mathbf{v}}^* G^3 \bar{\mathbf{v}}=\frac{3}{z}\mathbf{v}^*X^* \mathcal{G}_1^3 X \mathbf{v}+\mathbf{v}^* \mathcal{G}_2^3 \mathbf{v}=2 m_2^{\prime \prime}+\frac{3 m_2'}{z}+O_{\prec}(n^{-\frac12}).
\end{equation*}
The last  non-negligible term can be estimated similarly as
\begin{equation*}
\sum_{i,j} ((\mathcal{G}_2-m_2)\mathcal{A}_3)_{j i} (\mathcal{G}_2^2 \mathcal{B}_3 \mathcal{G}_1)_{j i}=\omega_3 \varpi_3 m_1(\frac{m_2^{\prime \prime}}{2}-m_2 m_2')+O_{\prec}(n^{-\frac12}).
\end{equation*}
Consequently, we have 
\begin{align*}
& \frac{1}{z}\sum_{i,j} ((\mathcal{G}_2-m_2)\mathcal{A}_3)_{j i}  (X^* \mathcal{G}_1^2 X \mathcal{B}_3 \mathcal{G}_1)_{j i}+\sum_{i,j} ((\mathcal{G}_2-m_2)\mathcal{A}_3)_{j i} (\mathcal{G}_2^2 \mathcal{B}_3 \mathcal{G}_1)_{j i} \\
& =\sum_{i,j} ((\Pi_3^{\mathrm{l}}-\Pi_{2,1}^{\mathrm{l}})A_1 \Pi_1)_{j'i} B_{j'i}+O_{\prec}(n^{-1/2}).
\end{align*}
Note that the sum of the four contributing terms is extactly
\begin{align*}
\Tr (\Pi_3^{\mathrm{l}}-\Pi_{2,1}^{\mathrm{l}})A_1 \Pi_1 B+O_{\prec}(n^{-\frac12}).
\end{align*}
To wrap up the proof of \eqref{eq_piag2bg}, it suffices to show all the other terms in the expansion of $\sum_{i,j}(\Xi_1A)_{j'i} (G^2 BG)_{j'i}$ can be bounded by $O_\prec(n^{-\frac12})$.  To see that, for instance, we focus on
\begin{equation*}
z^{-3/2} \sum_{i,j} (X^* \mathcal{G}_1 \mathcal{A}_1)_{j'i} (X^* \mathcal{G}_1^2 X \mathcal{B}_3 \mathcal{G}_1)_{j'i}= \omega_1 \varpi_1 (z^{-3/2} \mathbf{v}^* X^* \mathcal{G}_1^2 XX^* \mathcal{G}_1 \mathbf{u}) (\mathbf{u}^* \mathcal{G}_1 \mathbf{u}).
\end{equation*}
Note that
\begin{align*}
& z^{-3/2}\mathbf{v}^* X^* \mathcal{G}_1^2 XX^* \mathcal{G}_1 \mathbf{u}= \bar{\mathbf{u}}^*G^3 \bar{\mathbf{v}}-\mathbf{u}^*(\frac{1}{\sqrt{z}}\mathcal{G}_1^3 X+\frac{1}{\sqrt{z}} \mathcal{G}^2_1 X \mathcal{G}_2+\frac{1}{\sqrt{z}} \mathcal{G}_1 X \mathcal{G}_2^2)\mathbf{v}, \\
& =\bar{\mathbf{u}}G^3 \bar{\mathbf{v}}-\frac{3}{\sqrt{z}} \mathbf{u}^* \mathcal{G}_1^3X \mathbf{v}= \bar{\mathbf{u}}G^3 \bar{\mathbf{v}}-\frac{3}{2\sqrt{z}} \bar{\mathbf{u}}^* (\sqrt{z}G)^{\prime \prime} \bar{{\mathbf{v}}}=O_{\prec}(n^{-\frac12}),
\end{align*}
where in the third step we use $\mathcal{G}_1^3=\frac{1}{2}\mathcal{G}_1^{\prime \prime}$ and in the last step we use (\ref{eq_genrealcontrol}). Consequently, 
\begin{align*}
z^{-3/2} \sum_{i,j} (X^* \mathcal{G}_1 \mathcal{A}_1)_{j'i} (X^* \mathcal{G}_1^2 X \mathcal{B}_3 \mathcal{G}_1)_{j'i} =O_{\prec}(n^{-\frac12}).
\end{align*}
All the rest terms can be bounded by $O_{\prec}(n^{-\frac12})$ analogously; we omit the details. The proof of \eqref{eq_piag2bg} is now complete.

The remaining two terms in \eqref{180920200} can be estimated the same way as \eqref{180920201} and \eqref{eq_piag2bg}; the details are omitted. We get 
\begin{align}\label{eq_piagbg2ij}
\sum_{i,j} (\Xi_1A_1)_{j'i} (G BG^2 )_{j'i} &=\text{Tr} (\Pi_2^{\mathrm{l}}-\Pi_{1,1}^{\mathrm{l}})A_1 \Pi_2 B+O_{\prec}(n^{-\frac12}), \nonumber\\
\sum_{i,j} (\Xi_1 A_1)_{j'i} (GBG)_{ij'} &=\text{Tr}(\Pi_2^{\mathrm{l}}-\Pi_{1,1}^{\mathrm{l}})A_1 \Pi_1 B+O_{\prec}(n^{-\frac12}).
\end{align}

Plugging \eqref{180920201}, (\ref{eq_piag2bg}) and \eqref{eq_piagbg2ij} into \eqref{180920200},  recalling the definition of $\mathfrak{a}_{11}$ in \eqref{def:missvar}, we conclude that 
\begin{align}
\hslash(0,1,0) & =\mathfrak{a}_{11} \mathfrak{q}^{(k-2)} + O_{\prec}(n^{-\frac12+4\nu}). \label{18091560}
\end{align}
This completes the proof.
\end{proof}

\subsection{Proof of Lemma \ref{lem_24h1h2h3}} \label{s.18092002} We use this subsection to prove Lemma \ref{lem_24h1h2h3}.  
\begin{proof}[Proof of Lemma \ref{lem_24h1h2h3}]  We first study the second derivatives.  By \eqref{18091411} and (\ref{eq_secondgm}), we have 
\begin{align*}
&\hslash(2,0,0)= \frac1n \sum_{i,j} \Big( \frac{\partial^2 G}{\partial x_{ij}^2} A_1 \Big)_{j' i} \mathfrak{q}^{(k-1)}\nonumber\\
&=\frac{2z}{n} \sum_{i,j} \Big( \big(G_{j'j^{\prime}} G_{ij^{\prime}}+G_{j'i} G_{j^{\prime} j^{\prime}} \big)(GA_1)_{ii}+\big(G_{j'j^{\prime}} G_{ii}+G_{j'i} G_{j^{\prime} i} \big)(GA_1)_{j^{\prime}i} \Big) \mathfrak{q}^{(k-1)}.
\end{align*} 
First of all,  by (\ref{eq_localoutside}) and (\ref{eq_genrealcontrol}), we find that 
\begin{equation*}
\frac{1}{n} \sum_{i,j} G_{j'j'} G_{ii} (GA_1)_{j'i} =\frac{1}{n} \sum_{i,j} (\Pi_1)_{ii} (\Pi_1)_{j'j'}  (\Pi_1 A_1)_{j'i}+O_{\prec}(n^{-\frac12}).
\end{equation*}
It is simple to check that  $(GA)_{ii}=(\mathcal{G}_1 \mathcal{A}_1+z^{-1/2} \mathcal{G}_1X \mathcal{A}_3)_{ii}$.  By (\ref{eq_boundm1m2}) and (\ref{eq_genrealcontrol}), we get 
\begin{align}\label{eq_controlbyitem}
\frac{1}{n} \sum_{i,j} G_{j'j'} G_{ij'} (GA_1)_{ii}= O_{\prec}\Big( n^{-\frac32} \sum_{i,j} (\mathcal{G}_1 \mathbf{u} \mathbf{u}^*+\mathcal{G}_1 X \mathbf{v} \mathbf{u}^*)_{ii} \Big)=O_{\prec}(n^{-\frac12}).
\end{align}
Similarly, we also have
\begin{align*}
&\frac{1}{n} \sum_{i,j} G_{j'j'} G_{j'i} (GA_1)_{ii} =O_{\prec}(n^{-\frac12}),  \nonumber\\
&\frac{1}{n} \sum_{i,j} G_{j'i} G_{j'i} (GA_1)_{j'i} =O_{\prec}(n^{-1}).
\end{align*}
Putting the above estimates together, and  recalling $\mathfrak{d}_1^a$ in (\ref{18091910}), we conclude that 
\begin{align*}
\hslash(2,0,0)= \mathfrak{d}_1^a \mathfrak{q}^{(k-1)}+O_{\prec}(n^{-\frac12}).
\end{align*}
Next, the estimation of
\begin{align*}
\hslash(1,1,0)&=\frac{k-1}{n} \sum_{i,j} \Big( \frac{\partial G}{\partial x_{ij}}A_1 \Big)_{j'i} \frac{\partial Q}{\partial x_{ij}} \mathfrak{q}^{(k-2)}
\end{align*}
follows closely the same steps as the derivation of (\ref{eq_h1partih2h3}). By \eqref{eq_firstgm}, 
\begin{align}\label{eq_h110v2}
\hslash(1,1,0)&=-\frac{(k-1)\sqrt z}{n} \sum_{i,j} \Big( G_{j'j^{\prime}}(GA_1)_{ii}+G_{j'i}(GA_1)_{j'i} \Big)\frac{\partial Q}{\partial x_{ij}} \mathfrak{q}^{(k-2)}.
\end{align}

We shall prove that 
\begin{align} \label{eq_trianglediscussion}
\frac{1}{n} \sum_{i,j} G_{j' j'} (GA_1)_{ii} \frac{\partial Q}{\partial x_{ij}}= \sqrt{\frac{z}{n}} \sum_{\sum_{(i,j)\in \mathcal S(\nu)}} (\Pi_1A)_{ii}(\Pi_1)_{j'j'} c_{ij}+O_{\prec}(n^{-\frac12+4\nu}),
\end{align}
which will be used several times later. We postpone the proof of \eqref{eq_trianglediscussion} till the end of this subsection. 

Again by (\ref{eq_genrealcontrol}), recalling the definitions of $c_{ij}$ in (\ref{defn_cij1}) and $A$ in \eqref{eq:new-ab}, we have that
\begin{align}
\frac{1}{n} \sum_{i,j} G_{j'i} (GA_1)_{j'i} \frac{\partial Q}{\partial x_{ij}}=O_{\prec}(n^{-3/2} \sum_{i,j} A_{ij'} c_{ij} )=O_{\prec}(n^{-\frac12}). \label{180920301}
\end{align}
Inserting  (\ref{eq_trianglediscussion}) and (\ref{180920301}) back into (\ref{eq_h110v2}), by recalling $\mathfrak{a}_{21}$ in (\ref{18091930}), we conclude that 
\begin{align*}
\hslash(1,1,0)
=\mathfrak{a}_{21} \mathfrak{q}^{(k-2)}+O_{\prec}(n^{-\frac12+4\nu}).
\end{align*} 

Using a discussion similar to (\ref{eq_h1h2parih3}), we also have 
\begin{align*}
\hslash(1,0,1)=\sqrt{\frac{z}{n}} \sum_{(i,j) \in \mathcal{B}(\nu)} \ii tc_{ij}\Big( \frac{\partial G}{\partial x_{ij}}A_1 \Big)_{j'i} \mathfrak{q}^{(k-1)}=O_{\prec}(n^{-\frac12+4\nu}). 
\end{align*}
Actually, all the terms containing the derivatives of $h_3$ can be estimated in the same way. Thus both $\hslash(0,1,1)$ and $\hslash(0,0,2)$ are also bounded by $O_{\prec}(n^{-\frac12+4\nu})$. We omit the details.

It remains to estimate
\begin{multline} \label{eq_partilq2}
\hslash(0,2,0)=\frac{1}{n} \sum_{i,j} (\Xi_1A_1)_{j'i}\Big((k-1) \frac{\partial^2 Q}{\partial x^2_{ij}} \mathfrak{q}^{(k-2)} 
+(k-1)(k-2)\Big(\frac{\partial Q}{\partial x_{ij}}\Big)^2 \mathfrak{q}^{(k-3)}\Big).
\end{multline}
The calculation of \eqref{eq_partilq2}  is similar to that of (\ref{180920200}) and due to an extra factor $n^{-1/2}$ in front, we shall show that $\hslash(0,2,0)$ can be bounded by $O_{\prec}(n^{-1/2})$. We only list the main differences here. We expand the product on the right hand side of \eqref{eq_partilq2} using the expressions of $(\Xi_1A_1)_{j'i}$ in \eqref{eq_g-pia}, $\partial Q/\partial x_{ij}$ in \eqref{eq_parti_1} and $\partial^2 Q/\partial x_{ij}^2$ in (\ref{eq_parti_2}).  

Most derivations of the items in (\ref{180920200}) can be directly applied to those in \eqref{eq_partilq2} except three items, which are discussed below. 
Denote $\mathbf{e}_i$ with $i\in [M]$ as the standard basis in $\mathbb{R}^M$  and $\mathbf{f}_j$ with $j\in [N]$ as those in $\mathbb{R}^N$,

First, by (\ref{eq_localoutside}) and (\ref{eq_genrealcontrol}),    
\begin{align}
 \sum_{i,j}(X^* \mathcal{G}_1 \mathcal{A}_1)_{j'i} (X^* \mathcal{G}_1 \mathcal{A}_1 \mathcal{G}_1)_{j'i} (X^* \mathcal{G}_1 \mathcal{A}_1 \mathcal{G}_1)_{j'i}&=\sum_{i,j} (\mathbf{e}_{j} ^* X^* \mathcal{G}_1 \mathcal{A}_1 \mathbf{e}_i)(\mathbf{e}_j^*X^* \mathcal{G}_1 \mathcal{A}_1 \mathcal{G}_1 \mathbf{e}_i)^2 \nonumber\\ 
&=O_{\prec}\Big(n^{-\frac32} \sum_{i,j} \mathbf{u}^3(i)  \Big)=O_{\prec}(n^{-\frac12}). \label{180920405}
\end{align}
Second,  using (\ref{eq_localoutside}) and the fact that $\mathbf{u}, \mathbf{v}$ are unit vectors, we get
\begin{align*}
\frac{1}{\sqrt{n}} \sum_{i,j} (X^* \mathcal{G}_1 \mathcal{A}_1)_{ji} (GAG)_{j'j'} G_{ii}&=\frac{m_1 m_2^2}{\sqrt{n}} \sum_{i,j} (X^* \mathcal{G}_1 \mathcal{A}_1)_{ji} A_{j'j'}+O_{\prec}(n^{-\frac12}), \\
& = \frac{m_1 m_2^2 \omega_1^2 }{\sqrt{n}} \sum_{i,j} \mathbf{f}_j^* X^* \mathcal{G}_1  \mathbf{u} \mathbf{u}(i) \mathbf{v}^2(j)+O_{\prec}(n^{-\frac12}) \\
& =\frac{m_1 m_2^2 \omega_1^2}{n} \sum_{i,j} \mathbf{u}(i) \mathbf{v}^2(j)+O_{\prec}(n^{-\frac12}) \\
&=O_{\prec}(n^{-\frac12}).
\end{align*}
Third, we invoke (\ref{eq_localoutside}) to get that $\sum_{j} \mathbf{f}_j^* X^* \mathcal{G}_1 \mathbf{u}=\sqrt{n} \mathbf{f} X^* \mathcal{G}_1 \mathbf{u}=O_{\prec}(1)$, where $\mathbf{f}=\frac{1}{\sqrt{n}} \mathbf{1}$.  Then it follows that
\begin{align}\label{eq_hardest}
\frac{1}{\sqrt{n}} \sum_{i,j} (X^* \mathcal{G}_1 \mathcal{A}_1)_{ji} (GAG)_{ii} G_{j'j'}&=\frac{m^2_1 m_2}{\sqrt{n}} \sum_{i,j} (X^* \mathcal{G}_1 \mathcal{A}_1)_{ji} A_{ii}+O_{\prec}(n^{-\frac12}), \nonumber \\
& = \frac{m_1^2 m_2 \omega_1^2 }{\sqrt{n}} \sum_{i,j} \mathbf{f}_j^* X^* \mathcal{G}_1  \mathbf{u} \mathbf{u}^3(i)+O_{\prec}(n^{-\frac12}) \nonumber \\
&=\frac{m_1^2 m_2 \omega_1^2 }{\sqrt{n}} \Big( \sum_i \mathbf{u}^3(i) \Big) \Big( \sum_{j} \mathbf{f}_j^* X^* \mathcal{G}_1 \mathbf{u} \Big)+O_{\prec}(n^{-\frac12}) \nonumber \\
& =O_{\prec}(n^{-\frac12}).
\end{align}
Finally,  we can conclude that
\begin{equation*}
\hslash(0,2,0)=O_{\prec}(n^{-\frac12}).
\end{equation*}
This finishes the discussion of the second order derivatives.  We continue with the third derivatives.  We start with 
\begin{multline}\label{eq:proof120}
\hslash(1,2,0)=n^{-\frac32} \sum_{i,j}  (\frac{\partial G}{\partial x_{ij}}A_1)_{j'i}\Big((k-1) \frac{\partial^2 Q}{\partial x^2_{ij}} \mathfrak{q}^{(k-2)}
+(k-1)(k-2)\Big(\frac{\partial Q}{\partial x_{ij}}\Big)^2 \mathfrak{q}^{(k-3)}\Big).
\end{multline} 
Recalling (\ref{eq_firstgm}) and (\ref{eq_parti_2}), by (\ref{eq_localoutside}) and (\ref{eq_genrealcontrol}), the first term on the right hand side of \eqref{eq:proof120} is estimated by
\begin{align}\label{eq_thirdnonvaish}
&n^{-\frac32} \sum_{i,j} (\frac{\partial G}{\partial x_{ij}}A_1)_{j'i} \frac{\partial^2 Q}{\partial x^2_{ij}}\nonumber\\
&=\frac{-2z^{\frac32}}{n} \sum_{i,j}  G_{j'j'} (GA_1)_{ii} \Big((GAG)_{ii} G_{j'j'}-\frac{1}{2z} (GBG)_{ii} G_{j'j'}+\frac{1}{2}(GBG)_{ii} G^2_{j'j'} \nonumber \\
&\qquad\qquad\qquad\qquad+\frac{1}{2}(G^2 B G)_{ii} G_{j'j'}+\frac{1}{2}(GBG^2)_{ii} G_{j'j'} \Big)+O_{\prec}(n^{-\frac12}) \nonumber \\
&=-\frac{2z^{\frac32}}{n} \sum_{i,j} (\Pi_1 A)_{ii} (\Pi_1)_{j'j'} s_{ij}+O_{\prec}(n^{-1/2}). 
\end{align}
In the last equation above,  we recall the definition of $s_{ij}$ in \eqref{defn_sij1}. Furthermore, recalling  (\ref{eq:derivativeQ1}) and (\ref{eq_firstgm}), by  (\ref{eq_genrealcontrol}),  it is easy to see that the second term on the right hand side of \eqref{eq:proof120} is
\begin{equation*}
n^{-3/2} \sum_{i,j} \Big( \frac{\partial G}{\partial x_{ij}} A_1 \Big)_{j'i} \Big( \frac{\partial Q}{\partial x_{ij}} \Big)^2 \mathfrak{q}^{(k-3)}= O_{\prec}\Big( n^{-3/2} \sum_{i,j} (A_1)_{ii} c^2_{ij} \Big)=O_{\prec}(n^{-1}).
\end{equation*}
For the last equation above, we refer to the definition of $c_{ij}$ in \eqref{defn_cij1}. 
Using $\mathfrak{a}_{31}$ defined in (\ref{18091931}), we hence conclude that 
\begin{equation*}
\hslash(1,2,0)=\mathfrak{a}_{31} \mathfrak{q}^{(k-2)}+O_{\prec}(n^{-1/2}).
\end{equation*}
Next we study  
\begin{align*}
\hslash(0,3,0)=n^{-\frac32} \sum_{ij} (\Xi_1 A_1)_{j'i} \Big( & (k-1) \frac{\partial^3 Q}{\partial x_{ij}^3} \mathfrak{q}^{(k-2)}+3\frac{(k-1)!}{(k-3)!} \frac{\partial^2 Q}{\partial x_{ij}^2} \frac{\partial Q}{\partial x_{ij}} \mathfrak{q}^{(k-3)}\\
&+\frac{(k-1)!}{(k-4)!} \Big(\frac{\partial Q}{\partial x_{ij}} \Big)^3 \mathfrak{q}^{(k-4)} \Big).
\end{align*}
We briefly argue that $\hslash(0,3,0)$ is bounded by $O_{\prec}(n^{-\frac12})$, using a  discussion similar to those of $\hslash(0,1,0)$ in \eqref{180920200} and $\hslash(0,2,0)$ in  (\ref{eq_partilq2}). 

Recalling (\ref{eq_parti_1}) and (\ref{eq_parti_2}), it is easy to see that
\begin{align*}
n^{-\frac32} \sum_{i,j} (\Xi_1 A_1)_{j'i}  \frac{\partial^2 Q}{\partial x_{ij}^2} \frac{\partial Q}{\partial x_{ij}}=O_{\prec}(n^{-1} \sum_{i,j} \mathbf{u}(i) \mathbf{v}^3(j))=O_{\prec}(n^{-\frac12}).
\end{align*}
Similarly, by \eqref{eq_parti_1} and \eqref{eq_parti_3}, it can also be shown that 
\begin{align*}
n^{-\frac32} \sum_{i,j} (\Xi_1 A_1)_{j'i} \Big(\frac{\partial Q}{\partial x_{ij}} \Big)^3&=O_{\prec}(n^{-\frac12} \sum_{i,j} \mathbf{u}^3(i) \mathbf{v}^3(j))=O_{\prec}(n^{-\frac12}),\\
n^{-\frac32} \sum_{i,j} (\Xi_1A_1)_{j'i} \frac{\partial^3 Q}{\partial x_{ij}^3}&=O_{\prec}(n^{-\frac12}).
\end{align*}
This completes the discussion of $\hslash(0,3,0).$ The same arguments can be applied to show that
\begin{align*}
\hslash(2,1,0)=n^{-\frac32} \sum_{i,j} \Big( \frac{\partial^2 G}{\partial x_{ij}^2} A_1 \Big)_{j'i} \frac{\partial Q}{\partial x_{ij}} \mathfrak{q}^{(k-2)}=O_{\prec}(n^{-1} \sum_{i,j} \mathbf{u}^2(i) \mathbf{v}^2(j))=O_{\prec}(n^{-1}).
\end{align*}
and
\begin{align*}
\hslash(3,0,0)=n^{-\frac32} \sum_{i,j} \Big(\frac{\partial^3 G}{\partial x^3_{ij}} A_1 \Big)_{j'i}\mathfrak{q}^{(k-1)}=O_{\prec}(n^{-3/2} \sum_{i,j} (A_1)_{j'i} )=O_{\prec}(n^{-\frac12})
\end{align*}
by using the expressions \eqref{eq_secondgm} and \eqref{eq_thirdgm} respectively. 

For all the rest of the items containing the derivatives of $h_3,$  they can be easily estimated using a discussion similar to (\ref{eq_h1h2parih3}). 

Finally, using (\ref{eq_secondgm})-(\ref{eq_parti_4}), (\ref{eq_parti_1}),  (\ref{eq_firstgm}) and (\ref{eq_localoutside}), all the fourth order derivatives can bounded by $O_{\prec}(n^{-\frac12})$. The discussion is similar to that of (\ref{eq_partilq2}); we omit further details here.   
This concludes our proof. 
\end{proof}

\begin{proof}[Proof of (\ref{eq_trianglediscussion})]  
We split the left hand side of \eqref{eq_trianglediscussion} as the sum of the following three items
\begin{align*}
\frac{1}{n} \sum_{i,j} (G_{j' j'}-m_2) (GA_1)_{ii} \frac{\partial Q}{\partial x_{ij}},\ \frac{1}{n} \sum_{i,j} m_2 (\Xi_1 A_1)_{ii} \frac{\partial Q}{\partial x_{ij}}, \ \frac{1}{n} \sum_{i,j} m_1 m_2 (A_1)_{ii} \Big( \frac{\partial Q}{\partial x_{ij}}-\sqrt{nz} c_{ij} \Big).
\end{align*}

First of all, by (\ref{eq_genrealcontrol}) and (\ref{eq_parti_1}), we have 
\begin{align*}
\frac{1}{n} \sum_{i,j} (G_{j' j'}-m_2) (GA_1)_{ii} \frac{\partial Q}{\partial x_{ij}}=O_{\prec} \Big( n^{-1} \sum_{(i,j) \in \mathcal{S}(\nu)} \mathbf{u}^3(i) \mathbf{v}(j)  \Big)=O_{\prec}(n^{-\frac12}).
\end{align*}  
Similarly, we also have
\begin{align*}
\frac{1}{n} \sum_{i,j} m_2 (\Xi_1 A_1)_{ii} \frac{\partial Q}{\partial x_{ij}}&=\frac{m_2 \omega_1 }{n} \sum_{i,j} \mathbf{e}_i^* \Xi_1\mathbf{u} \mathbf{u}(i) \frac{\partial Q}{\partial x_{ij}} \\
&=O_{\prec}(n^{-1} \sum_{i,j} \mathbf{u}(i)^2 \mathbf{v}(j))=O_{\prec}(n^{-\frac12}).
\end{align*}
Furthermore, using a discussion similar to that of (\ref{eq_hardest}), we get 
\begin{align*}
\frac{1}{n} \sum_{i,j} m_1 m_2 (A_1)_{ii} \Big( \frac{\partial Q}{\partial x_{ij}}-\sqrt{nz} c_{ij} \Big)=O_{\prec}(n^{-\frac12} \sum_{i}  \mathbf{u}^3(i) )=O_{\prec}(n^{-\frac12}),
\end{align*}
where we apply the fact that 
\begin{align*}
(GAG)_{j'i}-m_1m_2 A_{j'i}=O_{\prec}\Big( \frac{\mathbf{u}(i)}{\sqrt{n}}   \Big).
\end{align*}
Summing up the above three estimates, we can conclude the proof of \eqref{eq_trianglediscussion}. 
\end{proof}

\subsection{Proofs of Lemmas \ref{lem_bestimate}-\ref{lem_a2b2estimatehigh}}\label{subsec:proofofotherlemma}
In this subsection, we will prove Lemmas \ref{lem_bestimate}-\ref{lem_a2b2estimatehigh}. The proofs are analogous to those of Lemma \ref{lem_estimateh1h2h3} and Lemma \ref{lem_24h1h2h3}; we only outline the main steps.

We record a basic identity for later estimates.  For any deterministic matrix $W \in \mathbb{R}^{(M+n) \times (M+n)},$ it is elementary to check that
\begin{equation}\label{eq_g2w1st}
\Big( \frac{\partial G^2}{\partial x_{ij}}W\Big)_{ab}=-\sqrt{z} (G_{aj'}^2 (GW)_{ib}+G_{ai}^2 (GM)_{j'b}+G_{aj'} (G^2W)_{ib}+G_{ai}(G^2W)_{j'b}).
\end{equation}

\begin{proof}[Proof of Lemma \ref{lem_bestimate}] 
Recalling $\tilde h_1$ in (\ref{eq_htitle}), by a discussion similar to (\ref{eq_parh1h2h3}),  we get 
\begin{align*}
\tilde{\hslash}(1,0,0)&=\frac{1}{\sqrt{n}}\sum_{i,j} \frac{\partial \tilde{h}_1}{\partial x_{ij}}h_2 h_3 =\frac{1}{\sqrt{n}} \sum_{i,j} \Big( \frac{\partial G^2}{\partial x_{ij}} B_1 \Big)_{j'i} \mathfrak{q}^{(k-1)} \\
&=-\sqrt{nz} \big((2m_2'+\frac{m_2}{z})\text{Tr}(GB_1)+m_2 \text{Tr}(G^2B_1) \big) \mathfrak{q}^{(k-1)}+O_{\prec}(n^{-\frac12}).
\end{align*}
In the last step above, we use (\ref{eq_g2w1st}), (\ref{eq_localoutside}) and  (\ref{eq_genrealcontrol}). 
Next, we turn towards to the term
\begin{align}\label{eq:prooft010}
\tilde{\hslash}(0,1,0)=\frac{1}{\sqrt{n}} \sum_{i,j} \tilde{h}_1 \frac{\partial h_2}{\partial x_{ij}} h_3=\frac{(k-1)}{n}\sum_{i,j} (\Xi_2B_1)_{j'i} \frac{\partial Q}{\partial x_{ij}} \mathfrak{q}^{(k-2)},
\end{align}
which will be estimated following exactly the same steps as those of \eqref{eq_h1partih2h3}.
Observe that 
\begin{equation*}
(\Xi_2B)_{j'i}=\Big((\frac{1}{\sqrt{z}} X^* \mathcal{G}_1^2+\frac{1}{\sqrt{z}} \mathcal{G}_2 X^* \mathcal{G}_1) \mathcal{B}_1+(\frac{1}{z}X^* \mathcal{G}_1^2 X+\mathcal{G}_2^2-2m_2'-\frac{m_2}{z}) \mathcal{B}_3 \Big)_{j'i}. 
\end{equation*}
By \eqref{eq_parti_1}, after expanding the product on the right hand side of \eqref{eq:prooft010}, we find the following estimates. 
\begingroup 
\allowdisplaybreaks
\begin{align*}
\frac{1}{\sqrt{n}} \sum_{i,j} (\Xi_2B_1)_{j'i} (GAG)_{j'i} =\text{Tr} (\Pi_3^{\mathrm{l}}-\Pi_{1,2}^{\mathrm{l}})B_1 \Pi_1 A+O_{\prec}(n^{-\frac12}), \\
\frac{1}{\sqrt{n}} \sum_{i,j} (\Xi_2B_1)_{j'i} (GBG)_{j' i} =\text{Tr} (\Pi_3^{\mathrm{l}}-\Pi_{1,2}^{\mathrm{l}})B_1 \Pi_1 B+O_{\prec}(n^{-\frac12}), \\
 \frac{1}{\sqrt{n}} \sum_{i,j}  (\Xi_2 B_1)_{j'i}   (G^2 BG)_{j'i}=\text{Tr}(\Pi_4^{\mathrm{l}}-\Pi_{2,2}^{\mathrm{l}})B_1 \Pi_1 B+O_{\prec}(n^{-\frac12}), \\
\frac{1}{\sqrt{n}} \sum_{i,j}  (\Xi_2 B_1)_{j'i}   (G BG^2)_{j'i}=\text{Tr}(\Pi_3^{\mathrm{l}}-\Pi_{1,2}^{\mathrm{l}}) B_1 \Pi_2 B+O_{\prec}(n^{-\frac12}).
\end{align*}
\endgroup
Putting these estimates together and invoking $\tilde{\mathfrak{b}}_{11} $ in \eqref{def:missvar}, we have 
\begin{align*}
&\tilde{\hslash}(0,1,0)=\tilde{\mathfrak{b}}_{11} \mathfrak{q}^{(k-2)}+O_{\prec}(n^{-\frac12 + 4 \nu}).
\end{align*}
Lastly, $\tilde{\hslash}(0,0,1)$ can be estimated using a discussion  similar to (\ref{eq_h1h2parih3}). We can therefore conclude our proof. 
\end{proof} 

The proof of Lemma \ref{lem_bestimate2} follows along the exact lines of Lemma \ref{lem_24h1h2h3} with minor changes. We only sketch it below. 
\begin{proof}[Proof of Lemma \ref{lem_bestimate2}]  First of all,  by (\ref{eq_genrealcontrol}) and (\ref{eq_partig2w}), using a discussion similar to (\ref{eq_h200}), we have that
\begin{align*}
\tilde{\hslash}(2,0,0)=\frac{1}{n} \sum_{i,j}  \frac{\partial^2 \tilde{h}_1 }{\partial x_{ij}^2} h_2h_3
=\frac{1}{n} \sum_{i,j} \Big( \frac{\partial^2 G^2}{\partial x_{ij}^2} B_1 \Big)_{j'i} \mathfrak{q}^{(k-1)}=\tilde{\mathfrak{d}}_1 \mathfrak{q}^{(k-1)}+O_{\prec}(n^{-\frac12}),
\end{align*} 

Second, following the same steps in (\ref{eq_h110v2}), together with (\ref{eq_g2w1st}) and (\ref{eq_parti_1}), we find that 
\begin{align*}
\tilde{\hslash}(1,1,0)=\frac{1}{n} \sum_{i,j}  \frac{\partial \tilde{h}_1 }{\partial x_{ij}} \frac{\partial h_2}{\partial x_{ij}} h_3=\frac{k-1}{n} \sum_{i,j} \Big( \frac{\partial^2 G}{\partial x_{ij}^2}B_1 \Big)_{j'i} \frac{\partial Q}{\partial x_{ij}} \mathfrak{q}^{(k-2)}=\tilde{\mathfrak{b}}_{21} \mathfrak{q}^{(k-2)}+O_{\prec}(n^{-\frac12}).
\end{align*}

Third,  the same arguments of (\ref{eq_120}) using (\ref{eq_g2w1st}) and (\ref{eq_parti_2})  yield  
\begin{align*}
\tilde{\hslash}(1,2,0)=n^{-\frac32}  \sum_{i,j} \frac{\partial \tilde{h}_1}{\partial x_{ij}} \frac{\partial^2 h_2}{\partial x^2_{ij}} h_3 = \tilde{\mathfrak{b}}_{31} \mathfrak{q}^{k-2}+O_{\prec}(n^{-1/2}).
\end{align*}
For the rest of the items, we can apply discussions similar to those of the corresponding items from Lemma \ref{lem_24h1h2h3}. We omit the details here. 

\end{proof}
Lemma \ref{lem_a2b2estimate} is an analogue of Lemma \ref{lem_estimateh1h2h3} and \ref{lem_bestimate} for the matrices $A_2$ and $B_2$; the proof is analogous.   
\begin{proof}[Proof of Lemma \ref{lem_a2b2estimate}] Recall (\ref{eq_firstgm}). Using a  discussion similar to that of (\ref{eq_parh1h2h3}), by  (\ref{eq_localoutside}) and (\ref{eq_genrealcontrol}), we find that 
\begin{align*}
\mathfrak{h}(1,0,0)&=\frac{1}{\sqrt{n}} \sum_{i,j} \frac{\partial \mathbbm{h}_1}{\partial x_{ij}} h_2 h_3= \frac{1}{\sqrt{n}} \sum_{i,j} \Big( \frac{\partial G}{\partial x_{ij}} A_2 \Big)_{ij^{\prime} } \mathfrak{q}^{(k-1)}\nonumber \\
& =-\sqrt{nz} \frac{1}{n} \sum_{i,j} \Big( G_{ii}(GA_2)_{j'j'}+G_{ij'}(GA_2)_{ij'} \Big) \mathfrak{q}^{(k-1)} \nonumber \\
&=-\sqrt{nz} y m_{1n}(\text{Tr} GA_2) \mathfrak{q}^{(k-1)}-\frac{\sqrt{z}}{\sqrt{n}} (\text{Tr} GA_2^*G) \mathfrak{q}^{(k-1)}\nonumber \\
&= -\sqrt{nz} y m_1 (\text{Tr} GA_2) \mathfrak{q}^{(k-1)}+O_{\prec}(n^{-\frac12}),
\end{align*}
where we recall that $y=\frac{M}{n}.$ Similarly, using (\ref{eq_g2w1st}),  we also have
\begin{align*}
\tilde{\mathfrak{h}}(1,0,0)&=\frac{1}{\sqrt{n}} \sum_{i,j} \frac{\partial \tilde{\mathbbm{h}}_1}{\partial x_{ij}} h_2 h_3  =\frac{1}{\sqrt{n}} \sum_{i,j} \Big( \frac{\partial G^2}{\partial x_{ij}} B_2 \Big)_{ij'} \mathfrak{q}^{(k-1)} \\
& =-\sqrt{nz} \frac{1}{n} \sum_{i,j} \Big( G_{ii}^2 (GB_2)_{j'j'}+G_{ii} (G^2 B_2)_{j'j'} \Big) \mathfrak{q}^{(k-1)}+O_{\prec}(n^{-\frac12}) \\
&=-\sqrt{nz} y\Big((2m_1'+\frac{m_1}{z}) (\text{Tr} GB_2)+m_1 \text{Tr} G^2 B_2 \Big) \mathfrak{q}^{(k-1)}+O_{\prec}(n^{-\frac12}),
\end{align*}
Next, we estimate $$\mathfrak{h}(0,1,0)=\frac{k-1}{\sqrt{n}} \sum_{i,j} (\Xi_1A_2)_{i j'}  \frac{\partial Q}{\partial x_{ij}} \mathfrak{q}^{(k-2)}.$$ Recall the expression of $\partial Q/\partial x_{ij}$ in \eqref{eq_parti_1}. As seen in the discussion below  \eqref{180920200}, the key observation is that 
$$\sum_{i,j} (\Xi_1A_2)_{i j'} (GAG)_{i j'}  = \Tr (G \mathtt{I}^{\mathrm u} G - G \mathtt{I}^{\mathrm u} \Pi_1) A_2 G A$$
Thus we shall prove
\begin{align*}
\sum_{i,j} (\Xi_1A_2)_{i j'} (GAG)_{i j'}=\Tr(\Pi_2^{\mathrm{u}}-\Pi_{1,1}^{\mathrm{u}})A_2 \Pi_1 A +O_{\prec}(n^{-\frac12}). 
\end{align*}
The proof follows from 
\begin{equation*}
(\Xi_1A_2)_{ij'}=((\mathcal{G}_1-m_1) \mathcal{A}_2+z^{-1/2} \mathcal{G}_1X \mathcal{A}_4)_{ij}
\end{equation*}
and exactly the same arguments as \eqref{180920201}.  Likewise, we also get
\begingroup
\allowdisplaybreaks
\begin{align*}
&\frac{1}{\sqrt{n}} \sum_{i,j} (\Xi_1A_2)_{ij'} (GBG)_{ij'}=\text{Tr}(\Pi_2^{\mathrm{u}}-\Pi_{1,1}^{\mathrm{u}})A_2 \Pi_1B+O_{\prec}(n^{-\frac12}),\\
&\frac{1}{\sqrt{n}}\sum_{i,j} (\Xi_1A_2)_{ij'} (G^2BG)_{ij'}=\text{Tr} (\Pi_3^{\mathrm{u}}-\Pi_{2,1}^{\mathrm{u}})A_2 \Pi_1 B+O_{\prec}(n^{-\frac12}),\\
&\frac{1}{\sqrt{n}} \sum_{i,j} (\Xi_1A_2)_{ij'} (GBG^2)_{ij'}=\text{Tr} (\Pi_2^{\mathrm{u}}-\Pi_{1,1}^{\mathrm{u}})A_2 \Pi_2 B+O_{\prec}(n^{-\frac12}). 
\end{align*}
\endgroup
Putting the above estimates together and recalling $\mathfrak{a}_{12}$ below \eqref{def:missvar}, we finish the computation for (\ref{eq_partifrakh1h2h3}). The other term $\tilde{\mathfrak{h}}(0,1,0)$ can be estimated analogously  by noting 
\begin{equation*}
(\Xi_2B_2)_{ij'}=((\mathcal{G}_1^2+\frac{1}{z} \mathcal{G}_1 XX^* \mathcal{G}_1-2m_1'-\frac{m_1}{z}) \mathcal{B}_2+(\frac{1}{\sqrt{z}} \mathcal{G}_1^2 X+\frac{1}{\sqrt{z}} \mathcal{G}_1 X\mathcal{G}_2)\mathcal{B}_4)_{ij}.
\end{equation*}
Finally, $\mathfrak{h}(0,0,1)$ and $\tilde{\mathfrak{h}}(0,0,1)$ can be estimated using a discussion similar to that of (\ref{eq_h1h2parih3}). The details are omitted.
\end{proof}

The remaining part of this section is the proof of Lemma \ref{lem_a2b2estimatehigh}, which is an analogue of Lemma \ref{lem_24h1h2h3} and \ref{lem_bestimate2} for the matrices $A_2$ and $B_2.$ 
\begin{proof}[Proof of Lemma \ref{lem_a2b2estimatehigh}]
We shall  outline our computation on the dominating terms.   The discussions of  the negligible terms are similar to those in Lemma \ref{lem_24h1h2h3} and \ref{lem_bestimate2}, and are therefore omitted.  

Recall (\ref{eq_secondgm}). Using the same proof of (\ref{eq_h200}),  we first get
\begin{align*}
\mathfrak{h}(2,0,0)&=\frac{1}{n} \sum_{i,j} \frac{\partial^2 \mathbbm{h}_1}{\partial x_{ij}^2} h_2 h_3 =\frac{1}{n} \sum_{i,j}  \Big(\frac{\partial^2 G}{\partial x_{ij}^2} A_2\Big)_{ij'} \mathfrak{q}^{(k-1)} \\
& =\frac{2z }{n} \sum_{i,j} G_{ii} G_{j'j'} (GA_2)_{ij'} \mathfrak{q}^{(k-1)}+O_{\prec}(n^{-\frac12}) \\
&=\frac{2z}{n} \sum_{i,j} (\Pi_1)_{ii} (\Pi_1)_{j'j'} (\Pi_1 A_2)_{ij'} \mathfrak{q}^{(k-1)}+O_{\prec}(n^{-\frac12}) \\
&=\mathfrak{d}_2^a \mathfrak{q}^{(k-1)}+O_{\prec}(n^{-\frac12}).
\end{align*}
Likewise,  applying (\ref{eq_partig2w}), we find that
\begin{align*}
\tilde{\mathfrak{h}}(2,0,0)&=\frac{1}{n}  \sum_{i,j} \frac{\partial^2 \tilde{\mathbbm{h}}_1 }{\partial x_{ij}^2} h_2 h_3  =\frac{1}{n} \sum_{i,j} \Big( \frac{\partial^2 G^2}{\partial x_{ij}^2} B_2 \Big)_{ij'} \mathfrak{q}^{(k-1)}=\tilde{\mathfrak{d}}_2 \mathfrak{q}^{(k-1)}+O_{\prec}(n^{-\frac12}).
\end{align*}

Next, recall (\ref{eq_firstgm}) and (\ref{eq_parti_1}). By a  discussion similar to that of (\ref{eq_h110v2}),  we conclude that 
\begin{align*}
\mathfrak{h}(1,1,0)&=\frac{1}{n} \sum_{i,j} \frac{\partial \mathbbm{h}_1}{\partial x_{ij}} \frac{\partial h_2}{\partial x_{ij}} h_3\\
&=-\frac{(k-1) z }{\sqrt{n}} \sum_{i,j} (\Pi_1)_{ii} (\Pi_1 A_2)_{j'j'} c_{ij} \mathfrak{q}^{(k-2)}+O_{\prec}(n^{-\frac12+4\nu})\\
&=\mathfrak{a}_{22} \mathfrak{q}^{(k-2)}+O_{\prec}(n^{-\frac12+4\nu}).
\end{align*}
Similarly, recalling (\ref{eq_g2w1st}), we have 
\begin{align*}
\tilde{\mathfrak{h}}(1,1,0)=\frac{1}{n} \sum_{i,j} \frac{\partial \tilde{\mathbbm{h}}_1}{\partial x_{ij}} \frac{\partial h_2}{\partial x_{ij}} h_3=\tilde{\mathfrak{b}}_{22} \mathfrak{q}^{(k-2)}+O_{\prec}(n^{-\frac12}).
\end{align*}
Finally, recall (\ref{eq_fourgm}) and (\ref{eq_parti_2}). The same arguments as (\ref{eq_thirdnonvaish}) yield
\begingroup
\allowdisplaybreaks
\begin{align*}
&\mathfrak{h}(1,2,0)=n^{-\frac32} \sum_{i,j} \frac{\partial \mathbbm{h}_1}{\partial x_{ij}} \frac{\partial^2 h_2}{\partial x^2_{ij}} h_3=\mathfrak{a}_{32}\mathfrak{q}^{(k-2)}+O_{\prec}(n^{-\frac12}), \\
&\tilde{\mathfrak{h}}(1,2,0)=n^{-\frac32}  \sum_{i,j} \frac{\partial \tilde{\mathbbm{h}}_1}{\partial x_{ij}} \frac{\partial^2 h_2}{\partial x_{ij}^2} h_3=\tilde{\mathfrak{b}}_{32} \mathfrak{q}^{(k-2)}+O_{\prec}(n^{-\frac12}).
\end{align*}
\endgroup
This concludes our proof. 
\end{proof}

\section{Proof of Theorems \ref{thm. right subspace}} \label{s. 18092610}
 In this section, we prove Theorem \ref{thm. right subspace}.  The proof follows along the same lines of the proof of Theorem \ref{thm_mainthm}, and is summarized as follows. First, by Lemma \ref{lem:key}, we reduce the problem to study the quantity $\mathcal Q$ defined below. After necessary notations are introduced, as done in the beginning of Section \ref{sec_mainthm}, it suffices to prove Proposition \ref{lem:subspace}, which is an analogue of Proposition \ref{lem:main}. The proof of Proposition \ref{lem:subspace} essentially relies on a recursive  estimate presented in Proposition \ref{prop_landr}. Thus the main goal of this section is to prove Proposition \ref{prop_landr}. 
% We shall apply the same arguments as the proof of its counterpart, Proposition \ref{prop_iteration}, and emphasize the differences. 

Let $\mathbf{z}=(z_1,\cdots,z_r)$ denote a vector with all the entries $z_\beta \in \mathbf{S}_o$. Following the discussion in the beginning of Section \ref{sec_mainthm}, with a slight abuse of notation, we introduce a few definitions. Let 
\begin{equation*}
\mathcal{Q} \equiv \mathcal{Q}\mathcal(\mathbf{z}):=\sqrt{n} \sum_{\beta=1}^r \Big( \text{Tr}(\Xi_1(z_\beta) )A_\beta^R +\text{Tr}(\Xi_1^{\prime}(z_\beta))B_\beta^R \Big).
\end{equation*}
Denote the index set as 
\begin{equation*}
\mathcal{B}(\nu):=\bigcup_{\beta=1}^r\mathcal{B}_\beta(\nu),
\end{equation*}
where $\mathcal{B}_\beta(\nu)$ is defined as 
\begin{equation*}
\mathcal{B}_\beta(\nu):=\Big\{ (i,j)\in [M]\times[n] : |\mathbf{u}_\beta(i)|>n^{-\nu}, \ |\mathbf{v}_\beta(j)|>n^{-\nu} \Big\}.
\end{equation*}
Since $r$ is fixed and all the vectors $\mathbf{u}_\beta$ and $\mathbf{v}_\beta$ for  $\beta \in [r]$ are unit vectors,  it is easy to conclude that $|\mathcal{B}(\nu)| \leq C n^{4 \nu}$ for some constant $C>0$. 

For $\beta \in [r]$,  invoke $\Delta_d(z_\beta)$ by plugging $z_\beta$ in (\ref{defn_detministic1}). We also introduce the random variable 
\begin{equation*}
\Delta_r(z_\beta):=\sqrt{n z_\beta} \sum_{(i,j) \in \mathcal{B}(\nu)} x_{ij} (c_\beta)_{ij},
\end{equation*}
where $(c_\beta)_{ij}\equiv (c_\beta (z_\beta))_{ij}$ is defined by inserting $z_{\beta}$ into $c_{ij}$ in (\ref{defn_cij1}).
Similarly, we denote  $(s_\beta)_{ij} \equiv (s_\beta(z_\beta))_{ij}$ by plugging $z_\beta$ into $s_{ij}$ in (\ref{defn_sij1}).
Let $C_\beta$ and $S_\beta$ be $M\times n$ matrices with entries $(c_\beta)_{ij}$ and $(s_\beta)_{ij}$ respectively.   
 Denote 
\begin{equation*}
\Delta_r \equiv \Delta_r(\mathbf{z}):=\sum_{\beta=1}^r \Delta_r(z_\beta), \quad \Delta_d \equiv \Delta_d(\mathbf{z}):=\sum_{\beta=1}^r \Delta_d(z_\beta),
\end{equation*}
and $\Delta=\Delta_d+\Delta_r.$ Furthermore, we denote 
\begin{equation} \label{delta_new}
Q=\mathcal{Q}-\Delta.
\end{equation}
Note $\mathcal{V}^E(z_\beta)$ is defined in (\ref{eq_indivialve}) by plugging $z_{\beta}$. Set
\begin{equation*}
\mathcal{V}^E(\mathbf{z})=\sum_{\beta=1}^r  \mathcal{V}^E(z_\beta).
\end{equation*}
Then we  define the  function 
\begin{align*}
V &\equiv V(\mathbf{z}) \\
&=\mathcal{V}^E(\mathbf{z})+2\frac{\kappa_3}{\sqrt{n}} \text{Tr} \Big( ( \sum_{\beta=1}^rz_{\beta} S_\beta) (\sum_{\beta=1}^r  \sqrt{z_{\beta}} C_\beta)^* \Big)+\frac{\kappa_4}{n}\text{Tr} \Big(( \sum_{\beta=1}^rz_{\beta} S_\beta) ( \sum_{\beta=1}^rz_{\beta} S_\beta)^* \Big)\\
&\quad+\sum_{(i,j)\in \mathcal{S}(\nu)} \Big(\sum_{\beta=1}^r \sqrt{z_{\beta}} (c_{\beta})_{ij} \Big)^2.
\end{align*}
Recall $p_\beta=p(d_\beta)$ in \eqref{18012901}. Let $\mathbf{z}_0:=(p_1,\ldots,p_r).$
\begin{pro}\label{lem:subspace}
Under  the assumptions of Theorem \ref{thm. right subspace}, we have that $Q(\mathbf{z}_0)$ and $\Delta(\mathbf{z}_0)$ are asymptotically independent. Furthermore, 
\begin{equation*}
Q(\mathbf{z}_0) \simeq \mathcal{N}\left(0, V(\mathbf{z}_0) \right).
\end{equation*}
\end{pro}
Theorem \ref{thm. right subspace} follows from Proposition \ref{lem:subspace}. The arguments are the same as the proof of Theorem \ref{thm_mainthm} in Section \ref{sec_mainthm}. Again, the final presentation of the results in Theorem \ref{thm. right subspace} are obtained by plugging the values $p_\beta$ for $1\le \beta\le r$ using the continuity of Green functions and performing tedious calculations. We omit the details. 

Similar to the discussion of Proposition \ref{lem:main}, to prove Proposition \ref{lem:subspace}, it suffices to establish the following recursive estimates. It is an analogue of Proposition \ref{prop_iteration}. 
\begin{pro} \label{prop_landr} Suppose the assumptions of Theorem \ref{thm. right subspace} hold. Let $z_\beta=p_\beta+\ii n^{-C}$ and $z_{0\beta}=p_\beta$ for all $\beta \in [r]$. We have  
\begin{align*}
\mathbb{E} Q(z_\beta) e^{\ii t \Delta(z_{0\beta})} =  O_{\prec}(n^{-1/2+\nu}),
\end{align*}
and for any fixed integer $k \geq 2,$
\begin{equation*}
\mathbb{E} Q^k(z_\beta) e^{\ii t \Delta(z_{0\beta})}=(k-1) V \mathbb{E}Q^{k-2}(z_\beta) e^{\ii t \Delta(z_{0\beta})}+ O_{\prec}(n^{-1/2+\nu}) . 
\end{equation*}
\end{pro}

\section{Proof of Proposition \ref{prop_landr}}

Proposition \ref{prop_landr} can be proved in a  way similar to Proposition \ref{prop_iteration}. Recall from Section \ref{s. proof of recursive estimate} that the proof of Proposition \ref{prop_iteration} is based on Lemma \ref{lem.rewrite of Q} and Lemma \ref{lem. recursive estimate for each term}. We present the analogues of these two lemmas and their proofs in the following two steps. We shall only outline the key estimates and focus on discussing the differences.   

\vspace{1ex}

\noindent{\bf Step 1.} In the first step, we will rewrite $Q$ in (\ref{delta_new}).  Recall (\ref{eq_diagonalmatrix}) and for each $\beta\in [r]$, denote 
\begin{align*}
&A_{\beta,1}:=A_\beta^R \mathtt{I}^{\mathrm{u}}, \quad A_{\beta,1}:=A_\beta^R \mathtt{I}^{\mathrm{l}}, \\
&B_{\beta,1}:=B_\beta^R \mathtt{I}^{\mathrm{u}}, \quad B_{\beta,1}:=B_\beta^R \mathtt{I}^{\mathrm{l}}.
\end{align*}
 Furthermore, for $\alpha=1,2,$ we define  
\begin{align*}
f_{\beta,\alpha}:= -m_{\alpha }(z_\beta) \Tr [H(z_\beta)\Xi_1(z_\beta) A_{\beta,\alpha} ]+F_{\beta,\alpha} \Tr [ G(z_\beta) A_{\beta,\alpha} ],
\end{align*}
and 
\begin{align*}
g_{\beta,\alpha}&:=-\frac{1}{2} m_{\alpha}(z_\beta) \Tr [ H(z_\beta) \Xi_2(z_\beta) B_{\beta,\alpha}] + \frac{F_{\beta,\alpha}}{2} \Tr [ G^2(z_\beta) B_{\beta,\alpha}] \\ 
&\quad+\frac{1}{2} (m_{\alpha}(z_\beta) - \frac{1}{z_\beta}) \Tr[ G(z_\beta) B_{\beta,\alpha} ]
 - m'_{\alpha}(z_\beta) \Tr [  B_{\beta,\alpha}(z_\beta) ] \\
& \quad +m_{\alpha}' (z_\beta)\Tr [ H (z_\beta)\Pi_1(z_\beta) B_{\beta,\alpha}(z_\beta) ],
\end{align*}
where $F_{\beta,\alpha}$ is defined in \eqref{defn_f1f2} with $z=z_\beta$.
Finally, for $\beta \in [r],$ we denote 
\begin{equation}\label{defn_qbeta}
Q_{\beta}:=\sqrt{n} \sum_{\alpha=1,2}(f_{\beta, \alpha}+g_{\beta, \alpha})+\sqrt{n z_\beta} \sum_{(i,j) \in \mathcal{S}(\nu)} x_{ij} (c_\beta)_{ij}-\Delta_d(z_\beta).
\end{equation}

\begin{lem}\label{lem_rewriteqr} Under the assumptions of Proposition \ref{lem:subspace}, we have 
\begin{align} \label{eq_newq}
Q= \sum_{\beta=1}^r Q_\beta.
\end{align}
\end{lem}
Indeed, Lemma \ref{lem_rewriteqr} is the analogue of Lemma \ref{lem.rewrite of Q}, and its proof is also a straightforward extension of the rank one case. We omit the details here.

As a consequence,  to prove Proposition \ref{prop_landr}, it suffices to study the following  
\begin{align} \label{eq_newexpanr}
\mathbb{E} Q^k e^{it \Delta}
& = \sqrt{n} \sum_{\beta=1}^r \sum_{\alpha=1}^2  \mathbb{E} \Big( f_{\beta, \alpha}+g_{\beta, \alpha} \Big) Q^{k-1} e^{it \Delta} \nonumber \\
&\quad+ \sqrt{n} \sum_{(i,j)  \in \mathcal{S}(\nu)} (\sum_{\beta=1}^r \sqrt{z_\beta} (c_\beta)_{ij}) \mathbb{E} x_{ij}  Q^{k-1} e^{it \Delta}- \Delta_d \mathbb{E} Q^{k-1} e^{it \Delta}.
\end{align}

\noindent{\bf Step 2. } In the second step, we will use the cumulant expansion to estimate the items on the right hand side of (\ref{eq_newexpanr}) and prove the analogue of Lemma \ref{lem. recursive estimate for each term}.  

Observe that for the rank $r$ case, we have 
\begin{align}\label{eq_partiq}
\frac{\partial Q}{\partial x_{ij}}=\sum_{\beta=1}^r \frac{\partial Q_\beta}{\partial x_{ij}}.
\end{align} 
The estimates of the cumulant expansion for the terms in (\ref{eq_newexpanr}) follow along the exact lines of Lemma \ref{lem_estimateh1h2h3}-\ref{lem_a2b2estimatehigh}, together with linearity of expectation. The main difference is that we will have cross terms from $A_{\beta_1}^R A_{\beta_2}^R$, $B_{\beta_1}^R B_{\beta_2}^R$ and $A_{\beta_1}^R B_{\beta_2}^R$ for $\beta_1, \beta_2 \in [r]$. However, by the orthogonality of the singular vectors, it is easy to check  (via the definitions of $A_\beta^R$ and $B_\beta^R$ in \eqref{eq:new-ab}) that $$A_{\beta_1}^R A_{\beta_2}^R=B_{\beta_1}^R B_{\beta_2}^R=A_{\beta_1}^R B_{\beta_2}^R=0$$ if $\beta_1 \neq \beta_2$. Consequently, these cross terms essentially make no contribution. We specify one example here. In the proof of the analogue of (\ref{eq_h1partih2h3}), we shall encounter an term of the following form
\begin{equation*}
\frac{1}{\sqrt{n}} \sum_{i,j} \Big(\Xi_1(z_\beta) A_{\beta,1} \Big)_{j'i} \frac{\partial Q}{\partial x_{ij}} Q^{k-2} e^{\ii t \Delta}=\frac{1}{\sqrt{n}}\sum_{\gamma=1}^r \sum_{i,j} \Big(\Xi_1(z_\beta) A_{\beta,1} \Big)_{j'i} \frac{\partial Q_{\gamma}}{\partial x_{ij}} Q^{k-2} e^{\ii t \Delta}.
\end{equation*}
Applying (\ref{eq_parti_1}) for each $\partial Q_\gamma/\partial x_{ij}$, by  (\ref{eq_localoutside}) and orthogonality of the singular vectors, we find the only contributing part is  
\begin{align*}
\frac{1}{\sqrt{n}} \sum_{i,j} \Big(\Xi_1(z_\beta) A_{\beta,1} \Big)_{j'i} \frac{\partial Q_{\beta}}{\partial x_{ij}} Q^{k-2} e^{\ii t \Delta}
\end{align*}
and what remains is exactly the same as the proof of (\ref{eq_h1partih2h3}). This explains why most quantities appearing in Theorem \ref{thm. right subspace} and its proof are similar to, and most of time are simply the sum of those in the proof of Theorem \ref{thm_mainthm}. In the following discussion, we shall concentrate on these cross terms from different singular values and vectors, and show they are actually negligible due to the orthogonality of singular vectors.

We first introduce some notations. Recall (\ref{18091910}). For $\beta \in [r]$, we denote  $\mathfrak{d}_{\beta,\alpha}^{a}, \mathfrak{d}_{\beta,\alpha}^{b}, \tilde{\mathfrak{d}}_{\beta,\alpha}$ by replacing $z$ with $z_\beta$ and $A_{\alpha}, B_{\alpha}$ with $A_{\beta, \alpha}, B_{\beta, \alpha}$ ($\alpha=1,2$) correspondingly. We also define $\mathfrak{a}_{\beta,1\alpha}, \mathfrak{b}_{\beta,1\alpha},\tilde{\mathfrak{b}}_{\beta,1\alpha}$ for $\alpha =1 ,2$ in the same fashion using (\ref{def:missvar}). Next, we denote  
\begin{align*}
&\mathfrak{a}_{\beta,21}:= -\frac{ (k-1)z_\beta}{\sqrt{n}} \sum_{(i,j)\in \mathcal S(\nu)} \big(\Pi_1(z_\beta) \big)_{j' j'} \big( \Pi_1(z_\beta) A_{\beta,1} \big)_{ii} (\sum_{\gamma=1}^r \sqrt{z_{\gamma}} C_{\gamma})_{ij},\nonumber\\
&\tilde{\mathfrak{b}}_{\beta,21}:= -\frac{ (k-1)z_\beta}{\sqrt{n}} \sum_{(i,j)\in \mathcal S(\nu)} \Big[ \big(\Pi_1 (z_\beta) \big)_{j' j'} \big(\Pi_2(z_\beta) B_{\beta,1} \big)_{ii} + \big(\Pi_2(z_\beta) \big)_{j' j'} \big(\Pi_1(z_\beta) B_{\beta,1} \big)_{ii}  \Big]  \\
&\hspace{4.5cm}\times \big(\sum_{\gamma=1}^r \sqrt{z_{\gamma}} C_{\gamma} \big)_{ij},
\end{align*}
and define $\mathfrak{a}_{\beta,22}, \tilde{\mathfrak{b}}_{\beta,22}$ analogously.
Further,  we denote 
\begin{align*}
&\mathfrak{a}_{\beta,31}:=-\frac{2(k-1)z_\beta^{3/2}}{n} \sum_{i,j}  \big(\Pi_1(z_\beta) \big)_{j'j'} \big(\Pi_1(z_\beta) A_{\beta,1} \big)_{ii} (\sum_{\gamma=1}^r z_\gamma S_\gamma)_{ij}, \nonumber\\  
&\tilde{\mathfrak{b}}_{\beta,31}=-\frac{2(k-1)z_\beta^{3/2}}{n} \sum_{i,j} \Big[ \big(\Pi_1(z_\beta) \big)_{j'j'} \big(\Pi_2(z_\beta) B_{\beta,1} \big)_{ii} +  \big(\Pi_2(z_\beta) \big)_{j'j'} \big(\Pi_1(z_\beta) B_{\beta,1} \big)_{ii} \Big]\\
&\hspace{4cm}\times \big(\sum_{\gamma=1}^r z_\gamma S_\gamma \big)_{ij}, 
\end{align*}
and define $\mathfrak{a}_{\beta,32}, \tilde{\mathfrak{b}}_{\beta,32}$ analogously.
Finally, we denote 
\begin{align*}
\mathfrak{a}_{\beta,0\alpha} &:=\mathfrak{a}_{\beta,1\alpha}+\kappa_3 \mathfrak{a}_{\beta,2\alpha}+\frac{\kappa_4}{2} \mathfrak{a}_{\beta,3\alpha},\nonumber\\
\mathfrak{b}_{\beta,0\alpha} &:=\frac{m_\alpha(z_\beta)}{2} \tilde{\mathfrak{b}}_{\beta,1\alpha}+m_{\alpha}'(z_\beta) \mathfrak{b}_{\beta,1\alpha}+\frac{\kappa_3 m_\alpha(z_\beta)  }{2} \tilde{\mathfrak{b}}_{\beta,2\alpha}\\
&\quad+\kappa_3  m_\alpha'(z_\beta)  \mathfrak{b}_{\beta,2\alpha}+\frac{ \kappa_4m_\alpha(z_\beta) }{4} \tilde{\mathfrak{b}}_{\beta,3\alpha}+\frac{ \kappa_4 m_\alpha'(z_\beta) }{2} \mathfrak{b}_{\beta,3\alpha}. 
\end{align*}
 We adopt the notation
\begin{equation*}
\mathfrak{q}^{(l)}=Q^l e^{\ii t \Delta}.
\end{equation*}
With these preparations, we present the following analogue of Lemma \ref{lem. recursive estimate for each term}.
\begin{lem} \label{lem_propkey} Under the assumptions of Proposition \ref{prop_landr}, for each $\beta\in [r]$ and $\alpha=1,2,$  we have
\begin{align}
&\sqrt n \E  f_{\beta, \alpha}  \mathfrak{q}^{(k-1)}=-\sqrt{z_\beta} m_{\alpha} (z_\beta)\mathbb{E} \big( \frac{\kappa_3}{2} \mathfrak{d}_{\beta,\alpha}^a \mathfrak{q}^{(k-1)} +\mathfrak{a}_{\beta,0\alpha} \mathfrak{q}^{(k-2)} \big)  +O_{\prec}(n^{-\frac12+4\nu}), \label{eq_rankrpart1} \\
&\sqrt{n}\mathbb{E} g_{\beta,\alpha} \mathfrak{q}^{(k-1)}= - \sqrt{z_\beta}\, \mathbb{E} \Big(\frac{ \kappa_3 }{4} \big(m_{\alpha}(z_\beta)\tilde{\mathfrak{d}}_{\beta,\alpha}+2m_{\alpha}'(z_\beta) \mathfrak{d}^b_{\beta,\alpha} \big)\mathfrak{q}^{(k-1)} +\mathfrak{d}_{\beta,0\alpha} \mathfrak{q}^{(k-2)}  \Big)  \nonumber\\
&\hspace{3cm}+O_{\prec}(n^{-\frac12+4\nu}). \nonumber 
\end{align}
In addition, we have 
\begin{align*}
&\sqrt{n } \sum_{(i,j)\in \mathcal{S}(\nu)}  ( \sum_{\beta=1}^r \sqrt{z_{\beta}} (c_{\beta})_{ij}  \E  x_{ij} \mathfrak{q}^{(k-1)}=(k-1) \Big[ \sum_{(i,j)\in \mathcal{S}(\nu)} \big(\sum_{\beta=1}^r \sqrt{z_{\beta}} (c_{\beta})_{ij} \big)^2 \\
&\qquad\qquad\quad  +  \frac{ \kappa_3 }{\sqrt{n}} \sum_{(i,j)\in \mathcal{S}(\nu)} \big( \sum_{\beta=1}^r(z_{\beta} s_\beta)_{ij} \big) \big(\sum_{\beta=1}^r  \sqrt{z_{\beta}} (c_\beta)_{ij}\big) \Big] \E \mathfrak{q}^{(k-2)}+O_{\prec}(n^{-\frac12+4\nu}). 
\end{align*}
\end{lem}

Similar to the proof of Proposition \ref{prop_iteration}, Proposition \ref{prop_landr} follows immediately from Lemma \ref{lem_rewriteqr} and \ref{lem_propkey}. We omit the details here. 

Next, we turn to the proof of Lemma \ref{lem_propkey}. 
We will only  focus our discussion on the term $\sqrt{n} \mathbb{E} f_{\beta,1} \mathfrak{q}^{(k-1)}$ and the other terms can be estimated likewise. Using a discussion similar to (\ref{eq_defnfa1}), for each fixed $\beta\in [r],$ we have  
\begin{equation*}
\sqrt{n} \mathbb{E} f_{\beta,1} \mathfrak{q}^{(k-1)}=\mathbb{E} \Big(-m_1 \sqrt{nz_{\beta}} \sum_{i,j} x_{ij} \big(\Xi_1(z_\beta) A_{\beta,1} \big)_{j'i} +\sqrt{n} F_{1} \text{Tr} \big( G(z_\beta) A_{\beta,1} \big) \Big) \mathfrak{q}^{(k-1)}.
\end{equation*}
As seen in the proof of (\ref{eq_parta1}), we need the following estimates which are analogues of those in Lemma \ref{lem_estimateh1h2h3} and \ref{lem_24h1h2h3}. We adopt the notations in (\ref{18091411}) by denoting $$h_1=\big(\Xi_1(z_\beta) A_{\beta,1} \big)_{j'i},\quad h_2=Q^{k-1}, \quad h_3=e^{\ii t \Delta}.$$
\begin{lem} \label{lem_rankrestimate} For the derivatives of $h_1 h_2 h_3,$ we have 
\begin{align}
&\hslash(1,0,0)=-\sqrt{n z_{\beta}} m_2(z_\beta) \operatorname{Tr}\big( G(z_\beta)A_{\beta,1} \big) \mathfrak{q}^{(k-1)}+O_{\prec}(n^{-1/2}), \label{eq_rpath1h2h3}\\
& \hslash(0,1,0)=\mathfrak{a}_{\beta,11} \mathfrak{q}^{(k-2)}+O_{\prec}(n^{-1/2}), \label{eq_rh1parth2h3} \\
& \hslash(2,0,0)=\mathfrak{d}^a_{\beta,1} \mathfrak{q}^{(k-1)}+O_{\prec}(n^{-1/2}),  \nonumber \\
& \hslash(1,1,0)=\mathfrak{a}_{\beta,21} \mathfrak{q}^{(k-2)}+O_{\prec}(n^{-1/2}),  \nonumber \\
& \hslash(1,2,0)=\mathfrak{a}_{\beta,31} \mathfrak{q}^{(k-2)}+O_{\prec}(n^{-1/2}).  \nonumber
\end{align}
Furthermore, all the other terms $\hslash(l_1,l_2,l_3)$ for $l_1+l_2+l_3\le 4$ can be bounded by $O_{\prec}(n^{-1/4+4\nu}).$
\end{lem}

It is easy to see that (\ref{eq_rankrpart1}) follows directly from Lemma \ref{lem_rankrestimate}. Thus the final task is to prove Lemma \ref{lem_rankrestimate}.  In the proof,  we will  use the orthogonality of the singular vectors, that is, for $\beta_1 \neq \beta_2,$
\begin{equation}\label{eq_lineaindependece}
\langle \mathbf{u}_{\beta_1}, \mathbf{u}_{\beta_2} \rangle=0, \quad  \langle \mathbf{v}_{\beta_1}, \mathbf{v}_{\beta_2} \rangle=0.
\end{equation} 

\begin{proof}[Proof of Lemma \ref{lem_rankrestimate}] First of all, (\ref{eq_rpath1h2h3}) can be estimated similarly as (\ref{eq_parh1h2h3}). The other four dominating terms can be analyzed analogously and we shall only focus on the estimate of (\ref{eq_rh1parth2h3}). Observe that
\begin{equation} \label{eq_h1partih2h3k}
\hslash(0,1,0)=\frac{1}{\sqrt{n}} \sum_{i,j} h_1 \frac{\partial h_2}{\partial x_{ij}} h_3=\frac{(k-1)}{\sqrt{n}} \sum_{i,j} \big(\Xi_1(z_\beta) A_{\beta,1} \big)_{j'i} \frac{\partial Q}{\partial x_{ij}} \mathfrak{q}^{k-2}.
\end{equation}
Plugging in (\ref{eq_partiq}), we have 
\begin{align*}
\hslash(0,1,0)&=\frac{(k-1)}{\sqrt{n}} \sum_{\gamma=1}^r \sum_{i,j} \big(\Xi_1(z_\beta) A_{\beta,1} \big)_{j'i} \frac{\partial Q_\gamma}{\partial x_{ij}} Q^{k-2} e^{\ii t \Delta},
\end{align*}
where by (\ref{eq_parti_1}),
\begin{align}\label{eq_parti_k1}
\frac{\partial Q_{\gamma}}{\partial x_{ij}}&=-\sqrt{nz_\gamma} \sum_{\substack{ l_1,l_2\in \{ i, j'\} \\ l_1\neq l_2 }}\Big[ \big(G(z_\beta)A_\gamma^R G(z_\beta) \big)_{l_1 l_2} -\frac{1}{2z_\gamma} \big(G(z_\beta) B_\gamma^R G(z_\beta) \big)_{l_1 l_2} \nonumber \\
&\qquad + \frac{1}{2}  \sum_{(a_1,a_2) \in \mathcal{P}(2,1)} \big(G^{a_1}(z_\beta) B_\gamma^R G^{a_2}(z_\beta) \big)_{l_1 l_2} \Big]-\mathbf{1}\Big( (i,j) \in \mathcal{B}(\nu) \Big)\sqrt{nz_\gamma} (C_\gamma)_{ij}.
\end{align}
Using a discussion similar to (\ref{eq_h1partih2h3}), we have that 
\begin{align*}
\frac{(k-1)}{\sqrt{n}} \sum_{i,j} \big(\Xi_1(z_\beta) A_{\beta,1} \big)_{j'i} \frac{\partial Q_\beta}{\partial x_{ij}} \mathfrak{q}^{(k-2)}=\mathfrak{a}_{\beta,11} \mathfrak{q}^{(k-2)}+O_{\prec}(n^{-1/2}).
\end{align*}
Therefore, it suffices to show that for $\gamma \neq \beta,$ 
\begin{equation}\label{eq_finalor}
\frac{1}{\sqrt{n}} \sum_{i,j} \big(\Xi_1(z_\beta) A_{\beta,1} \big)_{j'i}  \frac{\partial Q_{\gamma}}{\partial x_{ij}} \mathfrak{q}^{(k-2)}=O_{\prec}(n^{-\frac12}).
\end{equation} 
To prove this, we shall argue in a similar way to (\ref{eq_h1partih2h3}) by expanding the product above using \eqref{eq_parti_k1}.  We start with $$\sum_{i,j} \big(\Xi_1 (z_\beta)A_{\beta,1} \big)_{j'i} \big(G(z_\beta)A_{\gamma}^R G(z_\beta) \big)_{j'i}.$$ Recall (\ref{eq_g-pia}) and (\ref{eq_gag}). By (\ref{eq_genrealcontrol}) and (\ref{eq_lineaindependece}), we have 
\begin{align*}
&\sum_{i,j} \big((\mathcal{G}_2(z_\beta)-m_2(z_\beta) \big) \mathcal{A}_{\beta,3})_{j'i} (\mathcal{G}_2 \mathcal{A}_{\gamma,3} \mathcal{G}_1)_{j'i} \\
& =\omega_{\beta,3} \omega_{\gamma,3}  \text{Tr} \Big( \big(\mathcal{G}_2(z_\beta)-m_2(z_\beta) \big) \mathbf{v}_{\beta} \mathbf{u}^*_{\beta} \mathcal{G}_1(z_\beta) \mathbf{u}_{\gamma} \mathbf{v}_{\gamma}^* \mathcal{G}_2(z_\beta) \Big) \\
& =\omega_{\beta,3} \omega_{\gamma,3} \Big( \mathbf{u}^*_{\beta} \mathcal{G}_1(z_\beta) \mathbf{u}_{\gamma} \Big) \Big(\mathbf{v}_{\gamma}^* \mathcal{G}_2 \big(\mathcal{G}_2(z_\beta)-m_2(z_\beta) \big) \mathbf{v}_{\beta} \Big) = O_{\prec}(n^{-\frac12}),
\end{align*}
where the coefficients $\omega_{\beta,3}$ are defined using the block decomposition of $A_{\beta}^R$ as in (\ref{eq_abmatrix}).  We can estimate the other terms in the expansion (in light of (\ref{eq_g-pia}) and (\ref{eq_gag})) using similar discussions. Hence, we  conclude that 
\begin{equation*}
\sum_{i,j} \big(\Xi_1 (z_\beta)A_{\beta,1} \big)_{j'i} \big(G(z_\beta)A_{\gamma}^RG(z_\beta) \big)_{j'i}=O_{\prec}(n^{-1/2}).
\end{equation*}
Likewise, we can show that each of the following terms
\begin{align*}
&\sum_{i,j}\big( \Xi_1(z_\beta) A_{\beta,1} \big)_{j'i} \big( G(z_\beta) A_{\gamma}^R G(z_\beta) \big)_{ij'},  \quad \sum_{i,j} \big( \Xi_1(z_\beta) A_{\beta,1} \big)_{j'i}\big( G(z_\beta)^2 B_{\gamma}^R G(z_\beta) \big)_{j'i}, \\
& \sum_{i,j} \big(\Xi_1(z_\beta) A_{\beta,1} \big)_{j'i} \big(G^2(z_\beta) B_{\gamma}^R G(z_\beta) \big)_{ij'}, \quad \sum_{i,j} \big(\Xi_1(z_\beta) A_{\beta,1} \big)_{j'i}\big( G(z_\beta) B_{\gamma}^R G^2(z_\beta) \big)_{j'i}, \\ &\sum_{i,j} \big(\Xi_1(z_\beta) A_{\beta,1} \big)_{j'i} \big(G(z_\beta)B_{\gamma}^R G^2(z_\beta) \big)_{ij'}
\end{align*}
 can be bounded by $O_{\prec}(n^{-1/2}).$ In view of (\ref{eq_parti_k1}), we conclude the proof of (\ref{eq_finalor}). This completes our proof.   
\end{proof}

\end{document}